\pdfoutput=1
\newif\ifpersonal
\RequirePackage[l2tabu,orthodox]{nag} %detect whether obsolete packages are used
\documentclass[a4paper]{amsart} %reqno places equation numbers on the right
\usepackage{amsmath,amsthm,amssymb,mathrsfs,mathtools,bm,eucal,tensor} % math related
\usepackage{microtype,fixltx2e,lmodern} % latex technical issues
\usepackage[utf8]{inputenc} % input encoding
\usepackage[T1]{fontenc} % font encoding
\usepackage{enumerate,comment,braket,xspace,csquotes} % utilities

\usepackage{tikz-cd} %commutative diagrams
\tikzcdset{cells={font=\everymath\expandafter{\the\everymath\displaystyle}}} %every cell is in displaymath style

\usepackage[centering,vscale=0.7,hscale=0.8]{geometry}
\usepackage[hidelinks,linktoc=all]{hyperref}
\usepackage[capitalize]{cleveref}

\numberwithin{equation}{section}
\theoremstyle{plain}
\newtheorem{thm-intro}[equation]{Theorem}
\newtheorem{thm}[equation]{Theorem}
\newtheorem{lem}[equation]{Lemma}
\newtheorem{prop}[equation]{Proposition}

\newtheorem{cor}[equation]{Corollary}

\makeatletter
\newtheorem*{rep@theorem}{\rep@title}
\newcommand{\newreptheorem}[2]{
\newenvironment{rep#1}[1]{
 \def\rep@title{#2 \ref{##1}}
 \begin{rep@theorem}}
 {\end{rep@theorem}}}
\makeatother
\newreptheorem{theorem}{Theorem}
\newreptheorem{cor}{Corollary}

\theoremstyle{definition}
\newtheorem{defin}[equation]{Definition}
\newtheorem{notation}[equation]{Notation}
\newtheorem{rem}[equation]{Remark}
\newtheorem{eg}[equation]{Example}
\newtheorem{construction}[equation]{Construction}

\ifpersonal
\newcommand{\personal}[1]{\textcolor[rgb]{0,0,1}{(Personal: #1)}}
\newcommand{\todo}[1]{\textcolor{red}{(Todo: #1)}}
\else
\newcommand*{\personal}[1]{\ignorespaces}
\newcommand*{\todo}[1]{\ignorespaces}
\fi

\newcommand{\C}{\mathbb C}

\newcommand{\rB}{\mathrm B}
\newcommand{\rH}{\mathrm H}

\newcommand{\rL}{\mathrm L}
\newcommand{\rR}{\mathrm R}
\newcommand{\rT}{\mathrm T}

\newcommand{\cC}{\mathcal C}
\newcommand{\cD}{\mathcal D}

\newcommand{\cF}{\mathcal F}
\newcommand{\cH}{\mathcal H}
\newcommand{\cG}{\mathcal G}

\newcommand{\cJ}{\mathcal J}
\newcommand{\cK}{\mathcal K}

\newcommand{\cM}{\mathcal M}

\newcommand{\cO}{\mathcal O}
\newcommand{\cP}{\mathcal P}

\newcommand{\cS}{\mathcal S}
\newcommand{\cT}{\mathcal T}
\newcommand{\cU}{\mathcal U}

\newcommand{\cX}{\mathcal X}

\DeclareFontFamily{U}{BOONDOX-calo}{\skewchar\font=45 }
\DeclareFontShape{U}{BOONDOX-calo}{m}{n}{<-> s*[1.05] BOONDOX-r-calo}{}
\DeclareFontShape{U}{BOONDOX-calo}{b}{n}{<-> s*[1.05] BOONDOX-b-calo}{}
\DeclareMathAlphabet{\mathcalboondox}{U}{BOONDOX-calo}{m}{n}
\newcommand{\bbP}{\mathbb P}

\newcommand{\bP}{\mathbf P}

\makeatletter
\let\save@mathaccent\mathaccent
\newcommand*\if@single[3]{%
	\setbox0\hbox{${\mathaccent"0362{#1}}^H$}%
	\setbox2\hbox{${\mathaccent"0362{\kern0pt#1}}^H$}%
	\ifdim\ht0=\ht2 #3\else #2\fi
}
\newcommand*\rel@kern[1]{\kern#1\dimexpr\macc@kerna}
\newcommand*\widebar[1]{\@ifnextchar^{{\wide@bar{#1}{0}}}{\wide@bar{#1}{1}}}
\newcommand*\wide@bar[2]{\if@single{#1}{\wide@bar@{#1}{#2}{1}}{\wide@bar@{#1}{#2}{2}}}
\newcommand*\wide@bar@[3]{%
	\begingroup
	\def\mathaccent##1##2{%
		\let\mathaccent\save@mathaccent
		\if#32 \let\macc@nucleus\first@char \fi
		\setbox\z@\hbox{$\macc@style{\macc@nucleus}_{}$}%
		\setbox\tw@\hbox{$\macc@style{\macc@nucleus}{}_{}$}%
		\dimen@\wd\tw@
		\advance\dimen@-\wd\z@
		\divide\dimen@ 3
		\@tempdima\wd\tw@
		\advance\@tempdima-\scriptspace
		\divide\@tempdima 10
		\advance\dimen@-\@tempdima
		\ifdim\dimen@>\z@ \dimen@0pt\fi
		\rel@kern{0.6}\kern-\dimen@
		\if#31
		\overline{\rel@kern{-0.6}\kern\dimen@\macc@nucleus\rel@kern{0.4}\kern\dimen@}%
		\advance\dimen@0.4\dimexpr\macc@kerna
		\let\final@kern#2%
		\ifdim\dimen@<\z@ \let\final@kern1\fi
		\if\final@kern1 \kern-\dimen@\fi
		\else
		\overline{\rel@kern{-0.6}\kern\dimen@#1}%
		\fi
	}%
	\macc@depth\@ne
	\let\math@bgroup\@empty \let\math@egroup\macc@set@skewchar
	\mathsurround\z@ \frozen@everymath{\mathgroup\macc@group\relax}%
	\macc@set@skewchar\relax
	\let\mathaccentV\macc@nested@a
	\if#31
	\macc@nested@a\relax111{#1}%
	\else
	\def\gobble@till@marker##1\endmarker{}%
	\futurelet\first@char\gobble@till@marker#1\endmarker
	\ifcat\noexpand\first@char A\else
	\def\first@char{}%
	\fi
	\macc@nested@a\relax111{\first@char}%
	\fi
	\endgroup
}
\makeatother

\newcommand{\PSh}{\mathrm{PSh}}
\newcommand{\Sh}{\mathrm{Sh}}

\newcommand{\Ab}{\mathrm{Ab}}

\newcommand{\tauet}{\tau_\mathrm{\acute{e}t}}

\newcommand{\bPsm}{\bP_\mathrm{sm}}

\newcommand{\Modh}{\textrm{-}\mathrm{Mod}^\heartsuit}
\newcommand{\Mod}{\textrm{-}\mathrm{Mod}}

\newcommand{\Coh}{\mathrm{Coh}}
\newcommand{\Cohb}{\mathrm{Coh}^\mathrm{b}}
\newcommand{\Cohh}{\mathrm{Coh}^\heartsuit}

\newcommand{\St}{\mathrm{St}}

\newcommand{\Aff}{\mathrm{Aff}}

\newcommand{\An}{\mathrm{An}}
\newcommand{\Afd}{\mathrm{Afd}}
\newcommand{\Top}{\mathcal T\mathrm{op}}

\newcommand{\dAn}{\mathrm{dAn}}
\newcommand{\dAnc}{\mathrm{dAn}_{\mathbb C}}
\newcommand{\dAnk}{\mathrm{dAn}_k}
\newcommand{\Ank}{\mathrm{An}_k}
\newcommand{\cTan}{\cT_{\mathrm{an}}}
\newcommand{\cTank}{\cT_{\mathrm{an}}(k)}

\newcommand{\cTdisck}{\cT_{\mathrm{disc}}(k)}
\newcommand{\cTet}{\cT_{\mathrm{\acute{e}t}}}
\newcommand{\cTetk}{\cT_{\mathrm{\acute{e}t}}(k)}
\newcommand{\Str}{\mathrm{Str}}
\newcommand{\Strloc}{\mathrm{Str}^\mathrm{loc}}
\newcommand{\RTop}{\tensor*[^\rR]{\Top}{}}

\newcommand{\dAfd}{\mathrm{dAfd}}
\newcommand{\dAfdk}{\mathrm{dAfd}_k}
\newcommand{\dStn}{\mathrm{dStn}}

\newcommand{\trunc}{\mathrm{t}_0}

\newcommand{\CAlg}{\mathrm{CAlg}}

\newcommand{\Cat}{\mathrm{Cat}}

\newcommand{\fib}{\mathrm{fib}}

\newcommand{\anL}{\mathbb L\an}

\newcommand{\cTanc}{\cTan(\mathbb C)}

\newcommand{\dAff}{\mathrm{dAff}}
\newcommand{\afp}{^{\mathrm{afp}}}
\newcommand{\bfMap}{\mathbf{Map}}
\newcommand{\cHom}{\cH \mathrm{om}}
\newcommand{\dAnSt}{\mathrm{dAnSt}}
\newcommand{\PrL}{\mathcal P \mathrm{r}^{\mathrm{L}}}

\newcommand{\Perf}{\mathrm{Perf}}
\newcommand{\QCoh}{\mathrm{QCoh}}
\newcommand{\dSt}{\mathrm{dSt}}

\newcommand{\AnPerf}{\mathrm{AnPerf}}
\newcommand{\bfAnPerf}{\mathbf{AnPerf}}
\newcommand{\bfPerf}{\mathbf{Perf}}

\newcommand{\bfAnMap}{\mathbf{AnMap}}

\newcommand{\Ind}{\mathrm{Ind}}

\newcommand{\an}{^\mathrm{an}}
\newcommand{\alg}{^\mathrm{alg}}

\newcommand{\et}{_\mathrm{\acute{e}t}}

\newcommand{\inv}{^{-1}}

\newcommand{\canal}{$\mathbb C$-analytic\xspace}

\newcommand{\kanal}{$k$-analytic\xspace}

\newcommand{\op}{^\mathrm{op}}
\newcommand{\Cech}{\check{\mathcal C}}
\newcommand{\DM}{Deligne-Mumford\xspace}

\newcommand*{\longhookrightarrow}{\ensuremath{\lhook\joinrel\relbar\joinrel\rightarrow}}

\usetikzlibrary{decorations.markings} %arrows for open immersions and closed immersions
\tikzset{
  closed/.style = {decoration = {markings, mark = at position 0.5 with { \node[transform shape, xscale = .8, yscale=.4] {/}; } }, postaction = {decorate} },
  open/.style = {decoration = {markings, mark = at position 0.5 with { \node[transform shape, scale = .7] {$\circ$}; } }, postaction = {decorate} }
}

\DeclareMathOperator{\Fun}{Fun}

\DeclareMathOperator{\Hom}{Hom}

\DeclareMathOperator{\Map}{Map}

\DeclareMathOperator{\Sp}{Sp}

\DeclareMathOperator{\Spec}{Spec}

\DeclareMathOperator{\Sym}{Sym}

\DeclareMathOperator*{\colim}{colim}
\DeclareMathOperator*{\fcolim}{``colim''}
\DeclareMathOperator*{\flim}{``lim''}

\DeclareMathOperator*{\cotimes}{\widehat{\otimes}}

\begin{document}
	
\title{Analytification of mapping stacks}

\author{Julian Holstein}
\address{Julian HOLSTEIN, Department of Pure Mathematics and Mathematical Statistics, University of Cambridge, Cambridge CB3 0WB, United Kingdom}
\email{JulianVSHolstein@gmail.com}

\author{Mauro PORTA}
\address{Mauro PORTA, Institut de Recherche Mathématique Avancée, 7 Rue René Descartes, 67000 Strasbourg, France}
\email{porta@math.unistra.fr}

\date{\today}
%\subjclass[2010]{Primary ; Secondary }
%\keywords{}

%\begin{abstract}
%\end{abstract}

\maketitle

\personal{PERSONAL COMMENTS ARE SHOWN!!!}

\begin{abstract}
	Derived mapping stacks are a fundamental source of examples of derived enhancements of classical moduli problems.
	For instance, they appear naturally in Gromov-Witten theory and in some branches of geometric representation theory.
	In this paper, we show that in many cases the mapping stacks construction commutes with the (complex or non-archimedean) analytification functor.
	Along the way, we establish several properties of the stack of analytic perfect complexes and study some incarnations of analytic Tannaka duality.
\end{abstract}

\tableofcontents

\section{Introduction}

One of the main uses of derived algebraic geometry is to provide well-behaved derived enhancements of classical moduli problems: while the original moduli problem is often highly singular, its derived counterpart has controlled singularities, which usually means that it is lci in the derived sense.
This phenomenon is extremely useful in constructions that involve virtual fundamental classes.
Examples can be found in Gromov-Witten theory and in geometric representation theory \cite{Mann_Robalo_Brane_actions,Porta_Yu_NQK,Schurg_Toen_Vezzosi_Determinant,Negut_Shuffle}.
It is often the case that these derived enhancements arise from mapping stack constructions, which are also one of the primary sources for interesting examples of derived schemes and stacks.

More recently, derived techniques have become available also in the (complex and non-archimedean) analytic setting.
The motivations that led to the development this theory come from mirror symmetry and from nonabelian Hodge theory.
We refer to the introductions of \cite{Porta_Yu_DNAnG_I,Porta_Derived_Riemann_Hilbert} for more details on these programs.
In particular, the mapping stack construction and its basic properties have been extended to the analytic setting in \cite{Porta_Yu_Mapping}.

In this paper we study the natural question of whether the mapping stack construction commutes with analytification.
Let us formulate a more precise statement.
We let $k$ denote either the field of complex numbers or a non-archimedean field with a non-trivial valuation.
We also let $\dAff_k^{\mathrm{afp}}$ denote the $\infty$-category of derived affine schemes almost of finite presentation\footnote{A derived affine scheme $\Spec(A)$ is almost of finite presentation over $k$ if each $\pi_i(A)$ is finitely generated as $\pi_0(A)$-module. This condition ensures that we can consider its analytification.} over $k$ and by $\dAn_k$ the $\infty$-category of derived analytic spaces over $k$ (see \cref{sec:review} for a review of these notions).
Given derived stacks $X, Y \colon (\dAff_k^{\mathrm{afp}})\op \to \cS$, we define $\bfMap(X,Y)$ as the derived stack
\[ \begin{tikzcd}
	\bfMap(X,Y) \colon (\dAff_k^{\mathrm{afp}})\op \longrightarrow \cS
\end{tikzcd} \]
sending $T$ to $\Map_{\dSt_k^{\mathrm{afp}}}(X \times T, Y)$.
Similarly, given derived analytic stacks $\mathscr X, \mathscr Y \colon \dAn_k\op \to \cS$, we define $\bfAnMap( \mathscr X , \mathscr Y )$ as the derived analytic stack
\[ \bfAnMap(\mathscr X , \mathscr Y) \colon \dAn_k\op \longrightarrow \cS \]
sending $U$ to $\Map_{\dAnSt_k}(\mathscr X \times U, \mathscr Y)$.
Given derived stacks $X, Y \colon (\dAff_k^{\mathrm{afp}})\op \to \cS$, there is a canonical map
\[ \bfMap(X,Y)\an \longrightarrow \bfAnMap(X\an, Y\an) , \]
where $(-)\an$ denotes the derived analytification functor (see \cref{subsec:derived_analytification} for a review of its construction and properties).
The goal of this paper is to provide sufficiently general conditions on $X$ and $Y$ to guarantee that the above map is an equivalence. 
The key property for $X$ is that it satisfies the \emph{GAGA property}, i.e.\ the analytification functor
		\[ \Perf(X) \longrightarrow \Perf(X\an) \]
is an equivalence. For $Y$ we demand that it is \emph{tannakian}, which means that for any derived stack $X$ the natural map
	\[ \Map_{\dSt_k}(X,Y) \longrightarrow \Fun^\otimes(\Perf(Y), \Perf(X))\]
	is fully faithful and its image has a precise characterization, see \cref{def:tannakian_stacks}.
Our main theorem is then the following: 

\begin{thm-intro}[{cf.\ \cref{thm:tannakian_target}}] \label{thm-intro:tannakian_target}
	Let $X, Y \colon (\dAff_k^{\mathrm{afp}})\op \to \cS$ be derived stacks locally almost of finite presentation.
	Suppose that:
	\begin{enumerate}
		\item the analytification functor
		\[ \QCoh(X) \longrightarrow \cO_{X\an} \Mod \]
		is $t$-exact and conservative;
		\item $X$ satisfies the GAGA property;
		\item $Y$ is a geometric stack which is tannakian and satisfies $\QCoh(Y) \simeq \Ind(\Perf(Y))$;
		\item the mapping stack $\bfMap(X,Y)$ is geometric.
	\end{enumerate}
	Then the canonical morphism
	\begin{equation} \label{eq:analytification_mapping_space}
		\Map_{\dSt_k^{\mathrm{afp}}}(X,Y) \longrightarrow \Map_{\dAnSt_k}(X\an,Y\an)
	\end{equation}
	is an equivalence.
	If furthermore $X$ satisfies the universal GAGA property (see \cref{defin:strong_GAGA} for the precise meaning), then the morphism
	\begin{equation} \label{eq:analytification_mapping_stacks}
		\bfMap(X,Y)\an \longrightarrow \bfAnMap(X\an, Y\an)
	\end{equation}
	is an equivalence of derived analytic stacks.
\end{thm-intro}

The assumptions on $Y$ are satisfied when it is a quasi-compact quasi-separated Deligne-Mumford stack or it is the classifying stack of an affine group scheme of finite type in characteristic $0$ (see the corollary to \cite[Theorem B]{Hall_Rydh_2017}).
On the other hand, it is usually harder to check that the GAGA assumptions on $X$ are satisfied.
We devote the entire \cref{sec:GAGA} to verifying these assumptions in a number of important examples.
Thanks to \cite{Porta_Yu_Higher_analytic_stacks_2014} we know that $X$ satisfies the GAGA property when $X$ is a proper geometric stack over $k$.
In this paper, we also verify that if $G$ is a reductive algebraic group over $\mathbb C$, then $\rB G$ satisfies the GAGA property, although it is not a proper geometric stack in the sense of \cite{Porta_Yu_Higher_analytic_stacks_2014}.
The main theorem we prove in this direction is the following:

\begin{thm-intro}[{cf.\ \cref{cor:analytification_relative_stack_perfect_complexes}}] \label{thm-intro:relative_GAGA}
	Let $k$ be either the field of complex numbers or a non-archimedean field equipped with a non-trivial valuation.
	Let $X$ be a proper derived geometric stack locally almost of finite presentation over $k$.
	Assume that the stack $\bfMap(X, \bfPerf_k)$ of perfect complexes on $X$ is locally geometric.
	Then its analytification is canonically equivalent to the analytic stack of perfect complexes on $X\an$, i.e.\ the canonical map
	\[ \bfMap(X, \bfPerf_k)\an \longrightarrow \bfAnMap(X\an, \bfAnPerf_k) \]
	is an equivalence.
	Here $\bfPerf_k$ (resp.\ $\bfAnPerf_k$) is the derived algebraic (resp.\ analytic) stack of perfect complexes (see \cref{sec:analytic_perfect_complexes}).
\end{thm-intro}

From this theorem we deduce that the same conclusion holds for the de Rham stack $X_{\mathrm{dR}}$, the Betti stack $X_{\mathrm B}$ and the Dolbeault stack $X_{\mathrm{Dol}}$ for some smooth scheme $X$.
Notice that in the assumptions of \cref{thm-intro:tannakian_target} we do not require $X$ to be a geometric stack, so we can apply our analytification theorem also to these cases.
Combining Theorems \ref{thm-intro:tannakian_target} and \ref{thm-intro:relative_GAGA}, we deduce the following vast generalization of the main result of \cite{Porta_Derived_Riemann_Hilbert}:

\begin{thm}[{cf.\ \cref{cor:generalized_RH_correspondence}}]
	Let $X$ be a smooth and proper scheme over $\mathbb C$.
	Let $Y$ be a derived stack locally almost of finite presentation satisfying the same assumptions as in \cref{thm-intro:tannakian_target}.
	Then there is a natural equivalence of derived analytic stacks
	\[ \bfMap(X_{\mathrm{dR}}, Y)\an \simeq \bfMap(X_{\mathrm{B}}, Y)\an , \]
	which reduces to the Deligne Riemann-Hilbert correspondence for rank $n$ vector bundles when $Y = \rB \mathrm{GL}_n$ obtained in \cite{Deligne_Equations_differentielles}.
\end{thm}

Let us spend a couple of extra words on \cref{thm-intro:tannakian_target}.
It is a threefold generalization of Lurie's main theorem in \cite{Lurie_Tannaka_duality}.
First of all, we allow our stacks $X$ and $Y$ to be derived.
Second, in loc.\ cit.\ it is only proven that \eqref{eq:analytification_mapping_space} is an equivalence, while we consider the analogous but stronger question for the map \eqref{eq:analytification_mapping_stacks}.
Finally, we remove the geometricity assumption on $X$.
These generalizations come at a cost: although the overall strategy for the proof of \cref{thm-intro:tannakian_target} follows loosely the one given in loc.\ cit.\ many auxiliary results have to be proven again and improved in our context.
Many of these results are of independent interest, and they constitute the other principal theorems of this paper.
Arguably the most important one is the following partial extension of Tannaka duality to the analytic setting:

\begin{thm-intro}[{cf.\ \cref{lem:refined_tannakian_image_necessary_conditions} and Propositions \ref{thm:refined_tannakian}, \ref{thm:refined_tannakian_image_kanal} and \ref{thm:refined_tannakian_image_canal}}] \label{thm-intro:analytic_tannaka_duality}
	Let $Y \in \dSt_k^{\mathrm{afp}}$ be a derived geometric stack locally almost of finite presentation and let $X \in \dAn_k$.
	Assume that:
	\begin{enumerate}
		\item $Y$ is tannakian;
		\item $\QCoh(Y) \simeq \mathrm{Ind}(\Perf(Y))$.
	\end{enumerate}
	Then the assignment sending a morphism $f \colon X \to Y\an$ to the composition
	\[ \Perf(Y) \longrightarrow \Perf(Y\an) \xrightarrow{f^*} \Perf(X) \]
	provides a fully faithful map
	\[ \Map_{\dAnSt_k}(X,Y\an) \longrightarrow \Fun^\otimes(\Perf(Y), \Perf(X)) . \]
	Furthermore, locally on $X$, we can identify the essential image of this functor with those symmetric monoidal functors
	\[ F \colon \QCoh(Y) \to \cO_X \Mod \]
	that commute with colimits, preserve perfect complexes and respect flat objects and connective objects.
\end{thm-intro}

The proofs of Theorems \ref{thm-intro:tannakian_target}, \ref{thm-intro:relative_GAGA} and \ref{thm-intro:analytic_tannaka_duality} heavily rely on two more technical results that we are going to describe next.
The first one concerns derived analytification for geometric stacks.
When $Y$ is a derived \DM stack locally almost of finite presentation over $k$, the definition of analytification implies that for every derived analytic space $X = (\cX_X, \cO_X)$ one has a canonical equivalence
\begin{equation} \label{eq:adjunction_analytification}
	\Map_{\dAn_k}(X, Y\an) \simeq \Map_{\RTop(\cTetk)}(X\alg, Y) .
\end{equation}
Here $X\alg = (\cX_X, \cO_X\alg)$ is the derived analytic space $X$ seen as a locally ringed space, and $\RTop(\cTetk)$ denotes the $\infty$-category of those locally ringed spaces whose structure sheaf has Henselian stalks.
When $Y$ is more generally a geometric stack, it can no longer be represented as an object in $\RTop(\cTetk)$ and $Y\an$ is no longer an object in $\dAn_k$.
Therefore, the above equivalence loses its meaning.
It is natural to replace $\dAn_k$ by $\dAnSt_k$, but the right hand side cannot be simply replaced by $\dSt_k$ because the construction $X \mapsto X\alg$ is not sufficiently well behaved.
We bypass this problem by proving the following:

\begin{thm-intro}[{cf.\ \cref{thm:generalized_adjunction}}] \label{thm-intro:generalized_adjunction}
	Let $X \in \dAn_k$ and let $\cX_X$ be its underlying $\infty$-topos.
	Let $\mathbf 1_X$ denote the final object of $\cX_X$.
	There exist functors
	\[ F^s_X \colon \dAnSt_k \longrightarrow \cX_X \quad , \quad G^s_X \colon \dSt_k^{\mathrm{afp}} \longrightarrow \cX_X \]
	and a natural transformation
	\[ \alpha \colon G^s_X \longrightarrow F^s_X \circ (-)\an \]
	satisfying the following conditions:
	\begin{enumerate}
		\item Given $Y \in \dSt_k^{\mathrm{afp}}$, there is a natural equivalence
		\[ \Map_{\dAnSt_k}(X, Y\an) \simeq F^s_X(Y\an)(\mathbf 1_X) . \]
		\item Given $Y \in \dSt_k^{\mathrm{afp}}$, the functor $G^s_X(Y)$ is the sheafification of the functor sending an \'etale morphism $U \to X$ to
		\[ \Map_{\dSt_k}( \Spec(\Gamma(U;\cO_U\alg)), Y ) . \]
		\item If $Y$ is a geometric stack, the natural transformation $\alpha$ is an equivalence.
	\end{enumerate}
\end{thm-intro}

Although more complicated than the adjunction which is available for derived \DM stacks, \cref{thm-intro:generalized_adjunction} is equally useful in practice, because it gives a way of describing morphisms \emph{into} $Y\an$ in terms of (sheaves of) maps into $Y$.
The entire \cref{sec:analytification_geometric_stacks} is devoted to formulating and proving this theorem.\\

The second main technical tool for this paper is a careful study of the stack of analytic perfect complexes, which we denote $\bfAnPerf_k$.
We analyze this stack in \cref{sec:analytic_perfect_complexes}.
In this section, the dichotomy between the non-archimedean and the complex analytic one is accentuated.
Indeed, in the non-archimedean case, the main results of this section are essentially straightforward corollaries of the results proven in \cite{Porta_Yu_Mapping}.
On the other hand, in the complex analytic case the proofs are significantly harder.
The main result we obtain is the following, which generalizes the more classical equivalence of \cite[Proposition 11.9.2]{Taylor_Several_complex}:

\begin{thm-intro}[{cf.\ \cref{thm:unbounded_coherent_pro_compact_Stein} and \cref{cor:perfect_complexes_compact_Stein}}] \label{thm-intro:perfect_complexes_compact_Stein}
	Let $X \in \dAn_\C$ be a derived complex analytic space and assume it is Stein.\footnote{This means that its truncation $\trunc(X)$ is a Stein space.}
	Let $K \subset \trunc(X)$ be a compact subset admitting a fundamental system of open Stein neighbourhoods.
	For every Stein neighbourhood $U$ of $K$ inside $X$, write $A_U \coloneqq \Gamma(U; \cO_U\alg)$.
	Then there is a canonical equivalence in $\Ind(\Cat_\infty^{\mathrm{st}, \otimes})$
	\[ \fcolim_{K \subset U \subset X} \Perf( A_U ) \simeq \fcolim_{K \subset U \subset X} \Perf( U ) , \]
	where the ind-objects are parametrized by all the Stein open neighbourhoods of $K$.
	Furthermore, after realizing these ind-objects, we obtain an equivalence in $\Cat_\infty^{\mathrm{st}, \otimes}$
	\[ \colim_{K \subset U \subset X} \Perf(U) \simeq \colim_{K \subset U \subset X} \Perf(A_U) \simeq \Perf(A_K) , \]
	where
	\[ A_K \coloneqq \colim_{K \subset U \subset X} A_U . \]
\end{thm-intro}

\begin{rem}
	The above theorem admits a relative version.
	Notably, if $X$ is a proper derived geometric $\C$-stack, then for any $X \in \dAn_\C$ the analytification induces an equivalence
	\begin{equation} \label{eq:ind_relative_GAGA}
		\fcolim_{K \subset U \subset X} \Perf( \Spec(A_U) \times X ) \simeq \fcolim_{K \subset U \subset X} \Perf( U \times X\an )
	\end{equation}
	in $\Ind(\Cat_\infty^{\mathrm{st},\otimes})$.
	This is the content of the \canal part of \cref{thm:relative_GAGA}.
	Combining this result with the technique introduced in \cref{thm-intro:generalized_adjunction} allows to easily deduce \cref{thm-intro:relative_GAGA}.
	However, we would like to emphasize that \cref{thm-intro:perfect_complexes_compact_Stein} is a fundamental ingredient in the proof of the equivalence \eqref{eq:ind_relative_GAGA}.
\end{rem}

We conclude this introduction by mentioning two more applications. In \cref{prop:analytification_on_limits} we show that the derived analytification functor commutes with finite limits of geometric stacks. In \cref{cor:period_domain} we revisit the derived period domain from \cite{DiNatale_Global_Period_2016} and construct it as a derived analytic moduli stack. The original derived period domain in loc.\ cit.\ was constructed in an ad hoc way by analytifying an algebraic moduli stack. In fact, the construction of the derived period map in loc.\ cit.\ could be simplified using our \cref{thm-intro:generalized_adjunction} to bridge the algebraic and analytic aspects of the problem. We will not pursue this approach in this paper.\\

\paragraph{\bf{Notation and conventions}}

In this paper we freely use the language of $\infty$-categories.
Although the discussion is often independent of the chosen model for $\infty$-categories, whenever needed we identify them with quasi-categories and refer to \cite{HTT} for the necessary foundational material.

The notations $\cS$ and $\Cat_\infty$ are reserved to denote the $\infty$-categories of spaces and of $\infty$-categories, respectively.
If $\cC \in \Cat_\infty$ we denote by $\cC^\simeq$ the maximal $\infty$-groupoid contained in $\cC$.
We let $\Cat_\infty^{\mathrm{st}}$ denote the $\infty$-category of stable $\infty$-categories with exact functors between them.
We also let $\PrL$ denote the $\infty$-category of presentable $\infty$-categories with left adjoints between them.
Similarly, we let $\PrL_{\mathrm{st}}$ denote the $\infty$-categories of stably presentable $\infty$-categories with left adjoints between them.
Finally, we set
\[ \Cat_\infty^{\mathrm{st}, \otimes} \coloneqq \CAlg( \Cat_\infty^{\mathrm{st}} ) \quad , \quad \mathcal P \mathrm r_{\mathrm{st}}^{\mathrm L, \otimes} \coloneqq \CAlg( \PrL_{\mathrm{st}} ) . \]

Given an $\infty$-category $\cC$ we denote by $\PSh(\cC)$ the $\infty$-category of $\cS$-valued presheaves.
We follow the conventions introduced in \cite[\S 2.4]{Porta_Yu_Higher_analytic_stacks_2014} for $\infty$-categories of sheaves on an $\infty$-site.

For a field $k$, we reserve the notation $\CAlg_k$ for the $\infty$-category of simplicial commutative rings over $k$.
We often refer to objects in $\CAlg_k$ simply as \emph{derived commutative rings}.
We denote its opposite by $\dAff_k$, and we refer to it as the $\infty$-category of \emph{derived affine schemes}.
We say that a derived ring $A \in \CAlg_k$ is \emph{almost of finite presentation} if $\pi_0(A)$ is of finite presentation over $k$ and $\pi_i(A)$ is a finitely presented $\pi_0(A)$-module.\footnote{Equivalently, $A$ is almost of finite presentation if $\pi_0(A)$ is of finite presentation and the cotangent complex $\mathbb L_{A/k}$ is an almost perfect complex over $A$.}
We denote by $\dAff_k^{\mathrm{afp}}$ the full subcategory of $\dAff_k$ spanned by derived affine schemes $\Spec(A)$ such that $A$ is almost of finite presentation.
When $k$ is either a non-archimedean field equipped with a non-trivial valuation or is the field of complex numbers, we let $\An_k$ denote the category of analytic spaces over $k$.
We denote by $\Sp(k)$ the analytic space associated to $k$.

Throughout the paper we need to consider both stacks with values in $\cS$ and with values in $\Cat_\infty$.
We use the following convention: if $(\cC, \tau)$ is an $\infty$-site and $F \colon \cC\op \to \Cat_\infty$ is a $\Cat_\infty$-valued stack, we denote by $\mathbf F$ the $\cS$-valued stack defined by
\[ \mathbf F(X) \coloneqq F(X)^\simeq \quad , \quad X \in \cC . \]
Given stacks $T \colon \cC\op \to \cS$ and $F \colon \cC\op \to \Cat_\infty$ we let $\bfMap(T,F)$ denote the $\Cat_\infty$-valued stack defined by
\[ \bfMap(T, F)(X) \coloneqq F(T \times X) . \]
Here we are implicitly extending $F$ to a functor $\St(\cC, \tau)\op \to \Cat_\infty$.
Notice that
\[ \big( \bfMap(T,F)(X) \big)^{\simeq} \simeq \bfMap(T, \mathbf F)(X) \quad , \quad X \in \cC , \]
where the $\bfMap$ on the right hand side now denotes the internal hom in $\St(\cC, \tau)$.

Finally, in this paper we are concerned with ind and (to a lesser extent) pro-objects.
Given an $\infty$-category $\cC$ we let $\Ind(\cC)$ and $\mathrm{Pro}(\cC)$ denote the $\infty$-categories of ind and pro objects in $\cC$, respectively.
If $I$ is a filtered category and $F \colon I \to \cC$ is a diagram, we let $\fcolim_{i \in I} F(i)$ denote the associated ind-object in $\Ind(\cC)$.
We use the notation $\flim$ for pro-objects.\\

\paragraph{\textbf{Acknowledgments}}

We are grateful to J.\ Ant\'onio, J.\ Calabrese, D.\ Calaque, G.\ Ginot, J.\ Hilburn, V.\ Melani, F.\ Petit, M.\ Robalo, C.\ Simpson, B.\ To\"en, G.\ Vezzosi and T.\ Y.\ Yu for useful discussions related to the content of this paper.

This research has been partially conducted while M.\ P.\ was supported by Simons Foundation grant number 347070 and the group GNSAGA. An important part of this research was accomplished when J.\ H.\ visited M.\ P.\ at the University of Pennsylvania supported by a ``Research in pairs'' grant (41653) under Scheme 4 of the London Mathematical Society.

\section{Review of derived analytic geometry} \label{sec:review}

We start by reviewing the basic notions and facts about derived analytic geometry.
We refer the reader to the papers \cite{DAG-IX,Porta_Yu_DNAnG_I,Porta_DCAGI,Porta_Yu_Representability} for more extensive discussions of the foundations.

\subsection{Definitions and basic facts}

We let $k$ denote either the field $\mathbb C$ of complex numbers or a complete non-archimedean field with non-trivial valuation.

\begin{notation}
	\begin{enumerate}
		\item Let $\cTdisck$ denote the full subcategory of $k$-schemes spanned by affine spaces $\mathbb A^n_k$.
		A morphism in $\cTdisck$ is said to be \emph{admissible} if it is an isomorphism.
		We endow $\cTdisck$ with the trivial Grothendieck topology.
		
		\item Let $\cTetk$ denote the category of smooth $k$-schemes.
		A morphism in $\cTank$ is said to be \emph{admissible} if it is an \'etale morphism.
		We endow $\cTank$ with the \'etale topology $\tauet$.
		
		\item Let $\cTank$ denote the category of smooth \kanal spaces.
		A morphism in $\cTank$ is said to be \emph{admissible} if it is an \'etale morphism.
		We endow $\cTank$ with the \'etale topology $\tauet$.
	\end{enumerate}
\end{notation}

\begin{defin}
	Let $\cX$ be an $\infty$-topos.
	A \emph{$\cTank$-structure} is a functor $\cO \colon \cTank \to \cX$ commuting with products and pullbacks along admissible morphisms.
	We denote by $\Str_{\cTank}(\cX)$ the full subcategory of $\Fun(\cTank, \cX)$ spanned by $\cTank$-structures.
\end{defin}

\begin{defin}
	Let $\cX$ be an $\infty$-topos.
	A $\cTank$-structure $\cO$ is said to be \emph{local} if it takes $\tauet$-covers to effective epimorphisms.
	A morphism of $\cTank$-structures $\cO \to \cO'$ is said to be \emph{local} if for every admissible morphism $U \to V$ the square
	\[ \begin{tikzcd}
		\cO(U) \arrow{r} \arrow{d} & \cO(V) \arrow{d} \\
		\cO'(U) \arrow{r} & \cO'(V)
	\end{tikzcd} \]
	is a pullback square in $\cX$.
	We denote by $\Strloc_{\cTank}(\cX)$ the (non full) subcategory of $\Str_{\cTank}(\cX)$ spanned by local structures and local morphisms between them.
\end{defin}

\begin{rem}
	One can give similar definitions for $\cTetk$ and $\cTdisck$.
	Notice that a $\cTdisck$-structure is simply a product preserving functor $\cO \colon \cTdisck \to \cX$.
	For this reason, we can canonically identify $\Str_{\cTdisck}(\cX)$ with the $\infty$-category of derived commutative rings $\CAlg_k(\cX)$ in $\cX$.
	When $\cX = \cS$ is the $\infty$-topos of spaces, $\CAlg_k(\cX)$ coincides with the underlying $\infty$-category of the model category of simplicial commutative $k$-algebras.
\end{rem}

\begin{eg} \label{eg:derived_analytic_space}
	\begin{enumerate}
		\item Let $X$ be a \canal space and let $X^{\mathrm{top}}$ denote its underlying topological space.
		Let $\cX \coloneqq \Sh(X^{\mathrm{top}})$ be the $\infty$-topos of sheaves on $X$.
		We define a $\cTanc$-structure $\cO$ on $\cX$ as the functor sending an object $U \in \cTanc$ to the sheaf $\cO(U) \in \cX$ defined by
		\[ \mathrm{Op}(X^{\mathrm{top}}) \ni V \mapsto \cO(U)(V) \coloneqq \Hom_{\An_{\mathbb C}}(V, U) , \]
		where $\mathrm{Op}(X^{\mathrm{top}})$ denotes the poset of open subsets of $X^{\mathrm{top}}$.
		Notice that $\cO(\mathbf A^1_{\mathbb C})$ coincides with the usual sheaf of holomorphic functions on $X$.
		
		\item Let $X$ be a rigid analytic space and let $X\et$ denote the small \'etale site of $X$.
		Let $\cX \coloneqq \Sh(X\et, \tauet)^\wedge$ be the hypercompletion of the $\infty$-topos of sheaves on $X\et$.
		Then we can define a $\cTank$-structure $\cO$ on $\cX$ as the functor sending $U \in \cTank$ to the sheaf $\cO(U) \in \cX$ defined by
		\[ X\et \ni V \mapsto \cO(U)(V) \coloneqq \Hom_{\Ank}(V, U) . \]
		Once again, $\cO(\mathbf A^1_k)$ coincides with the usual sheaf of analytic functions on $X$.
	\end{enumerate}
\end{eg}

The analytification functor introduced in the \canal case in \cite[\S VIII]{SGA1} and in the \kanal case in \cite{Berkovich_Spectral_1990} restricts to a functor
\[ (-)\an \colon \cTdisck \to \cTank . \]
Precomposition with $(-)\an$ provides for every $\infty$-topos $\cX$ a functor
\[ (-)\alg \colon \Str_{\cTank}(\cX) \longrightarrow \Str_{\cTdisck}(\cX) \simeq \CAlg_k(\cX) . \]
We refer to this functor as the \emph{underlying algebra functor}.

\begin{defin} \label{def:derived_analytic_space}
	A \emph{derived \kanal space} is a pair $(\cX, \cO_X)$ where $\cX$ is a hypercomplete $\infty$-topos and $\cO_X$ is a $\cTank$-structure on $\cX$ such that:
	\begin{enumerate}
		\item locally on $\cX$, $(\cX, \pi_0 \cO_X)$ is equivalent to a $\cTank$-structured topos arising from the construction of \cref{eg:derived_analytic_space};
		\item the sheaves $\pi_i(\cO_X\alg)$ are coherent sheaves of $\pi_0(\cO_X\alg)$-modules.
	\end{enumerate}
\end{defin}

\begin{thm}[{\cite{DAG-IX,Porta_Yu_DNAnG_I}}]
	Derived \kanal spaces assemble into an $\infty$-category $\dAnk$ that satisfies the following properties:
	\begin{enumerate}
		\item fibre products in $\dAnk$ exist;
		\item the construction of \cref{eg:derived_analytic_space} provides a fully faithful embedding of the category of ordinary \kanal spaces $\An_k$ in $\dAnk$.
	\end{enumerate}
\end{thm}

One of the difficult points of the above theorem is to actually construct $\dAnk$ as an $\infty$-category.
This is achieved by the general methods of \cite{DAG-V}, realizing $\dAnk$ as a full subcategory of the $\infty$-category of $\cTank$-structured $\infty$-topoi $\RTop(\cTank)$.
More generally, one can define $\RTop(\cT)$ whenever $\cT$ is a pregeometry.
We refer the reader to \cite[Definition 1.4.8]{DAG-V} for a detailed construction.

\begin{rem}
	The above theorem gives a first hint that the notion of derived analytic space introduced in \cite{DAG-IX,Porta_Yu_DNAnG_I} is a solid one.
	Since the appearance of these papers, the theory has been greatly developed.
	We mention a version of the GAGA theorem in the derived setting, that has been obtained in \cite{Porta_DCAGI}, and a detailed analysis of (derived) deformation theory in \cite{Porta_Yu_Representability} that led to an analytic version of Lurie's representability theorem.
	On the applications side, we mention derived versions of the Riemann-Hilbert correspondence \cite{Porta_Derived_Riemann_Hilbert} and of the Griffiths period map \cite{DiNatale_Global_Period_2016}.
\end{rem}

\subsection{Derived affinoid, Stein and compact Stein spaces}

To any derived analytic space $X = (\cX, \cO_X)$ we can canonically attach an analytic space
\[ \trunc(X) \coloneqq (\cX, \pi_0( \cO_X )) . \]
We refer to $\trunc(X)$ as the truncation of $X$.
The truncation allows to define the derived counterparts of Stein and $k$-affinoid spaces:

\begin{defin}
	A derived analytic space $X \in \dAn_k$ is said to be a derived Stein space (in the \canal case) or a derived $k$-affinoid space (in the non-archimedean case) if its truncation $\trunc(X)$ is a Stein or $k$-affinoid space, respectively.
	We denote by $\dStn_\C$ (resp.\ $\dAfd_k$) the full subcategory of $\dAn_\C$ (resp.\ $\dAn_k$) spanned by derived Stein spaces (resp.\ derived $k$-affinoid spaces).
\end{defin}

\begin{notation}
	In this paper we made an effort to present as far as possible statements that are equally true in the complex and non-archimedean analytic case.
	In particular, following the convention of \cite{Porta_Yu_Higher_analytic_stacks_2014}, we say ``analytic'' whenever the statement applies to both settings.
	When $k$ is not specified and can be either $\C$ or a non-archimedean field, we use the notation $\dAfd_k$ to also denote $\dStn_\C$.
\end{notation}

In \cite[\S 3.1]{Porta_DCAGI} and in \cite[\S 7.1]{Porta_Yu_DNAnG_I} it is shown that the \'etale topology defines a Grothendieck topology on $\dAfd_k$.
We set
\[ \dAnSt_k \coloneqq \Sh( \dAfd_k, \tauet )^\wedge . \]
We refer to this $\infty$-category as the $\infty$-category of derived analytic stacks.
Moreover, let $\bPsm$ denote the collection of smooth morphisms in $\dAfd_k$ (cf.\ \cite[Definition 5.46]{Porta_Yu_Representability}).
Then $(\dAfd_k, \tauet, \bPsm)$ is a geometric context in the sense of \cite[Definition 2.2]{Porta_Yu_Higher_analytic_stacks_2014}.
In particular, the notion of derived analytic geometric stack is defined.\\

In dealing with (derived) \canal geometry, a frequent difficulty one encounters is that we cannot identify coherent sheaves on a Stein space with modules of finite presentation over the global sections.
The classical solution to this problem, as can be found in \cite[Proposition 11.9.2]{Taylor_Several_complex}, is to work with compact Stein spaces.
In loc.\ cit.\ a compact Stein space $K$ is a locally ringed space which can be realized as a compact subset of a Stein space $U$, admitting a fundamental system of Stein open neighbourhoods.
The sheaf of functions on $K$ is the sheaf of overconvergent functions on $K$.
However, considering $K$ as an actual locally ringed space has several disadvantages: first of all, it is difficult to generalize to the derived setting, and second it often requires an extra noetherianity hypothesis on $K$.
We will circumvent these issues by considering a compact Stein as a pro-object in $\dAn_\C$ (see in particular \cref{thm:perfect_complexes_compact_Stein}):

\begin{construction} \label{construction:compact_Stein}
	Let $X \in \dAn_\C$ and let $K \subset \trunc(X)$ be a compact subset of $\trunc(X)$.
	If $U \subset X$ is an open immersion of derived analytic spaces, we write $K \subset U$ to mean that $K \subset \trunc(U)$.
	Suppose now that $K$ admits a fundamental system of Stein open neighbourhoods inside $\trunc(X)$.
	Using the equivalence of sites $\trunc(X)\et \simeq X\et$, we can interpret any open neighbourhood of $K$ inside $\trunc(X)$ as a derived analytic space which is open inside $X$.
	We therefore define
	\[ (K)_X \coloneqq \flim_{K \subset U \subset X} U \in \mathrm{Pro}( \dAnSt_\C ) , \]
	where the colimit ranges over all the open Stein neighbourhoods of $K$ inside $X$.
\end{construction}

\begin{defin} \label{def:compact_Stein}
	A derived compact Stein space is a pro-object which is equivalent to the pro-object $(K)_X$ arising from \cref{construction:compact_Stein}.
\end{defin}

Compact Stein spaces are especially useful in virtue of the following theorem:

\begin{thm}[{cf.\ \cite[7.3.4.10]{HTT}}] \label{thm:sheaves_compact_subsets}
	Let $X$ be a locally compact topological space and let $\cC$ be a presentable $\infty$-category in which filtered colimits are left exact.
	Then there is an equivalence of $\infty$-categories
	\[ \Sh(X; \cC) \simeq \Sh_\cK(X; \cC) , \]
	where the right hand side denotes the sheaves on compact subsets of $X$, in the sense of \cite[7.3.4.1]{HTT}.
\end{thm}

In applications, it is important to know explicitly how the above equivalence works.
Let therefore $X$ be a locally compact space.
Let us denote by $\cK(X)$ the set of compact subsets of $X$ and by $\cU(X)$ the set of open subsets of $U$.
We order both $\cK(X)$, $\cU(X)$ and their union $\cK(X) \cup \cU(X)$ by inclusion.
Let
\[ \kappa \colon \cK(X) \hookrightarrow \cK(X) \cup \cU(X) \quad , \quad u \colon \cU(X) \hookrightarrow \cK(X) \cup \cU(X) \]
be the natural inclusions.
Then \cite[7.3.4.9]{HTT} shows that the fully faithful functors
\[ \mathrm{Ran}_{\kappa} \colon \Sh_{\cK}(X;\cC) \hookrightarrow \Fun( ( \cK(X) \cup \cU(X) )\op, \cS ) \hookleftarrow \Sh(X;\cC) \colon \mathrm{Lan}_u \]
have the same essential image.
Let now $\cF \in \PSh(X;\cC) \coloneqq \Fun(\cU(X)\op, \cC)$.
Let $\widetilde{\cF} \coloneqq \mathrm{Lan}_u(\cF)$.
There is a natural transformation
\[ \eta \colon \widetilde{\cF} \longrightarrow \mathrm{Ran}_\kappa\left( \widetilde{\cF} |_{\cK(X)} \right) . \]
By restricting to $\cU(X)$ and using the full faithfulness of $\mathrm{Lan}_u$ we obtain a natural transformation
\[ \eta_\cU \colon \cF \longrightarrow \left. \mathrm{Ran}_\kappa\left( \widetilde{\cF} |_{\cK(X)} \right) \right|_{\cU(X)} . \]
Using \cite[7.3.4.9]{HTT} we immediately obtain the following result:

\begin{lem} \label{lem:sheaves_open_sheaves_compact}
	With the above notations, suppose furthermore that $\widetilde{\cF} |_{\cK(X)}$ belongs to $\Sh_{\cK}(X;\cC)$.
	Then $\eta_\cU$ exhibits $\left. \mathrm{Ran}_\kappa\left( \widetilde{\cF} |_{\cK(X)} \right) \right|_{\cU(X)}$ as the sheafification of $\cF$.
\end{lem}

\begin{proof}
	Let $\cG \in \Sh(X;\cC)$.
	Then we have
	\[ \Map_{\PSh(X;\cC)}(\cF, \cG) \simeq \Map_{\Fun(( \cK(X) \cup \cU(X) )\op,\cC)}( \mathrm{Lan}_u(\cF), \mathrm{Lan}_u(\cG) ) . \]
	Let $\widetilde{\cG} \coloneqq \mathrm{Lan}_u(\cG)$.
	Then \cite[7.3.4.9]{HTT} implies that $\widetilde{\cG} \simeq \mathrm{Ran}_\kappa( \widetilde{\cG} |_{\cK(X)} )$.
	In particular, $\eta$ induces an equivalence
	\[ \Map_{\Fun(( \cK(X) \cup \cU(X) )\op,\cC)}( \mathrm{Lan}_u(\cF), \mathrm{Lan}_u(\cG) ) \simeq \Map_{\Sh_\cK(X;\cC)}(\widetilde{\cF}|_{\cK(X)}, \widetilde{\cG}|_{\cK(X)}) . \]
	Applying again right Kan extension along $\kappa$ and restricting to $\cU(X)$ shows finally that $\eta_\cU$ induces an equivalence
	\[ \Map_{\PSh(X;\cC)}(\cF, \cG) \simeq \Map_{\Sh(X;\cC)}\left( \left. \mathrm{Ran}_\kappa\left( \widetilde{\cF} |_{\cK(X)} \right) \right|_{\cU(X)}, \cG \right) . \]
	The proof is therefore complete.
\end{proof}

\section{Analytification of geometric stacks} \label{sec:analytification_geometric_stacks}

\subsection{The derived analytification functor}\label{subsec:derived_analytification}

The analytification functor
\[ (-)\an \colon \cTetk \to \cTank \]
respects the classes of admissible morphisms and the coverings, and so it is a transformation of pregeometries.
As a consequence, \cite[Theorem 2.1.1]{DAG-V} shows that it gives rise to an adjunction of $\infty$-categories
\[ (-)\alg \colon \RTop(\cTank) \leftrightarrows \RTop(\cTetk) \colon (-)\an . \]
The functor $(-)\alg$ can be informally described as the functor mapping a $\cTank$-structured topos $(\cX, \cO_X)$ to the $\cTetk$-structured topos $(\cX, \cO_X\alg)$.
We refer to the right adjoint $(-)\an \colon \RTop(\cTetk) \to \RTop(\cTank)$ as the \emph{derived analytification functor}.
We can summarize the main properties of this functor in the following theorem:

\begin{thm}[{cf.\ \cite{Porta_DCAGI,Porta_Yu_Representability}}] \label{thm:generalities_analytification}
	Let $X = (\cX, \cO_X)$ be a derived \DM stack locally almost of finite presentation.
	Then:
	\begin{enumerate}
		\item $X\an \in \RTop(\cTank)$ is a derived analytic space.
		\item The canonical map $\varepsilon_X \colon (X\an)\alg \to X$ in $\RTop(\cTetk)$ is flat.
		\item If furthermore $X$ is an underived scheme, then under the fully faithful embedding $\An_k \hookrightarrow \dAnk$ the analytification $X\an$ introduced in \cite{Berkovich_Spectral_1990,SGA1} coincides with the derived analytification of $X$.
	\end{enumerate}
\end{thm}

In many situations of geometric interest, \DM stacks are too restrictive and need to be replaced by geometric stacks (also known as Artin stacks).
Having defined the derived analytification functor at the level of derived \DM stacks locally almost of finite presentation, it is straightforward to extend it to arbitrary derived stacks locally almost of finite presentation by left Kan extension (cf.\ \cite[\S 6.1]{Porta_Yu_Higher_analytic_stacks_2014} and \cite{Toen_Algebrisation_2008}).
This procedure is implicitly used in \cite{Porta_Yu_Representability,Porta_DCAGI}.
In this paper we need a slightly more general procedure that allows to define the analytification of arbitrary derived stacks (not necessarily locally almost of finite presentation).
The construction is as follows.\\

The functor $(-)\an \colon \RTop(\cTetk) \to \RTop(\cTank)$ restricts to
\[ (-)\an \colon \dAff_k^{\mathrm{afp}} \longrightarrow \dAfd_k . \]
This is a continuous morphism of sites, and therefore it induces a functor
\[ (-)^{\mathrm{an, afp}} \colon \dSt_k^{\mathrm{afp}} \longrightarrow \dAnSt_k . \]
On the other hand, let
\begin{equation} \label{eq:afp_inside_all}
	\begin{tikzcd}
		j \colon \dAff_k^{\mathrm{afp}} \arrow[hook]{r} & \dAff_k
	\end{tikzcd}
\end{equation}
be the natural inclusion.
Notice that $j$ is both continuous and cocontinuous morphism of sites.
In particular, restriction along $j$ provides a functor
\[ j^s \colon \dSt_k \to \dSt_k^{\mathrm{afp}} , \]
that admits both a left adjoint $j_s$ and a right adjoint ${}_s j$.
Notice that the unit $\mathrm{Id}_{\dSt_k^{\mathrm{afp}}} \to j^s \circ j_s$ is an equivalence.
Indeed, since both $j_s$ and $j^s$ are left adjoint, they commute with colimits.
We can therefore reduce ourselves to check this assertion on representables, where it is a direct consequence of \cite[Lemma 2.16]{Porta_Yu_Higher_analytic_stacks_2014}.

\begin{defin}
	We define the derived analytification $(-)\an \colon \dSt_k \to \dAnSt_k$ as the composition
	\[ \begin{tikzcd}[column sep = large]
	\dSt_k \arrow{r}{j^s} & \dSt_k^{\mathrm{afp}} \arrow{r}{(-)^{\mathrm{an,afp}}} & \dAnSt_k .
	\end{tikzcd} \]
\end{defin}

\subsection{A universal property of analytification}

When $Y$ is a derived \DM stack locally almost of finite type and $X$ is a derived analytic space, the very definition of $(-)\an$ implies that the canonical map
\[ \Map_{\dAnk}(X,Y\an) \longrightarrow \Map_{\RTop(\cTetk)}(X\alg, Y) \]
is an equivalence.
However it is unreasonable to expect to be able to lift the above adjunction to the level of the categories $\dAnSt_k$ and $\dSt_k$, even when restricting to geometric stacks on both sides.
The reason is that there is a significant difference between the object $X\alg$ and its restricted functor of points
\[ \Map_{\RTop(\cTetk)}(-, X\alg) \colon \dAff_k \longrightarrow \cS . \]
For instance the global sections of these two objects differ.
In order to bypass this difficulty, we adapt the method introduced first in \cite{Lurie_Tannaka_duality}, which consists in providing an alternative description of both $\Map_{\dAnk}(X, Y\an)$ and $\Map_{\RTop(\cTetk)}(X\alg, Y)$ as sheaves on $X$.\\

Let $X \in \dAnSt_k$.
Define
\[ \dAfd_X \coloneqq \dAfd_k \times_{\dAnSt_k} (\dAnSt_k)_{/X} . \]
We endow $\dAfd_X$ with the \'etale topology, that we still denote $\tauet$.
We denote by $\cX_X$ the corresponding $\infty$-topos:
\[ \cX_X \coloneqq \St(\dAfd_X, \tauet) . \]
We denote this $\infty$-topos by $\cX_X$.
Consider the forgetful functor
\[ F_X \colon \dAfd_X \longrightarrow \dAfd_k , \]
that sends a morphism $U \to X$ to the source $U$.
Then $F_X$ is both a continuous and cocontinuous morphism of sites, and therefore we obtain the following result:

\begin{lem} \label{lem:F_continuous_and_cocontinuous}
	The functor
	\[ F_X^s \colon \dAnSt_k \to \cX_X \]
	commutes with colimits and it coincides with the restriction of $F_X^p \colon \PSh(\dAnk) \to \cX$ to $\dAnSt_k$.
	In particular one has
	\begin{equation} \label{eq:F_global_section_maps}
		\Map_{\cX_X}(\mathbf 1_X, F_X^s(Y\an)) \simeq \Map_{\dAnSt_k}(X, Y\an) ,
	\end{equation}
	where $\mathbf 1_X$ denotes the final object of $\cX_X$.
\end{lem}

\begin{proof}
	The functor $F_X^s$ commutes with colimits thanks to \cite[Lemma 2.19]{Porta_Yu_Higher_analytic_stacks_2014} because $F_X$ is a cocontinuous morphism of sites.
	Similarly, it coincides with the restriction of $F_X^p$ in virtue of \cite[Lemma 2.13]{Porta_Yu_Higher_analytic_stacks_2014} because $F_X$ is continuous.	
	Finally, we observe that the identification \eqref{eq:F_global_section_maps} follows from the fact that $F_X^s$ coincides with the restriction of $F_X^p$ to $\dAnSt_k$.
\end{proof}

\begin{rem}
	The functor $F_X \colon \dAfd_X \to \dAfd_k$ does not preserve (finite) products.
	Therefore, it follows that $F^s_X$ is \emph{not} part of a geometric morphism of $\infty$-topoi.
	Nevertheless, it still has a left adjoint $(F_{X})_s$ and a right adjoint ${}_s F_X$.
	However, $(F_X)_s$ does \emph{not} commute with finite limits.
\end{rem}

We now turn to describe $\Map_{\RTop(\cTetk)}(X\alg,Y)$ as a the global sections of a sheaf on $\dAfd_X$.
The main point of doing this is that the new formulation will make sense for an arbitrary $Y \in \dSt_k$.
The sheaf of sections of $\cO_X\alg$ provides us with a functor
\[ G_X \colon \dAfd_X \longrightarrow \dAff_k \]
which informally sends $U \to X$ to
\[ G_X(U) \coloneqq \Spec(\cO_X\alg(U)) . \]
Notice that the functor $G_X$ does not factor through $\dAff_k^{\mathrm{afp}}$, and moreover is not continuous because it does not commute with fiber products (not even along \'etale morphisms).
However, we can at least prove that it is cocontinuous:

\begin{lem} \label{lem:cocontinuous}
	The functor $G_X \colon (\dAfd_X, \tauet) \longrightarrow (\dAff_k, \tauet)$ is a cocontinuous morphism of sites.
\end{lem}

\begin{proof}
	Consider the functor
	\[ \overline{\cO} \colon \CAlg_k \to \cX_X \]
	defined by
	\[ R \mapsto \Map_{\CAlg_k}(R, \cO(-)) \in \cX_X . \]
	We first show $\overline \cO$ takes covers to effective epimorphisms.
	The restriction of $\overline{\cO}$ to $\cTetk$ is a $\cTetk$-structure on $\cX_X$ that canonically coincides with $\cO$ itself.
	By \cite[Proposition 3.4.7]{DAG-V} we can extend $\overline{\cO}$ to a $\cG\et(k)$-structure on $\cX_X$.
	Here $\cG\et(k)$ is the geometric envelope of $\cTetk$, which can be explicitly described as the full subcategory of $\CAlg_k$ spanned by compact objects, see \cite[Definition 4.3.13, Proposition 4.3.15]{DAG-V}.
	It follows that $\overline{\cO}$ takes admissible maps between algebras of finite type (which are exactly $\tauet$-coverings) to effective epimorphisms in $\cX$.
	
	To check that $\overline{\cO}$ takes all $\tauet$-coverings to effective epimorphisms we note that it suffices to check on stalks. 
	Then to check $\overline{\cO}(f)$ is an effective epimorphism we consider an \'etale covering $R \to \prod R_{i}$ and need to produce a lift in the diagram $\cO_{X,x} \leftarrow R \rightarrow \prod R_{i}$. Using the local representation of an \'etale map by a standard \'etale map one may see that the \'etale covering is pro-admissible. 
	
	Then by \cite[Proposition 1.3.10]{DAG-V} pro-admissible maps form part of a factorization system, see \cite[Definition 5.2.8.8]{HTT}, on $\CAlg_{k}$, considered as $\mathrm{Pro}(\cG\et(k))$. 
	The right set of this factorization system is given by local morphisms and in particular $\cO_{X,x} \to k$ is right orthogonal to $R \to \prod R_{i}$, which shows the desired lift exists \cite[Remark 5.2.8.2]{HTT}.

	Let now $U \in \dAfd_X$ and fix an \'etale cover
	\[ \cO(U) \to \prod_i R_i . \]
	The above observation implies that the map
	\[ \coprod_i \overline{\cO}(R_i) \to \overline{\cO}(\cO(U)) \]
	is an effective epimorphism.
	In particular it is an effective epimorphism of sheaves after applying $\pi_{0}$ and thus for every $V \in \cX$ and every $f \in \pi_{0}\overline{\cO}(\cO(U))(V)$, we can find an effective epimorphism
	\[ \coprod V_j \to V \]
	in $\cX_X$ such that for every $j$ there exists some $i$ and some element $f_{ij} \in \pi_{0}\overline{\cO}(R_i)(V_j)$ whose image via
	\[ \overline{\cO}(R_i)(V_j) \to \overline{\cO}(\cO(U))(V_{j}) \]
	coincides with the image of $f$ via the restriction $\overline{\cO}(\cO(U))(V) \to \overline{\cO}(\cO(U))(V_j)$.
	Applying this reasoning to the case $V = U$ and
	\[ f \coloneqq \mathrm{id}_{\cO(U)} \in \overline{\cO}(\cO(U))(U) = \Map_{\CAlg_k}(\cO(U), \cO(U)) , \]
	we deduce the existence of an effective epimorphism
	\[ \coprod U_j \to U \]
	and factorizations
	\[ \cO(U) \to R_i \to \cO(U_j) . \]
	The proof is therefore complete.
\end{proof}

\begin{cor} \label{cor:G_cocontinuous}
	Let $X \in \dAnSt_k$ be a derived \kanal stack.
	Then the functor
	\[ G^s_X \colon \dSt_k \longrightarrow \cX_X \]
	induced by $G_X \colon \dAfd_X \to \dAff_k$ commutes with colimits.
	In particular if $U \to Y$ is an effective epimorphism and $U_{\bullet}$ is its \v{C}ech nerve then the canonical morphism
	\begin{equation} \label{eq:G_geometric_realization}
		\left| G_X^s(U_\bullet) \right| \longrightarrow G_X^s(Y)
	\end{equation}
	is an equivalence.
\end{cor}

\begin{proof}
	\Cref{lem:cocontinuous} guarantees that the morphism of sites $G_X \colon (\cX, \tauet) \to (\dAff_k, \tauet)$ is cocontinuous.
	Therefore, \cite[Lemma 2.19]{Porta_Yu_Higher_analytic_stacks_2014} shows that it induces a well defined $\infty$-functor
	\[ G^s_X \colon \mathrm{dSt}_k \longrightarrow \cX \]
	which is furthermore left adjoint to ${}_s G_X$.
	In particular, $G^s_X$ commutes with arbitrary colimits.
\end{proof}

Let now $Y \in \dSt_k^{\mathrm{afp}}$ be a derived stack locally almost of finite presentation.
Then $Y\an$ is defined as an object in $\dAnSt_k$.
In particular, for every $X \in \dAn_k$ both $F_X^s(Y\an)$ and $G_X^s(Y)$ are defined.
The main goal of this section is to prove that they are canonically equivalent whenever $Y$ is furthermore geometric.\\

Let $U \in \dAfd_X$ and represent it as $U = (\cU, \cO_U)$.
Therefore $U\alg = (\cU, \cO_U\alg)$ and the universal property of the $\Spec$ functor of \cite[\S 2.2]{DAG-V} induces a natural transformation in $\RTop(\cTetk)$
\[ \varepsilon_U \colon U\alg \longrightarrow \Spec( \cO_U\alg(U) ) . \]
For any $Y \in \dAff_k^{\mathrm{afp}}$, this provides us with a natural transformation
\[ \alpha_{U,Y} \colon \Map_{\dSt_k}(\Spec(\cO_U\alg(U)), Y) \longrightarrow \Map_{\RTop(\cTetk)}(U\alg, Y) \simeq \Map_{\dAnSt_k}(U, Y\an) . \]
Notice that
\[ \Map_{\dSt_k}(\Spec(\cO_U\alg(U)), Y) \simeq G_X^p(Y)(U) , \]
and
\[ \Map_{\dAnSt_k}(U, Y\an) \simeq F_X^s(Y\an)(U) . \]
As $G^p \circ j_p$ commutes with colimits, the morphisms $\alpha_{U,Y}$ extend to a natural transformation between functors $\PSh(\dAff_k^{\mathrm{afp}}) \to \PSh(\dAfd_X)$:
\[ \widetilde{\alpha} \colon G^p_X \circ j_p \longrightarrow F_X^s \circ (-)\an \circ j_s \simeq F_X^s \circ (-)^{\mathrm{an, afp}} , \]
where $j_p$ and $j_s$ are the functors induced by the morphism of sites \eqref{eq:afp_inside_all}.
As $F_X^s \circ (-)^{\mathrm{an,afp}}$ is a sheaf, we see that this natural transformation induces
\[ \alpha \colon G^s_X \circ j_s \longrightarrow F_X^s \circ (-)^{\mathrm{an,afp}} . \]
We can now state the main theorem of this section:

\begin{thm} \label{thm:generalized_adjunction}
	Let $X \in \dAnSt_k$ be a derived analytic stack.
	If $Y \in \dSt_k^{\mathrm{afp}}$ is a geometric derived stack locally almost of finite presentation, the morphism
	\[ \alpha_Y \colon G^s_X(Y) \longrightarrow F_X^s(Y\an) \]
	is an equivalence in $\cX_X$.
\end{thm}

In particular, $\alpha_Y$ induces an equivalence
\[ \Map_{\cX_X}(\mathbf 1_X, G^s_X(Y)) \simeq \Map_{\cX_X}(\mathbf 1_X, F_X^s(Y\an)) . \]
In virtue of \cref{lem:F_continuous_and_cocontinuous}, we can identify the right hand side with $\Map_{\dAnSt_k}(X, Y\an)$.
The left hand side plays instead the role of $\Map_{\dSt_k}(X\alg, Y)$.
However, since the functor $G_X \colon \dAfd_X \to \dAff_k$ is not continuous, the functor $G^s_X$ is typically not a right adjoint.
This prevents us from rewriting $\Map_{\cX_X}(\mathbf 1_X, G^s_X(Y))$ as a mapping space computed in $\dSt_k$.
We will nevertheless see that one can effectively use \cref{thm:generalized_adjunction} in order to deal with the analytification of higher geometric stacks such as $\bfPerf_k$.

\begin{proof}[Proof of \cref{thm:generalized_adjunction}]
	We proceed by induction on the geometric level of $Y$.
	Suppose first that $Y = \Spec(A)$ is affine.
	For any $U \in \dAfd_X$, we have
	\begin{align*}
		G^p_X(Y)(U) & \simeq \Map_{\dSt_k}(\Spec(\cO_U\alg(U)), Y ) \\
		& \simeq \Map_{\RTop(\cTetk)}(U\alg, Y) \\
		& \simeq \Map_{\RTop(\cTank)}(U, Y\an) \\
		& \simeq \Map_{\dAnSt_k}(U, Y\an) \\
		& \simeq F^s_X(Y)(U) .
	\end{align*}
	The composition is $\widetilde{\alpha}$.
	As it is an equivalence and $F^s_X(Y)$ is a sheaf, we conclude that $G^p_X(Y) \simeq G^s_X(Y)$, and that $\alpha \colon G^s_X(Y) \to F^s_X(Y)$ is an equivalence as well.
	
	Let now $Y$ be an $n$-geometric derived stack locally almost of finite presentation.
	Choose an $n$-atlas $u \colon U \to Y$ and let $U_\bullet$ be its \v{C}ech nerve.
	Then \cref{lem:F_continuous_and_cocontinuous} and \cref{cor:G_cocontinuous} imply that
	\[ | G^s_X(U_\bullet) | \simeq G^s_X(Y) \quad , \quad | F^s_X(U_\bullet) | \simeq F^s_X(Y) . \]
	As the natural transformation $\alpha_{U_n} \colon G^s_X(U_n) \to F^s_X(U_n)$ is an equivalence by induction hypothesis for every $[n] \in \mathbf \Delta\op$, we conclude that $\alpha_Y \colon G^s_X(Y) \to F^s_X(Y)$ is an equivalence as well.
\end{proof}

\subsection{Controlling the analytification}

In this paper we are mostly concerned with the following type of question.
Suppose that $X \in \dSt_k^{\mathrm{afp}}$ is a derived geometric stack locally almost of finite presentation.
Its analytification $X\an$ is obtained via a left Kan extension.
This prevents us from providing an easy description of $X\an$ in terms of its functor of points.
Nevertheless, when $X$ itself parametrizes algebraic families of certain kind of objects (such as vector bundles, principal $G$-bundles, perfect complexes, morphisms between algebraic stacks etc.), then there is often an analytic analogue $Y$ parametrizing  analytic families of the same type of objects.
It is then a natural question to compare the analytification of $X$ with its analytic counterpart $Y$.
Our current goal is to describe a general strategy to prove similar statements (see \cref{prop:general_strategy_analytification} for a precise statement and a proof).
In the rest of the paper we will repeatedly apply this strategy.\\

To start we assume given a derived geometric stack locally almost of finite presentation $X \in \dSt_k^{\mathrm{afp}}$, an analytic stack $Y \in \dAnSt_k$ and a morphism
\[ \varepsilon \colon X\an \longrightarrow Y , \]
which we wish to prove an equivalence.
Notice that we do not assume a priori that $Y$ is geometric.
It is enough to check that $\varepsilon$ induces an equivalence
\[ \Map_{\dAnSt_k}(U, X\an) \simeq \Map_{\dAnSt_k}(U, Y) \]
for all $U \in \dAnSt_{k}$.
Using \cref{lem:F_continuous_and_cocontinuous} we see that it is enough to check that $\varepsilon$ induces an equivalence
\[ F^s_U(X\an) \longrightarrow F^s_U(Y) . \]
Using \cref{thm:generalized_adjunction} and the geometricity of $X$, we can replace $F^s_U(X\an)$ by $G^s_U(X)$.
In this way, we get rid of the analytification.
However, checking in practice that the morphism
\[ G^s_U(X) \longrightarrow F^s_U(X\an) \longrightarrow F^s_U(Y) \]
induced by $\varepsilon$ is an equivalence is as difficult as the original problem of proving that $\varepsilon$ is an equivalence.
The reason is that, once again, $G^s_U(X)$ is not explicitly defined, but it is rather the result of a sheafification process.

In the non-archimedean setting, it happens that in the situations we will consider in the subsequent sections, the map
\[ G^p_U(X) \longrightarrow F^s_U(Y) \]
is already an equivalence.
This can ultimately be traced back to Tate's acyclicity and Kiehl's theorem (see for instance \cref{lem:perfect_complexes_affinoid}).
This implies that $G^p_U(X)$ is a sheaf and therefore that $G^p_U(X) \simeq G^s_U(X)$.
In the complex case this statement is typically false.
To remedy this, we are lead to with compact Stein spaces (see \cref{def:compact_Stein}).

\begin{prop} \label{prop:general_strategy_analytification}
	Let $X \in \dSt_k^{\mathrm{afp}}$ be a derived geometric stack locally almost of finite presentation.
	Let $Y \in \dAnSt_k$ be a derived analytic stack and let
	\[ \varepsilon \colon X\an \longrightarrow Y \]
	be a morphism in $\dAnSt_k$.
	Suppose that:
	\begin{enumerate}
		\item If $k$ is a non-archimedean field then for every $U \in \dAfd_k$ the map $\varepsilon$ induces an equivalence
		\[ G^p_U(X) \longrightarrow F^s_U(Y) . \]
		\item If $k = \mathbb C$ then for every $U \in \dStn_\C$ and every compact Stein subset $K \subset U$, $\varepsilon$ induces an equivalence
		\begin{equation} \label{eq:equivalence_ind_object}
			\colim_{K \subset V \subset U} G^p_U(X)(V) \simeq \colim_{K \subset V \subset U} F^s_U(Y)(V) .
		\end{equation}
	\end{enumerate}
	Then $\varepsilon \colon X\an \to Y$ is an equivalence.
\end{prop}

\begin{proof}
	As we already discussed, combining \cref{lem:F_continuous_and_cocontinuous} and \cref{thm:generalized_adjunction} it is enough to check that $\varepsilon$ induces an equivalence
	\[ G^s_U(X) \longrightarrow F^s_U(Y) \]
	for every derived $k$-affinoid (resp.\ Stein) space $U$.
	In the non-archimedean situation, the hypothesis guarantees that $G^p_U(X)$ is a sheaf and that it is equivalent to $F^s_U(Y)$.
	Since $G^s_U(X)$ is the sheafification of $G^p_U(X)$, we conclude that $G^s_U(X) \simeq F^s_U(Y)$ via the morphism induced by $\varepsilon$.
	Therefore $\varepsilon$ is an equivalence.
	
	In the \canal setting, we first use the correspondence provided by \cref{thm:sheaves_compact_subsets} to recast $G^p_U(X)$ and $F^s_U(Y)$ as presheaves defined on compact subsets of $U$.
	We will abuse notation and write $G^p_U(X)(K)$ for $\mathrm{Lan}_u(G^{p}_{U})(K)$ if $K$ is a compact Stein in $U$.
	Then the hypothesis guarantees that
	\[ G^p_U(X)(K) \simeq \colim_{K \subset V \subset U} G^p_U(X)(V) \simeq \colim_{K \subset V \subset U} F^s_U(Y)(V) \simeq F^s_U(Y)(K) . \]
	Therefore, $G^p_U(X)$ is a sheaf on compact Stein subsets of $U$ which is furthermore equivalent to $F^s_U(Y)$.
	As compact Stein subsets of $U$ form a basis for $U$, the conclusion now follows from \cref{lem:sheaves_open_sheaves_compact}.
\end{proof}

\section{Analytic perfect complexes} \label{sec:analytic_perfect_complexes}

As usual we let $k$ be either the field of complex numbers or a non-archimedean field equipped with a non-trivial valuation.
In this section we are concerned with the derived analytic stack $\bfAnPerf_k$ parametrizing families of perfect complexes over derived analytic spaces (see below for its precise definition).
Our main goal is to prove that there is a natural equivalence
\[ \bfPerf_k\an \simeq \bfAnPerf_k . \]
See \cref{prop:analytification_perfect_complexes_absolute}.
The proof is based on the general method described in \cref{prop:general_strategy_analytification}.
Building on the results obtained in \cite{Porta_Yu_Mapping}, it is easy to verify the assumptions of that proposition in the non-archimedean setting.
On the other hand, verifying the hypotheses in the \canal situation requires a lot of extra work.
For this reason, the biggest part of this section is essentially \canal in nature, and the main object of study is the category of perfect complexes on a compact Stein, seen as a pro-object in $\dAn_k$.

\subsection{The stack of perfect complexes}

We start with the basic definitions.
Let $\cX$ be an $\infty$-topos and let $\cO \in \CAlg_k(\cX)$ be a sheaf of connective derived $k$-algebras.
Formally speaking, we set
\[ \CAlg_k(\cX) \coloneqq \Str_{\cTdisck}(\cX) . \]
It can be naturally identified with the $\infty$-category of sheaves of simplicial commutative $k$-algebras on $\cX$.
For every object $U \in \cX$ we can form a stably symmetric monoidal $\infty$-category $\cO |_U \Mod$.
Its objects are the sheaves of $\cO |_U$-modules in the $\infty$-topos $\cX_{/U}$.
See \cite[\S 2.1]{DAG-VIII}.

These categories glue together into a sheaf on $\cX$ with values in $\Cat_\infty^{\mathrm{st}, \otimes}$, the $\infty$-category of stably symmetric monoidal $\infty$-categories.
We denote the resulting functor by
\[ \cO \Mod \colon \cX\op \longrightarrow \Cat_\infty^{\mathrm{st}, \otimes} . \]
Notice that when $\mathrm{char}(k) = 0$ the existence of this $\infty$-functor follows from the technology developed in \cite[\S 7.3.4]{Lurie_Higher_algebra}, and notably the equivalence
\[ \cO |_U \Mod \simeq \Sp(\CAlg(\cX_{/U})_{/\cO |_U}) , \]
which reduces the $\infty$-functoriality of the categories $\cO |_U \Mod$ to the $\infty$-functoriality of the comma categories $\cX_{/U}$.
When $\mathrm{char}(k) > 0$, the same strategy applies, but we have to use instead the identification
\[ \cO |_U \Mod \simeq \Sp(\Ab( \CAlg(\cX_{/U})_{/\cO |_U} )) , \]
proven in \cite[Corollary 8.3]{Porta_Yu_Representability}.

For every $U \in \cX$ we let $\Perf^{\mathrm{strict}}(U)$ denote the smallest full stable subcategory of $\cO |_U \Mod$ closed under retracts and containing $\cO |_U$.
Restriction along morphisms $V \to U$ in $\cX$ preserves strict perfect complexes.
So the assignment $U \mapsto \Perf^{\mathrm{strict}}(U)$ can be promoted to a sub-presheaf of $\cO \Mod$.
We let $\Perf_{\cX, \cO}$ denote its sheafification, computed in the $\infty$-category $\Cat_\infty^{\mathrm{st}, \mathrm{idem}}$ of idempotent complete stable $\infty$-categories.
It is straightforward to observe that the symmetric monoidal structure on $\cO |_U \Mod$ induces a symmetric monoidal structure on $\Perf(U)$, and so we can actually promote $\Perf_{\cX, \cO}$ to a sheaf with values in idempotent complete stably symmetric monoidal $\infty$-categories.\\

When $\cX = \dSt_k$ is the $\infty$-topos of derived stacks and $\cO$ is the global section functor, we denote $\Perf_{\cX, \cO}$ simply by $\Perf_k$.
According to our general convention, we denote by $\bfPerf_k$ the associated $\cS$-valued stack, determined by the relation
\[ \bfPerf_k(X) \coloneqq \Perf_k(X)^\simeq . \]
It coincides with the usual stack of perfect complexes (see \cite{Toen_Vaquie_Moduli_of_objects}).
Let now $\cX = \dAnSt_k$ be the $\infty$-topos of derived analytic stacks.
The functor
\[ \dAfd_k\op \longrightarrow \Cat_\infty^{\mathrm{st}, \otimes} \]
sending $U = (\cU, \cO_U)$ to $\cO_U\alg \Mod$ extends to a functor $\cO \Mod \colon \dAnSt_k\op \to \Cat_\infty^{\mathrm{st}, \otimes}$.
In this case, we simply denote by $\mathrm{AnMod}_k$ the stack $\cO \Mod$ and by $\AnPerf_k$ the stack $\Perf_{\dAnSt_k, \cO}$.
When $X \in \dAnSt_k$, we set
\[ \cO_X \Mod \coloneqq \mathrm{AnMod}_k(X) \quad , \quad \Perf(X) \coloneqq \AnPerf_k(X) . \]
When $X$ is a derived Stein (resp.\ derived $k$-affinoid) space, we can also identify $X$ with a $\cTank$-structured topos.
In this case, the above notations are compatible with \cite[\S 2.1]{DAG-VIII}.
Notice that $\cO_X\Mod$ has a canonical $t$-structure, where connective objects are defined locally:

\begin{lem}\label{lem:tstructure_mx}
	Let $X \in \dAnSt_k$ be a derived analytic stack.
	Then the the stable $\infty$-category $\cO_X \Mod$ has a $t$-structure where an object $\cF$ is connective if and only if for every morphism $f \colon U \to X$ with $U \in \dAfdk$, the pullback $f^* \cF \in \cO_U\alg \Mod$ is connective.
\end{lem}

\begin{proof}
	This is clear for $X \in \dAfdk$ using the $t$-structure on $\cO_X\alg$-modules for an $\infty$-topos $\cX$ \cite[Proposition 2.1.3]{DAG-VIII}.
	For general $X$ we need to define the $t$-structure on a limit of categories. 
	We may define connective objects locally and then use the two parts of \cite[Proposition 1.4.4.11]{Lurie_Higher_algebra} to extend to a uniquely defined $t$-structure. 
	\end{proof}

\subsection{Analytification of $\bfPerf_k$} \label{subsec:analytification_Perf}

We define the analytification $\bfPerf_k\an$ as in \cref{sec:analytification_geometric_stacks} by first restricting to $\dAff_k^{\mathrm{afp}}$ and then performing left Kan extension along the analytification $(-)\an \colon \dAff_k^{\mathrm{afp}} \to \dAfd_k$.
In a similar way, we define the analytification $\Perf_k\an$ of the $\Cat_\infty$-valued stack $\Perf_k$.

\begin{rem}
	The maximal $\infty$-groupoid functor $(-)^\simeq \colon \Cat_\infty \to \cS$ does not commute with colimits in general.
	Therefore, for $U \in \dAfd_k$, we can no longer identify $\bfPerf_k\an(U)$ with $(\Perf_k\an(U))^\simeq$.
\end{rem}

Since $\Perf_k\an$ is defined by left Kan extension, in order to give a morphism
\begin{equation} \label{eq:analytification_perfect_complexes}
	\varepsilon^* \colon \Perf_k\an \longrightarrow \AnPerf_k ,\footnote{The notation suggests that this morphism is induced by pullback along a certain map $\varepsilon$. We will make this idea explicit below.}
\end{equation}
it is enough to produce a natural transformation
\[ \Perf_k \to \AnPerf_k \circ (-)\an . \]
If $X \in \dAff_k^{\mathrm{afp}}$ then $X\an \in \dAfd_k$ and therefore the underlying $\cTetk$-structured topos $(X\an)\alg$ is well defined.
Pulling back along the natural map
\[ \theta_X \colon (X\an)\alg \longrightarrow X \]
provides an analytification functor that respects perfect complexes:
\[ (-)\an_X \colon \Perf(X) \to \Perf((X\an)\alg) . \]
Observe that, according to our definitions, $\Perf((X\an)\alg) \simeq \Perf(X\an)$.
The same construction also provides a morphism $\bfPerf_k\an \to \bfAnPerf_k$, which we still denote $\varepsilon^*$.\\

Our goal is to prove that $\varepsilon^* \colon \bfPerf_k\an \to \bfAnPerf_k$ is an equivalence, and we will do so by verifying the hypotheses of \cref{prop:general_strategy_analytification}.
Fix therefore $X \in \dAfd_k$.
The first step is to make the construction of the morphism
\[ \varepsilon_X^* \colon G^p_X(\bfPerf_k) \longrightarrow F^s_X(\bfAnPerf_k) \]
explicit.
Let $U \in X\et$ and define
\[ A_U \coloneqq \Gamma(U; \cO_U\alg) . \]
Then
\[ G^p_X(\bfPerf_k)(U) \simeq \Perf(A_U)^\simeq , \]
while
\[ F^s_X(\bfAnPerf_k)(U) \simeq \Perf(U)^\simeq . \]
The universal property of $\Spec$ proven in \cite[Theorem 2.2.12]{DAG-V} provides us with a canonical morphism in $\RTop(\cTet(\C))$
\[ \varepsilon_U \colon (\cU, \cO_U\alg) \longrightarrow \Spec(A_U) . \]
Pullback along $\varepsilon_U$ provides a functor
\[ \widetilde{\varepsilon}_U^* \colon \cO_{A_U} \Mod \longrightarrow \cO_U \Mod . \]
This is simply the pullback along $\varepsilon_U$, but we reserve the notation $\varepsilon_U^*$ for the restriction
\[ \begin{tikzcd}
	\varepsilon_U^* \colon A_U \Mod \simeq \QCoh(\Spec(A_U)) \arrow[hook]{r}{i_U} & \cO_{A_U} \Mod \arrow{r}{\widetilde{\varepsilon}_U^*} & \cO_U \Mod .
\end{tikzcd} \]
The functor $\varepsilon_U^*$ preserves perfect complexes and therefore further restricts to
\[ \varepsilon_U^* \colon \Perf(A_U) \longrightarrow \Perf(U) , \]
which coincide with the functor induced by $\varepsilon^* \colon \Perf_k\an \to \AnPerf_k$.\\

Notice that we also have a functor in the opposite direction
\[ \Gamma(U;-) \colon \cO_U \Mod \longrightarrow A_U \Mod . \]
Observe that $\varepsilon_U^*$ and $\Gamma(U;-)$ are not adjoint to each other.
However, the inclusion $i_U \colon A_U \Mod \hookrightarrow \cO_{A_U} \Mod$ admits a left adjoint
\[ L_U \colon \cO_{A_U} \Mod \longrightarrow A_U \Mod , \]
and, similarly the functor $\widetilde{\varepsilon}_U^* \colon \cO_{A_U} \Mod \to \cO_U \Mod$ admits a right adjoint
\[ \widetilde{\varepsilon}_{U*} \colon \cO_U \Mod \longrightarrow \cO_{A_U} \Mod . \]
Then $\Gamma(U; -)$ is naturally equivalent to the functor $L_U \circ \widetilde{\varepsilon}_{U*}$, and we have two canonical zig-zags of natural transformations
\begin{equation} \label{eq:zig_zag_natural_transformations}
	\begin{tikzcd}
		{} & L_U \circ i_U \arrow{dr} \arrow{dl} \\
		\Gamma(U;-) \circ \varepsilon_U^* & & \mathrm{Id}_{A_U \Mod}
	\end{tikzcd} \quad , \quad \begin{tikzcd}
		{} & \widetilde{\varepsilon}_U^* \circ \widetilde{\varepsilon}_{U*} \arrow{dr} \arrow{dl} \\
		\varepsilon_U^* \circ \Gamma(U;-) & & \mathrm{Id}_{\cO_U \Mod}
	\end{tikzcd}
\end{equation}
Notice that $L_U \circ i_U \to \mathrm{Id}_{A_U \Mod}$ is always an equivalence.
In particular, we obtain a well defined natural transformation
\[ \delta \colon \mathrm{Id}_{A_V \Mod} \longrightarrow \Gamma(U;-) \circ \varepsilon_U^* . \]
We now summarize the most basic properties of the functors we introduced so far:

\begin{prop} \label{prop:basic_properties_global_sections}
	Let $U \in \dAfd_k$ be a derived $k$-affinoid (resp.\ Stein) space.
	Then:
	\begin{enumerate}
		\item The lax symmetric monoidal functor
		\[ \Gamma(U;-) \colon \cO_U \Mod \longrightarrow A_U \Mod \]
		is $t$-exact and symmetric monoidal when restricted to the full subcategory $\Coh(U)$ of unbounded complexes with coherent cohomology.
		
		\item The functor
		\[ \varepsilon_U^* \colon A_U \Mod \longrightarrow \cO_U \Mod \]
		is $t$-exact and monoidal, and its restriction to $\Coh(A_U)$ factors through $\Coh(U)$.
		
		\item The restriction of the composition
		\[ \Gamma(U;-) \circ \varepsilon_U^* \colon A_U \Mod \longrightarrow A_U \Mod \]
		to $\Coh(A_U)$ factors through $\Coh(A_U)$.
		
		\item The natural transformation
		\[ \delta \colon \mathrm{Id}_{A_U \Mod} \longrightarrow \Gamma(U;-) \circ \varepsilon_U^* \]
		is an equivalence when restricted to $\Coh(A_U)$.
		
		\item The functor $\varepsilon_U^* \colon A_U \Mod \to \cO_U \Mod$ commutes with filtered colimits and it is conservative.
	\end{enumerate}
\end{prop}

\begin{proof}
	In the non-archimedean case, points (1) -- (4) have been proved in \cite[Theorems 3.1 and 3.4]{Porta_Yu_Mapping}.
	In the \canal setting, point (1) follows from Cartan's theorem B.
	The $t$-exactness part of point (2) follows from the flatness result proven in \cref{{lem:flatness_global_sections_nested_compact_Stein}}.
	At this point, we are left to check that $\varepsilon_U^*$ takes $\Cohh(A_U)$ to $\Cohh(U)$.
	This follows immediately from the fact that $A_U$ is taken to $\cO_U$ and the fact that every object in $\Cohh(A_U)$ admits a finite presentation.
	Points (3) follows immediately from point (4).
	Point (4) is a consequence of Cartan's theorem B and point (2).
	
	We now prove point (5).
	Observe that the inclusion
	\[ \begin{tikzcd}[column sep = small]
		A_U \Modh \arrow[hook]{r} & \cO_{A_U} \Modh
	\end{tikzcd} \]
	commutes with filtered colimits.
	Since $A_U \Mod \hookrightarrow \cO_{A_U} \Mod$ is $t$-exact and fully faithful, it follows that it also commutes with filtered colimits.
	Therefore, $\varepsilon_U^*$ has the same property.
	We now prove conservativity.
	Since $\varepsilon_U^*$ is an exact functor, it is enough to prove that if $\cF$ is such that $\varepsilon_U^*(\cF) \simeq 0$, then $\cF \simeq 0$.
	Since $\varepsilon_U^*$ is $t$-exact and the $t$-structures on $A_U \Mod$ and $\cO_U \Mod$ are complete, we can reduce ourselves to the case where $\cF \in A_U \Modh$.
	In this case, we can write
	\[ \cF \simeq \bigcup_\alpha \cF_\alpha , \]
	where the union ranges over all finitely generated $A_U$-submodules of $\cF$.
	Since $\varepsilon_U^*$ is $t$-exact, we see that $\varepsilon_U^*(\cF_\alpha)$ is a submodule of $\varepsilon_U^*(\cF)$.
	This implies that $\varepsilon_U^*(\cF_\alpha) = 0$.
	Since $\cF_\alpha \in \Cohh(A_U)$, point (4) implies that $\cF_\alpha = 0$ for every $\alpha$.
	In particular, $\cF = 0$, whence the conclusion.
\end{proof}

In the non-archimedean case, we can strengthen the above result:

\begin{lem}[{cf.\ \cite[Theorem 3.4]{Porta_Yu_Mapping}}] \label{lem:perfect_complexes_affinoid}
	Let $U$ be a derived $k$-affinoid space and let $A_U \coloneqq \Gamma(U; \cO_U\alg)$.
	Then the functors $\varepsilon_U^*$ and $\Gamma(U;-)$ realize an equivalence of stable $\infty$-categories
	\[ \Coh(U) \simeq \Coh(A_U) , \]
	which furthermore restricts to an equivalence
	\[ \Perf(U) \simeq \Perf(A_U) . \]
	In particular, there is an equivalence $\Perf(U)^\simeq \simeq \Perf(A_{U})^\simeq$.
\end{lem}

\begin{proof}
	Theorem 3.4 in \cite{Porta_Yu_Mapping} shows that the global section functor induces a $t$-exact equivalence of stable $\infty$-categories
	\[ \Coh(U) \simeq \Coh(A_U) . \]
	On the other hand, we know $\Perf(A_U)$ coincides with the smallest full stable subcategory of $\Coh^-(A_U)$ closed under retracts and containing $A_U$.
	As $A_U$ is mapped to $\cO_U$ under the above equivalence, we see that $\Perf(A_U)$ is mapped into $\Perf(U)$.
	It is therefore sufficient to prove that the global section functor takes $\Perf(U)$ to $\Perf(A_U)$.
	Fix $\cF \in \Perf(U)$ and let $\{U_i\}$ be a finite derived affinoid cover of $U$ so that $\cF |_{U_i}$ belongs to $\Perf^{\mathrm{strict}}(U_i)$.
	Let $A_i \coloneqq \Gamma(U_i; \cO_{U_i}\alg)$.
	In this case, we immediately see that $\Gamma(U_i; \cF|_{U_i}) \in \Perf(A_i)$.
	In particular, $\Gamma(U_i; \cF|_{U_i})$ has finite tor-amplitude.
	As the maps $A \to A_i$ are faithfully flat and
	\[ \Gamma(U_i; \cF|_{U_i}) \simeq \Gamma(X; \cF) \otimes_A A_i , \]
	we conclude that $\Gamma(X; \cF)$ has also finite tor-amplitude by \cite[Proposition 2.8.4.2(5)]{Lurie_SAG}.
	In particular, it belongs to $\Perf(A)$.
\end{proof}

This verifies the hypotheses of \cref{prop:general_strategy_analytification} for the map $\varepsilon^* \colon \bfPerf_k\an \to \bfAnPerf_k$ in the non-archimedean setting.
In the \canal case, the analogue of \cref{lem:perfect_complexes_affinoid} is simply false: indeed, the restriction of the functor $\Gamma(U;-)$ to $\Perf(U)$ does not factor through $\Perf(A_U)$.
However, \cref{prop:general_strategy_analytification} asks for a different statement: we have to check that for every compact Stein subset $K$ of $X$, the map $\varepsilon^* \colon \bfPerf_k\an \to \bfAnPerf_k$ induces an equivalence
\begin{equation} \label{eq:analytification_stack_perfect_complexes}
	\colim_{K \subset U \subset X} \Perf(A_U)^\simeq \longrightarrow \colim_{K \subset U \subset X} \Perf(U)^\simeq ,
\end{equation}
where the colimits run through the open Stein neighbourhoods $U$ of $K$ inside $X$.
In order to prove that \eqref{eq:analytification_stack_perfect_complexes} is an equivalence, it is easier (and more natural) to work with the stacks with values in $\Cat_\infty$.
We will prove below that $\varepsilon^* \colon \Perf_\C\an \to \AnPerf_\C$ induces an equivalence
\[ \colim_{K \subset U \subset X} \Perf( A_U ) \longrightarrow \colim_{K \subset U \subset X} \Perf( U ) . \]
However, since the maximal $\infty$-groupoid functor $(-)^\simeq$ does not commute with colimits, it is not straightforward to deduce that \eqref{eq:analytification_stack_perfect_complexes} is an equivalence from the above statement.
To circumvent this problem, we prove the following stronger statement:

\begin{thm} \label{thm:perfect_complexes_compact_Stein}
	Let $X \in \dAn_\C$ and let $K$ be a compact Stein subset of $X$.
	The functors
	\[ \varepsilon_U^* \colon \Perf( A_U ) \longrightarrow \Perf( U ) \]
	induce a morphism
	\[ \varepsilon_{(K)}^* \colon \fcolim_{K \subset U \subset X} \Perf( A_U ) \longrightarrow \fcolim_{K \subset U \subset X} \Perf( U ) \]
	in $\Ind(\Cat_\infty^{\mathrm{st}, \otimes})$, which is furthermore an equivalence.
\end{thm}

The proof of this theorem is technical and it will occupy the rest of this section.
Before delving into the details, let us record its main consequence:

\begin{prop} \label{prop:analytification_perfect_complexes_absolute}
	Let $k$ be either the field of complex numbers or a non-archimedean field equipped with a non-trivial valuation.
	The natural morphism $\varepsilon^* \colon \bfPerf_k\an \to \bfAnPerf_k$ is an equivalence of derived analytic stacks.
	In particular, $\bfAnPerf_k$ is a locally geometric derived analytic stack.
\end{prop}

\begin{proof}
	We know that $\bfPerf_k$ is a locally geometric stack.
	It is therefore enough to check that the hypotheses of \cref{prop:general_strategy_analytification} are satisfied.
	Fix a derived Stein (resp.\ $k$-affinoid) space $U$ and let $V \subset U$ be an open Stein subspace (resp.\ $k$-affinoid domain embedding).
	Unravelling the definitions, we see that
	\[ G^p_U( \bfPerf_k )(V) \simeq \Perf(A_V)^\simeq \quad , \quad F^s_U( \bfAnPerf_k )(V) \simeq \Perf( V )^\simeq . \]
	In the non-archimedean case, the conclusion therefore follows from \cref{lem:perfect_complexes_affinoid}.
	In the \canal case, the equivalence of ind-objects provided by \cref{thm:perfect_complexes_compact_Stein} induces an equivalence
	\[ \fcolim_{K \subset V \subset U} \Perf(A_V)^\simeq \longrightarrow \fcolim_{K \subset V \subset U} \Perf(V)^\simeq \]
	in $\Ind(\cS)$.
	By realizing this equivalence of ind-objects, we see that the hypotheses of \cref{prop:general_strategy_analytification} are satisfied.
	The second statement follows at once because the analytification functor preserves locally geometric stacks.
\end{proof}

\subsection{Proof of \cref{thm:perfect_complexes_compact_Stein}}

We now turn to the proof of \cref{thm:perfect_complexes_compact_Stein}.
As in \cref{lem:perfect_complexes_affinoid}, we will deduce the statement from the analogous statement concerning unbounded complexes with coherent cohomology.
As we already remarked, the difference from the non-archimedean setting is that the composition
\[ \begin{tikzcd}
	\Coh(U) \arrow{r} & \cO_U \Mod \arrow{r}{\Gamma(U;-)} & A_U \Mod
\end{tikzcd} \]
does not factor through $\Coh(A_U)$.
We therefore lack a candidate for the inverse of $\varepsilon_U^*$.
When working with ind-objects, however, the situation improves thanks to the following couple of lemmas:

\begin{lem} \label{lem:global_section_relatively_compact}
	Let $U \in \dStn_\C$ and let $V \Subset U$ be a relatively compact Stein subset (see \cref{def:relatively_compact_Stein}).
	Then the composition
	\[ \begin{tikzcd}
		\Coh(U) \arrow{r}{(-)|_V} & \Coh(V) \arrow{r}{\Gamma(V;-)} & A_V \Mod
	\end{tikzcd} \]
	factors, as a symmetric monoidal functor, through $\Coh(A_V)$.
	Moreover, if $\cF \in \Perf(U)$, then $\Gamma(V; \cF|_V) \in \Perf(A_V)$.
\end{lem}

\begin{proof}
	We know from \cref{prop:basic_properties_global_sections} that $\Gamma(V;-)$ is $t$-exact and monoidal.
	Therefore it is enough to check that when $\cF \in \Cohh(U)$, then $\Gamma(V;\cF|_V) \in \Cohh(A_V)$.
	This follows at once from Cartan's Theorem B and \cite[Lemma 8.12]{Porta_Yu_Higher_analytic_stacks_2014}.
	
	Suppose now that $\cF \in \Perf(U)$.
	We have to prove that $\Gamma(V;\cF|_V)$ has finite tor-amplitude.
	If $\cF \in \Perf^{\mathrm{strict}}(U)$ then $\cF |_V \in \Perf^{\mathrm{strict}}(V)$ and therefore $\Gamma(V; \cF|_V) \in \Perf^{\mathrm{strict}}(A_V)$.
	In general, we can find a cover $\{U_i\}_{i \in I}$ of $U$ such that $\cF |_{U_i}$ belongs to $\Perf^{\mathrm{strict}}(U_i)$.
	Choose now a cover $\{V_j\}_{j \in J}$ of $V$ satisfying the following properties:
	\begin{enumerate}
		\item for every $j \in J$ there exists $i \in I$ such that $V_j \subset U_i$;
		\item for every $j \in J$, the open $V_j$ is Stein and relatively compact inside $V$.
	\end{enumerate}
	It follows that each $\Gamma(V_j; \cF |_{V_j})$ is perfect over $A_{V_j}$.
	Since $V_j \Subset V \Subset U$, we can apply \cref{lem:flatness_global_sections_nested_compact_Stein} to deduce that the family of maps $\{A_V \to A_{V_j}\}$ is faithfully flat.
	Similarly, \cref{cor:coherent_sheaves_and_restriction_derived_Stein} shows that the natural morphism
	\[ \Gamma(V;\cF|_V) \otimes_{A_V} A_{V_j} \longrightarrow \Gamma(V_j ; \cF |_{V_j}) \]
	is an equivalence.
	At this point, \cite[Proposition 2.8.4.2(5)]{Lurie_SAG} implies that $\Gamma(V; \cF|_V)$ has finite tor-amplitude over $A_V$, completing the proof.
\end{proof}

This lemma shows that the functor $\Gamma(U;-) \colon \Coh(U) \to A_U \Mod$ induce a morphism in $\Ind(\Cat_\infty^{\mathrm{st}, \otimes})$
\[ \Gamma_{(K)} \colon \fcolim_{K \subset V \subset U} \Coh(U) \longrightarrow \fcolim_{K \subset V \subset U} \Coh(A_U) . \]
We now wish to prove that $\Gamma_{(K)}$ and $\varepsilon_{(K)}^*$ form an equivalence of ind-objects.
We need the following lemma.

\begin{lem} \label{lem:coherent_pushforward_relatively_compact}
	Let $U$ be a derived Stein space and let $W \Subset V \Subset U$ be two nested relatively compact Stein subsets.
	For any $\cF \in \Coh(U)$, the $\cO_{A_{W}}$-module $\widetilde{\varepsilon}_{W*}(\cF |_{W})$ is a coherent sheaf over $\Spec(A_W)$.
\end{lem}

\begin{proof}
	We already know from the previous lemma that the global sections of $\widetilde{\varepsilon}_{W*}(\cF |_{W})$ belong to $\Coh(A_W)$.
	It is therefore enough to prove that all its cohomologies are quasi-coherent.
	For this, it is enough to check that for every principal open $W'$ of $\Spec(A_W)$ the canonical map
	\[ \Gamma(W; \cF|_W) \otimes_{A_W} A_{W'} \longrightarrow (\widetilde{\varepsilon}_{W*}(\cF |_W))(W') \]
	is an equivalence.
	We immediately observe that
	\[ \widetilde{\varepsilon}_{W*}(\cF |_W)(W') = \cF( \varepsilon_W\inv(W') ) . \]
	It follows from the discussion in \cite[\S 1.4.4]{Grauert_Coherent_1984} and from the reconstruction theorem proved in \cite[\S IV.7.4]{Grauert_Theory_Stein_1979} that $\widetilde{W'} \coloneqq \varepsilon_W\inv(W')$ is itself a Stein space.
	Furthermore, we have $\widetilde{W'} \Subset V$.
	We can therefore apply \cref{cor:coherent_sheaves_and_restriction_derived_Stein} to the sequence of nested derived Stein $\widetilde{W'} \Subset V \Subset U$ to deduce that
	\[ \Gamma(W;\cF|_W) \otimes_{A_W} A_{\widetilde{W'}} \simeq \Gamma(V; \cF |_V) \otimes_{A_V} A_{\widetilde{W'}} . \]
	The proof is therefore complete.
\end{proof}

\begin{lem} \label{lem:mock_counit}
	Let $U$ be a derived Stein space and let $W \Subset V \Subset U$ be two nested relatively compact Stein subsets.
	Let $\rho_{U,W} \colon \cO_U \Mod \longrightarrow \cO_W \Mod$ be the restriction functor and let $j_U \colon \Coh(U) \hookrightarrow \cO_U \Mod$ be the natural inclusion.
	Then:
	\begin{enumerate}
		\item the natural transformation
		\[ \widetilde{\varepsilon}_{W*} \circ \rho_{U,W} \circ j_U \longrightarrow i_W \circ L_W \circ \widetilde{\varepsilon}_{W*} \circ \rho_{U,W} \circ j_U \]
		is an equivalence;
		\item the natural transformation
		\[ \vartheta \colon \varepsilon_W^* \circ \Gamma(W;-) \circ \rho_{U,W} \circ j_U \longrightarrow \rho_{U,W} \circ j_U \]
		induced by the zig-zag \eqref{eq:zig_zag_natural_transformations} and by the previous point is an equivalence.
	\end{enumerate}
\end{lem}

\begin{proof}
	Point (1) is a direct consequence of \cref{lem:coherent_pushforward_relatively_compact}, because $\widetilde{\varepsilon}_{W*} \circ \rho_{U,W} \circ j_U$ factors through $\Coh(A_W)$ and the unit of the adjunction $L_W \dashv i_W$ is an equivalence on the objects belonging to $A_W \Mod$.
	As for point (2), all the functors appearing are $t$-exact.
	It is therefore sufficient to check that $\theta$ is an equivalence when evaluated on objects in $\Cohh(U)$.
	This follows immediately from \cite[Lemma 8.11]{Porta_Yu_Higher_analytic_stacks_2014}.
\end{proof}

We are now ready to state and prove the key result:

\begin{thm} \label{thm:unbounded_coherent_pro_compact_Stein}
	Let $X \in \dAn_\C$ and let $K$ be a compact Stein subset of $X$.
	The morphism
	\begin{equation} \label{eq:coherent_compact_Stein}
		\varepsilon_{(K)}^* \colon \fcolim_{K \subset U \subset X} \Coh(A_U) \longrightarrow \fcolim_{K \subset U \subset X} \Coh(U)
	\end{equation}
	is an equivalence in $\Ind(\Cat_\infty^{\mathrm{st}, \otimes})$, whose inverse is given by $\Gamma_{(K)}$.
\end{thm}

\begin{proof}
	\Cref{prop:basic_properties_global_sections} implies that the natural transformation $\delta \colon \mathrm{Id}_{A_U \Mod} \longrightarrow \Gamma(U;-) \circ \varepsilon_U^*$ is an equivalence when evaluated on objects in $\Coh(A_U)$, for every open Stein neighbourhood $U$ of $K$ in $X$.
	This shows that $\Gamma_{(K)} \circ \varepsilon_{(K)}^*$ is equivalent to the identity of the left hand side of \eqref{eq:coherent_compact_Stein}.
	
	For the other direction, the natural transformation $\theta$ we constructed in \cref{lem:mock_counit} provides a path between morphisms in $\Ind(\Cat_\infty^{\mathrm{st}, \otimes})$
	\[ \varepsilon_{(K)}^* \circ \Gamma_{(K)} \longrightarrow  \mathrm{Id}_{\Coh((K)_U)} , \]
	where $\Coh((K)_U)$ denotes the right hand side of \eqref{eq:coherent_compact_Stein}.
	Moreover, \cref{lem:mock_counit} shows that this morphism is invertible, thereby completing the proof.
\end{proof}

From here, deducing \cref{thm:perfect_complexes_compact_Stein} is straightforward:

\begin{proof}[Proof of \cref{thm:perfect_complexes_compact_Stein}]
	It is enough to check that the functors $\varepsilon_{(K)}^*$ and $\Gamma_{(K)}$ restrict to morphism of ind-objects
	\[ \varepsilon_{(K)}^* \colon \fcolim_{K \subset U \subset X} \Perf(A_U) \longrightarrow \fcolim_{K \subset U \subset X} \Perf(U) \quad , \quad \Gamma_{(K)} \colon \fcolim_{K \subset U \subset X} \Perf(U) \longrightarrow \fcolim_{K \subset U \subset X} \Perf(A_U) . \]
	In the case of $\varepsilon_{(K)}^*$ this follows directly from the construction, while for $\Gamma_{(K)}$ this has already been in checked in \cref{lem:global_section_relatively_compact}.
\end{proof}

By realizing the equivalences of ind-objects obtained in Theorems \ref{thm:unbounded_coherent_pro_compact_Stein} and \ref{thm:perfect_complexes_compact_Stein} we obtain the following weaker result, which is closer in spirit to \cite[Proposition 11.9.2]{Taylor_Several_complex}.
Before stating it, let us fix a couple of notations: first of all, we write
\[ \varepsilon_K^* \colon \colim_{K \subset U \subset X} \Perf(A_U) \longrightarrow \colim_{K \subset U \subset X} \Perf(U) \]
and
\[ \Gamma_K \colon \colim_{K \subset U \subset X} \Perf(U) \longrightarrow \colim_{K \subset U \subset X} \Perf(A_U) \]
for the realizations of $\varepsilon_{(K)}^*$ and $\Gamma_{(K)}$.
We also set
\[ A_K \coloneqq \colim_{K \subset U \subset X} A_U . \]
Then we have:

\begin{cor} \label{cor:perfect_complexes_compact_Stein}
	Let $X \in \dAn_\C$ and let $K$ be a compact Stein inside $X$.
	Then there is a canonical equivalence
	\[ \Perf(A_K) \simeq \colim_{K \subset U \subset X} \Perf(U) . \]
\end{cor}

\begin{proof}
	By realizing the equivalence of \cref{thm:perfect_complexes_compact_Stein} we see that the functors $\Gamma_K$ and $\varepsilon_K^*$ induce an equivalence
	\[ \colim_{K \subset U} \Perf(A_U) \simeq \colim_{K \subset U \subset X} \Perf(U) . \]
	On the other hand, it is proven in \cite[4.5.1.8]{Lurie_SAG} that the construction $R \mapsto \Perf(R)$ commutes with filtered colimits.
	Therefore, we also have an equivalence
	\[ \Perf(A_K) \simeq \colim_{K \subset U \subset X} \Perf(A_U) . \]
	The conclusion follows.
\end{proof}

Let us give another application that stems from the combination of \cref{cor:perfect_complexes_compact_Stein} and \cref{thm:sheaves_compact_subsets}.

\begin{prop}
	Let $X$ be a derived \canal space.
	Then the subcategory $\Perf(X) \subset \Coh^-(X)$ coincides exactly with the subcategory of dualizable objects.
\end{prop}

\begin{proof}
	Denote by $\AnPerf_X$ and $\mathrm{AnCoh}^-_X$ the restrictions of $\AnPerf$ and $\mathrm{AnCoh}^-$ to $X^{\mathrm{top}}$.
	Let $\cF \in \Coh^-(X)$ be a dualizable object.
	Then we observe that for every compact Stein subset $K \subset X$, $\Gamma_{(K)}(\cF)$ is a dualizable object in $\Coh^-(A_K)$, and hence it belongs to $\Perf(A_K)$.
	We now use the equivalence
	\[ \Perf((K)_X) \simeq \Perf(A_K) \]
	provided by \cref{cor:perfect_complexes_compact_Stein} and we observe that under the equivalence 
	\[ \Sh(X^{\mathrm{top}}; \Cat_\infty^{\mathrm{st}, \mathrm{idem}}) \simeq \Sh_\cK(X^{\mathrm{top}}, \Cat_\infty^{\mathrm{st}, \mathrm{idem}}) . \]
	provided by \cref{thm:sheaves_compact_subsets}, $\AnPerf_X$ corresponds to the stack sending a compact Stein subset $K \subset X$ to $\Perf((K)_X)$.
	Therefore, $\cF$ defines a global section of $\AnPerf_X$, i.e.\ $\cF \in \Perf(X)$.
	
	Suppose vice-versa that $\cF \in \Perf(X)$.
	Then the functor
	\[ \cF \otimes_{\cO_X} - \colon \cO_X \Mod \longrightarrow \cO_X \Mod \]
	is left adjoint to
	\[ \cHom_{\cO_X}(\cF, -) \colon \cO_X \Mod \longrightarrow \cO_X \Mod . \]
	In particular, we have unit and counit transformations
	\[ \cG \longrightarrow \cF \otimes_{\cO_X} \cHom_{\cO_X}(\cF, \cG) \quad, \quad \cHom_{\cO_X}(\cF, \cF \otimes_{\cO_X} \cG) \longrightarrow \cG , \]
	satisfying the triangular identities.
	Furthermore, as $\cF$ is a perfect complex, the canonical morphism
	\[ \cHom_{\cO_X}(\cF, \cO_X) \otimes_{\cO_X} \cF \otimes_{\cO_X} \cG \longrightarrow \cHom_{\cO_X}(\cF, \cF \otimes_{\cO_X} \cG) \]
	is an equivalence.
	Taking $\cG = \cO_X$, we obtain the evaluation and coevaluation morphisms for $\cF$ and $\cHom_{\cO_X}(\cF, \cO_X)$.
	All that is left to check is therefore that $\cHom_{\cO_X}(\cF, \cO_X)$ is perfect.
	Combining once again Theorems \ref{cor:perfect_complexes_compact_Stein} and \ref{thm:sheaves_compact_subsets}, we are reduced to check that for every compact Stein subset $K \subset U$ we have
	\[ \Gamma_{(K)}( \cHom_{\cO_X}(\cF, \cO_X) ) \in \Perf(A_K) . \] 
	However, for every open Stein neighbourhood $V$ of $K$ inside $X$, one has
	\[ \Gamma(V; \cHom_{\cO_V}(\cF |_V, \cO_V)) \simeq \Hom_{A_V}( \Gamma(V; \cF |_V), A_V ) . \]
	Hence
	\[ \Gamma_{(K)}( \cHom_{\cO_X}(\cF, \cO_X) ) \simeq \Hom_{A_K}(\Gamma_{(K)}(\cF), A_K) \in \Perf(A_K) . \]
	This completes the proof.
\end{proof}

\section{GAGA properties} \label{sec:GAGA}

In this section we discuss several variations of the GAGA property on derived stacks locally almost of finite presentation.
We verify that they are satisfied in a number of different examples, proving a relative version of \cref{prop:analytification_perfect_complexes_absolute}.

\subsection{(Universally) GAGA stacks}

We start by generalizing the natural transformation $\varepsilon^* \colon \Perf_k\an \to \AnPerf_k$ introduced in the previous section.
Fix $X \in \dAff_k^{\mathrm{afp}}$ and $U \in \dAfd_k$.
As usual, we set $A_U \coloneqq \Gamma(U; \cO_U\alg)$.
The counit of the analytification adjunction $\varepsilon_X \colon (X\an)\alg \to X$ induces a well defined morphism
\[ \varepsilon_{X,U} \colon (X\an \times U)\alg \longrightarrow (X\an)\alg \times \Spec(A_U) \longrightarrow X \times \Spec(A_U) . \]
The pullback functor along $\varepsilon_{X,U}$ induces a well defined symmetric monoidal functor
\[ \varepsilon_{X,U}^* \colon \QCoh( X \times \Spec(A_U) ) \longrightarrow \cO_{X\an \times U} \Mod , \]
which is natural in both $X$ and $U$.
Moreover, it restricts to a symmetric monoidal functor
\[ \varepsilon_{X,U}^* \colon \Perf( X \times \Spec(A_U) ) \longrightarrow \Perf(X\an \times U) . \]
Naturality in $X$ allows to extend this map by colimits.
Therefore, we obtain a commutative square
\[ \begin{tikzcd}
	\Perf(X \times \Spec(A_U)) \arrow{r}{\varepsilon_{X,U}^*} \arrow[hook]{d} & \Perf(X\an \times U) \arrow[hook]{d} \\
	\QCoh(X \times \Spec(A_U)) \arrow{r}{\varepsilon_{X,U}^*} & \cO_{X\an \times U} \Mod
\end{tikzcd} \]
for every $X \in \dSt_k^{\mathrm{afp}}$.
When $U = \Sp(k)$, we write $\varepsilon_X^*$ instead of $\varepsilon_{X,\Sp(k)}^*$.

\begin{defin} \label{defin:strong_GAGA}
	Let $X \in \dSt_k^{\mathrm{afp}}$ be a derived stack locally almost of finite presentation.
	We say that:
	\begin{enumerate}
		\item \emph{$X$ satisfies the GAGA property} if the functor $\varepsilon_X^* \colon \mathrm{QCoh}(X) \to \cO_{X\an} \Mod$ is conservative and $t$-exact and the functor
		\[ \varepsilon_X^* \colon \Perf(X) \longrightarrow \Perf(X\an) \]
		is an equivalence of $\infty$-categories.
		
		\item Let $k$ be a non-archimedean field and let $U \in \dAfd_k$ be a derived $k$-affinoid space.
		We say that \emph{$X$ satisfies the GAGA property relative to $U$} if the functor
		\[ \varepsilon_{X,U}^* \colon \QCoh(\Spec(A_U) \times X) \longrightarrow \cO_{U \times X\an} \Mod \]
		is conservative and $t$-exact and the functor
		\[ \varepsilon_{X,U}^* \colon \Perf(\Spec(A_U) \times X) \longrightarrow \Perf(U \times X\an) \]
		is an equivalence.
		
		\item Let $U \in \dStn_\C$ be a derived Stein space and let $K \subset U$ be a compact Stein subset.
		We say that \emph{$X$ satisfies the pro-GAGA property relative to $(K)_U$} if for every Stein neighbourhood $V$ of $K$ inside $U$ the map
		\[ \varepsilon_{X,V}^* \colon \QCoh(\Spec(A_V) \times X) \longrightarrow \cO_{V \times X\an} \Mod \]
		is conservative and $t$-exact and the morphism
		\[ \varepsilon_{X,(K)}^* \colon \fcolim_{K \subset V \subset U} \Perf(\Spec(A_V) \times X) \longrightarrow \fcolim_{K \subset V \subset U} \Perf(V \times X\an) . \]
		induced by the functors $\varepsilon_{X,V}^*$ is an equivalence in $\Ind(\Cat_\infty^{\mathrm{st}, \otimes})$.
		We say that \emph{$X$ satisfies the GAGA property relative to $U$} if it satisfies the pro-GAGA property relative to $(K)_U$ for every compact Stein subset $K \subset U$.
			
		\item \emph{$X$ satisfies the universal GAGA property} if it satisfies the GAGA property relative to $U$ for every $U \in \dAfd_k$.
	\end{enumerate}
\end{defin}

Fix $X \in \dSt_k^{\mathrm{afp}}$.
For every $S \in \dAff_k^{\mathrm{afp}}$ the analytification functor
\[ \Perf(X \times S) \longrightarrow \Perf(X\an \times S\an) \]
induces a natural transformation of $\Cat_\infty$-valued stacks
\[ \bfMap(X, \Perf_k) \longrightarrow \bfAnMap(X\an, \AnPerf_k) \circ (-)\an , \]
which restricts to a natural transformation
\[ \bfMap(X, \bfPerf_k) \longrightarrow \bfAnMap(X\an, \bfAnPerf_k) \circ (-)\an . \]
By adjunction, this morphism determines a map
\[ \mu_X \colon \colon \bfMap(X, \bfPerf_k)\an \longrightarrow \bfAnMap(X\an, \bfAnPerf_k) . \]
The universal GAGA property enables us to check that $\mu_X$ is an equivalence:

\begin{prop} \label{prop:universal_GAGA_relative_stack_perfect_complexes}
	Let $X \in \dSt_k^{\mathrm{afp}}$ be a derived stack locally almost of finite presentation.
	Suppose that:
	\begin{enumerate}
		\item the mapping stack $\bfMap(X, \bfPerf_k)$ is locally geometric;
		\item the stack $X$ satisfies the universal GAGA property.
	\end{enumerate}
	Then the canonical morphism
	\[ \mu_X \colon \bfMap(X, \bfPerf_k)\an \longrightarrow \bfAnMap(X, \bfAnPerf_k) \]
	is an equivalence.
\end{prop}

\begin{proof}
	We apply \cref{prop:general_strategy_analytification}.
	Notice that for $U \in \dAfd_k$ and $V \in U\et$, we have
	\[ G^p_U(\bfMap(X, \bfPerf_k))(V) \simeq \Perf(X \times \Spec(A_V))^\simeq \quad , \quad F^s_U(\bfAnMap(X, \AnPerf_k))(V) \simeq \Perf(X\an \times V)^\simeq . \]
	In the non-archimedean case, the conclusion therefore follows directly from the assumption on $X$.
	In the \canal case we have to check that for every compact Stein subset $K \subset U$ the natural map
	\[ \colim_{K \subset V \subset U} \Perf(X \times \Spec(A_V))^\simeq \longrightarrow \colim_{K \subset V \subset U} \Perf(X\an \times V)^\simeq \]
	is an equivalence.
	Since $X$ satisfies the universal GAGA property, the natural map
	\[ \fcolim_{K \subset V \subset U} \Perf(X \times \Spec(A_V)) \longrightarrow \fcolim_{K \subset V \subset U} \Perf(X\an \times V) \]
	is an equivalence in $\Ind(\Cat_\infty^{\mathrm{st}, \otimes})$.
	The conclusion therefore follows by applying the maximal $\infty$-groupoid functor $(-)^\simeq$ and then realizing the equivalence in $\Ind(\cS)$.
\end{proof}

The following example is of course expected:

\begin{eg} \label{eg:Artin_strong_GAGA}
	A proper derived geometric stack locally almost of finite presentation over $\mathbb C$ satisfies the GAGA property. Indeed, it follows from \cite[Theorem 7.2]{Porta_DCAGI} that the analytification functor induces an equivalence $\Perf(X) \simeq \Perf(X\an)$.
	The argument given in loc.\ cit.\ is an extension to the derived setting of the analogous statement for underived stacks, which has been proven in \cite{Porta_Yu_Higher_analytic_stacks_2014} in both the \canal and the non-archimedean setting.
	The same extension argument works in the non-archimedean case, which shows that $X$ satisfies the GAGA property also in this situation.
	Furthermore, the map $\QCoh( \Spec(A_U) \times X ) \longrightarrow  \cO_{U\times X\an} \Mod$ is conservative and $t$-exact: this easily follows by choosing a smooth hypercover of $X$ by derived affine schemes, and observing that $t$-exactness and conservativeness can be checked locally with respect to this hypercover.
	In the affine case, the result follows from the flatness of the natural map $(X\an)\alg \to X$, see \cite[Corollary 5.15]{Porta_DCAGI} in the \canal case and \cite[Proposition 4.17]{Porta_Yu_Representability} in the \kanal case.
	
	This example covers a great variety of situations.
	Indeed, the following are special cases of proper derived geometric stacks locally almost of finite presentation over $k$:
	\begin{enumerate}
		\item Proper schemes and algebraic spaces.
		\item Proper Deligne-Mumford stacks. For instance if $X$ is a smooth and proper algebraic variety over $k$ and $\overline{\cM}_{g,n}(X)$ denotes the moduli stack of stable curves of genus $g$ with $n$ marked points, then $\overline{\cM}_{g,n}(X)$ is a proper \DM stack. The same holds true for its natural derived enhancement.
		\item Higher classifying stacks $\mathrm K(G,n)$ where G is a compact (and abelian if $n \ge 2$) algebraic group scheme.
		For the case of $\rB G$ with $G$ reductive, see \cref{eg:reductive_strong_GAGA}.
	\end{enumerate}
\end{eg}

We would like to prove that a proper derived geometric stack locally almost of finite presentation over $k$ also satisfies the universal GAGA property.
Notice that when $X = \Spec(k)$ saying that $X$ satisfies the GAGA property relative to $U$ is equivalent to \cref{lem:perfect_complexes_affinoid} (in the \kanal case) and to \cref{thm:perfect_complexes_compact_Stein} (in the \canal case).
In order to generalize these results to a more general $X$ requires some effort.
We start dealing with the non-archimedean case, where the argument is technically easier.
However, we first state explicitly a lemma implicitly used in \cite{Porta_DCAGI}:

\begin{lem} \label{lem:analytification_Hom_sheaf}
	Let $X$ be a derived analytic stack locally almost of finite presentation over $k$.
	Let $\cF, \cG \in \Cohb(X)$.
	Then the canonical map
	\[ \cHom_{\cO_X}(\cF, \cG)\an \longrightarrow \cHom_{\cO_{X\an}}(\cF\an, \cG\an) \]
	is an equivalence.
\end{lem}

\begin{proof}
	This question is local on $X$ and we can therefore suppose that $X$ is affine.
	Notice that the map
	\[ \gamma_{\cF, \cG} \colon \cHom_{\cO_X}(\cF, \cG)\an \longrightarrow \cHom_{\cO_{X\an}}(\cF\an, \cG\an) \]
	is defined for all $\cF, \cG \in \QCoh(X)$.
	It is tautologically an equivalence when $\cF = \cO_X$.
	From here, it follows that it is an equivalence whenever $\cF$ is perfect.
	When $\cF \in \Cohb(X)$ is arbitrary, we use \cite[7.2.4.11(5)]{Lurie_Higher_algebra} to write $\cF$ as a geometric realization
	\[ \cF \simeq |\cP_\bullet| \]
	of a simplicial diagram $\cP_\bullet$ such that each $\cP_n$ is perfect.
	We obtain
	\[ \cHom_{\cO_X}(\cF, \cG) \simeq \lim_{[n] \in \mathbf \Delta} \cHom_{\cO_X}(\cP_n, \cG) . \]
	On the other hand, since $(-)\an$ commutes with arbitrary colimits, we have a canonical equivalence $\cF\an \simeq |\cP_\bullet\an|$.
	As a consequence, we obtain the equivalence
	\[ \cHom_{\cO_{X\an}}(\cF\an, \cG\an) \simeq \lim_{[n] \in \mathbf \Delta} \cHom_{\cO_{X\an}}(\cP_n\an, \cG\an) . \]
	It is therefore enough to prove that the natural map
	\[ \bigg( \lim_{[n] \in \mathbf \Delta} \cHom_{\cO_X}(\cP_n, \cG) \bigg)\an \longrightarrow \lim_{[n] \in \mathbf \Delta} \cHom_{\cO_X}(\cP_n,\cG)\an  \]
	is an equivalence.
	In order to check this, it is enough to check that for every integer $m \in \mathbb Z$ the above map induces an isomorphism
	\[ \pi_m \bigg( \lim_{[n] \in \mathbf \Delta} \cHom_{\cO_X}(\cP_n, \cG) \bigg)\an \longrightarrow \pi_m \lim_{[n] \in \mathbf \Delta} \cHom_{\cO_X}(\cP_n, \cG)\an . \]
	Since the analytification is $t$-exact, we are therefore reduced to check that the map
	\[ \bigg( \pi_m \bigg( \lim_{[n] \in \mathbf \Delta} \cHom_{\cO_X}(\cP_n, \cG) \bigg) \bigg)\an \longrightarrow \pi_m \bigg( \lim_{[n] \in \mathbf \Delta} \cHom_{\cO_X}(\cP_n, \cG)\an \bigg) \]
	is an isomorphism.
	Observe now that there exists $m' \gg 0$ such that
	\[ \pi_m \bigg( \lim_{[n] \in \mathbf \Delta} \cHom_{\cO_X}(\cP_n, \cG) \bigg) \simeq \pi_m \bigg( \lim_{[n] \in \mathbf \Delta_{\le m'}} \cHom_{\cO_X}(\cP_n, \cG) \bigg) \]
	and similarly
	\[ \pi_m \bigg( \lim_{[n] \in \mathbf \Delta} \cHom_{\cO_X}(\cP_n, \cG)\an \bigg) \simeq \pi_m \bigg( \lim_{[n] \in \mathbf \Delta_{\le m'}} \cHom_{\cO_X}(\cP_n, \cG)\an \bigg) . \]
	Since $\mathbf \Delta_{\le m}$ is a finite category and $(-)\an$ is an exact functor between stable $\infty$-categories, we deduce that the canonical map
	\[ \bigg( \lim_{[n] \in \mathbf \Delta_{\le m'}} \cHom_{\cO_X}(\cP_n, \cG) \bigg)\an \longrightarrow \lim_{[n] \in \mathbf \Delta_{\le m'}} \cHom_{\cO_X}(\cP_n, \cG)\an \]
	is an equivalence.
	The conclusion follows.
\end{proof}

\begin{thm} \label{thm:relative_GAGA}
	Let $k$ be either the field of complex numbers or a non-archimedean field equipped with a non-trivial valuation.
	Let $X$ be a proper derived geometric stack locally almost of finite presentation over $k$.
	Then $X$ satisfies the universal GAGA property.
\end{thm}

In particular, if $X$ is as above, \cref{prop:universal_GAGA_relative_stack_perfect_complexes} gives us the following result:

\begin{cor} \label{cor:analytification_relative_stack_perfect_complexes}
	Let $k$ be either the field of complex numbers or a non-archimedean field equipped with a non-trivial valuation.
	Let $X$ be a derived geometric stack locally almost of finite presentation over $k$.
	Assume that:
	\begin{enumerate}
		\item the stack $X$ is proper;
		\item the stack $\bfMap(X, \bfPerf_k)$ is locally geometric.
	\end{enumerate}
	Then the canonical map
	\[ \mu_X \colon \bfMap(X, \bfPerf_k)\an \longrightarrow \bfAnMap(X\an, \bfAnPerf_k) \]
	is an equivalence.
\end{cor}

\begin{rem} \label{rem:geometricity_mapping_stacks}
	In the above corollary, the need for the geometricity assumption on $\bfMap(X, \bfPerf_k)$ ultimately comes from \cref{thm:generalized_adjunction}.
	The main theorem of \cite{Toen_Vaquie_Moduli_of_objects} shows that this assumption is satisfied when $X$ is a smooth and proper scheme over $k$.
	More generally, this problem can be seen as a particular instance of the geometricity of $\bfMap(X,Y)$ for $X, Y \in \dSt_k$.
	In \cite[Proposition 3.3.8]{DAG-XIV} it is shown that $\bfMap(X,Y)$ is geometric whenever $X$ is a proper and flat derived algebraic space and $Y$ is a derived \DM stack locally almost of finite presentation.
	These results can be improved: we expect $\bfMap(X, \bfPerf_k)$ to be locally geometric whenever $X$ is proper and of finite tor-amplitude.
	The main tool to prove this theorem is the version of Artin-Lurie's representability theorem for derived Artin stacks that will appear in \cite[Chapter 27]{Lurie_SAG}.
	As usual, the critical assumptions to be verified are the integrability of $\bfMap(X, \bfPerf_k)$ and the existence of its cotangent complex.
	The latter can easily be established by combining properness and finite tor-amplitude following the same method of \cite[Proposition 3.3.23]{DAG-XII}.
	Notice that we do not need the functor $f_+$ to be defined on the whole $\QCoh( X \times \bfMap(X, \bfPerf_k))$, but it is enough to have it defined on $\Perf( X \times \bfMap(X, \bfPerf_k))$: see \cite[Lemma 8.4]{Porta_Yu_Mapping} for a similar situation where $f_+$ can only be defined for perfect complexes.
	On the other hand, the integrability of $\bfMap(X, \bfPerf_k)$ can be reduced to the statement of the formal GAGA equivalence for the stack $X \times \Spec(R) \to \Spec(R)$, for $R$ a local complete Noetherian ring.
	This result can be obtained by extension from the analogous statements for proper schemes over $\Spec(R)$ in the same way as in \cite{Porta_Yu_Higher_analytic_stacks_2014}.
	
	Finally, let us remark that, in the absence of the strong version of Artin-Lurie's representability theorem one can combine the main theorem of \cite{Olsson_Mapping_stack} with the weak version of Lurie's representability theorem \cite[Theorem 18.1.0.2]{Lurie_SAG} to deduce the representability of $\bfMap(X,\bfPerf_k)$ from the representability of its truncation.
	In order for this method to go through, one needs to assume $X$ to be proper and flat over $k$.
\end{rem}

We now turn to the proof of the theorem:

\begin{proof}[Proof of \cref{thm:relative_GAGA}]
	Let us first deal with the non-archimedean setting.
	In this case, we have to check that the analytification functor
	\begin{equation} \label{eq:relative_analytification_perfect_complexes_non_archimedean}
		(-)\an \colon \Perf(\Spec(A_V) \times X) \longrightarrow \Perf(V \times X\an)
	\end{equation}
	is an equivalence.
	We will prove in fact that the analytification functor
	\begin{equation} \label{eq:relative_analytification_unbounded_coherent_non_archimedean}
		(-)\an \colon \Coh(X \times \Spec(A_V)) \longrightarrow \Coh(X\an \times V)
	\end{equation}
	is an equivalence.
	Notice that the flatness of derived analytification proved in \cite[Proposition 4.17]{Porta_Yu_Representability} together with \cite[Proposition 2.8.4.2(5)]{Lurie_SAG} implies that an object $\cF \in \Coh(X \times \Spec(A_V))$ is a perfect complex if and only if $\cF\an$ belongs to $\Perf(X\an \times V)$.
	From this remark and the claimed equivalence, it follows immediately that \eqref{eq:relative_analytification_perfect_complexes_non_archimedean} is an equivalence as well.
	
	We start by proving that the functor \eqref{eq:relative_analytification_unbounded_coherent_non_archimedean} is fully faithful.
	Recall that pushing forward along the natural closed immersions
	\[ \trunc(X \times \Spec(A_V)) \hookrightarrow X \times \Spec(A_V) \quad , \quad \trunc(X\an \times V) \hookrightarrow X\an \times V \]
	induces equivalences of abelian categories 
	\[ \Cohh(X \times \Spec(A_V)) \simeq \Cohh(\trunc(X \times \Spec(A_V))) \quad , \quad \Cohh(X\an \times V) \simeq \Cohh(\trunc(X\an \times V)) . \]
	Notice that $\trunc(X \times \Spec(A_V)) \simeq \trunc(X) \times \Spec(\pi_0(A_V))$.
	Furthermore $\pi_0(A_V) \simeq \Gamma(\trunc(V); \cO_{\trunc(V)}\alg)$.
	Applying \cite[Theorem 7.1]{Porta_Yu_Higher_analytic_stacks_2014} we see that the diagram
	\[ \begin{tikzcd}
		\Cohh(\trunc(X) \times \Spec(\pi_0(A_V))) \arrow{r}{(-)\an} \arrow{d}{p_*} & \Cohh(\trunc(X\an) \times \trunc(V)) \arrow{d}{p\an_*} \\
		\Coh(\pi_0(A_V)) \arrow{r} & \Coh(\trunc(V))
	\end{tikzcd} \]
	commutes.
	Since the diagram
	\[ \begin{tikzcd}
		\Coh(\pi_0(A_V)) \arrow{r} \arrow{d} & \Coh(\trunc(V)) \arrow{d} \\
		\Coh(A_V) \arrow{r} & \Coh(V)
	\end{tikzcd} \]
	commutes as well, we may form a cube with five commuting faces and deduce that its sixth face, the diagram 
		\[ \begin{tikzcd}
		\Cohh(X \times \Spec(A_V)) \arrow{d}{q_*} \arrow{r} & \Cohh(X\an \times V) \arrow{d}{q\an_*} \\
		\Coh(A_V) \arrow{r} & \Coh(V) 
	\end{tikzcd}, \]
	also commutes.
	From here, proceeding by induction on the cohomological amplitude as in the proof of \cite[Theorem 7.1]{Porta_DCAGI} we deduce that the diagram (written in \emph{homological} convention)
	\[ \begin{tikzcd}
		\Coh^-(X \times \Spec(A_V)) \arrow{d}{q_*} \arrow{r} & \Coh^-(X\an \times V) \arrow{d}{q\an_*} \\
		\Coh(A_V) \arrow{r} & \Coh(V)
	\end{tikzcd} \]
	commutes as well.
	Let now $\cF, \cG \in \Cohb(X \times \Spec(A_V))$.
	Applying \cref{lem:analytification_Hom_sheaf} we see that
	\begin{equation} \label{eq:analytification_Hom_sheaf}
		\cHom_{X \times \Spec(A_V)}(\cF, \cG)\an \simeq \cHom_{X\an \times V}(\cF\an, \cG\an) .
	\end{equation}
	Notice furthermore that the functor $\Coh(A_V) \to \Coh(V)$ coincides with the equivalence provided by \cite[Theorem 3.4]{Porta_Yu_Mapping}.
	Combining the equivalence $\Coh(A_V) \simeq \Coh(V)$ with \eqref{eq:analytification_Hom_sheaf} and with the commutativity of the above diagram, we deduce that the analytification functor restricts to a fully faithful functor
	\[ \Cohb(X \times \Spec(A_V)) \longrightarrow \Coh(X\an \times V) . \]
	From here, a second induction on the cohomological amplitude as the one that can be found in \cite[Theorem 7.2]{Porta_DCAGI} proves that the functor \eqref{eq:relative_analytification_unbounded_coherent_non_archimedean} is also fully faithful.
	
	For essential surjectivity, we first use \cite[Theorem 7.4]{Porta_Yu_Higher_analytic_stacks_2014} to deduce that the analytification induces an equivalence
	\[ \Cohh(X \times \Spec(A_V)) \simeq \Cohh(X \times V) . \]
	Next we bootstrap on this using full faithfulness of \eqref{eq:relative_analytification_unbounded_coherent_non_archimedean} to deduce that the analytification functor on unbounded coherent sheaves is also essentially surjective.
	The conclusion follows.\\
	
	We now turn to the \canal situation.
	In this case, we have to prove that for every compact Stein subset $K \subset U$ the morphism

\[ \fcolim_{K \subset V \subset U} \Perf(X \times \Spec(A_V)) \longrightarrow \fcolim_{K \subset V \subset U} \Perf(X\an \times V) \]
is an equivalence in $\Ind(\Cat_\infty^{\mathrm{st}, \otimes})$.
Just as in the non-archimedean setting, we prove below that actually the map
\begin{equation} \label{eq:relative_GAGA_complex_II}
	\fcolim_{K \subset V \subset U} \Coh(X \times \Spec(A_V)) \longrightarrow \fcolim_{K \subset V \subset U} \Coh(X\an \times V)
\end{equation}
is an equivalence, where the two colimits range over the open Stein neighbourhoods of $K$ inside $U$.
Notice that when $X = \Spec(k)$, this is exactly the result proven in \cref{thm:unbounded_coherent_pro_compact_Stein}.

In order to prove that the functor \eqref{eq:relative_GAGA_complex_II} is an equivalence we will need the following two claims, which will be proved below:
\begin{enumerate}
	\item \label{item:full_faithfulness_relative_GAGA_complex} for every open Stein neighbourhood $V$ of $K$ in $U$ the analytification map
	\[ \Coh(X \times \Spec(A_V)) \longrightarrow \Coh(X\an \times V) \]
	is fully faithful (see \cref{prop:relative_analytification_coherent_sheaves_fully_faithful} below).
	
	\item \label{item:essential_surjectivity_relative_GAGA_complex} Let $W \Subset V \Subset U$ be two relatively compact Stein neighbourhood of $K$ inside $U$.
	Then the map
	\[ \Coh(X\an \times V) \longrightarrow \Coh(X\an \times W) \]
	factors through $\Coh(X \times \Spec(A_W)) \to \Coh(X\an \times W)$ (see \cref{prop:relative_GAGA_complex_discrete} below).
\end{enumerate}
We can therefore promote the functors of assertion (2) to a morphism
\[ \fcolim_{K \subset V \subset U} \Coh(X\an \times V) \longrightarrow \fcolim_{K \subset V \subset U} \Coh(X \times \Spec(A_V)) . \]
It is easily checked that this forms an equivalence together with the functor \eqref{eq:relative_GAGA_complex_II}.
\end{proof}

To prove the claims we need the following preliminary result:

\begin{lem} \label{lem:relative_GAGA_I}
	Let $X$ be a proper derived geometric $\C$-stack.
	Let $U \in \dStn_{\mathbb C}$ be a derived Stein space.
	Write $A_U \coloneqq \Gamma(U;\cO_U)$ and let $p_U \colon X \times \Spec(A_U) \to \Spec(A_U)$ and $p_U\an \colon X\an \times U \to U$ be the two canonical projections.
	Then the diagram (written in \emph{homological} convention)
	\[ \begin{tikzcd}
		\Coh^-(X \times \Spec(A_U)) \arrow{d}{p_{U*}} \arrow{r}{\varepsilon_{X,U}^*} & \Coh^-(X\an \times U) \arrow{d}{p\an_{U*}} \\
		\Coh^-(\Spec(A_U)) \arrow{r}{\varepsilon_U^*} & \Coh^-(U)
	\end{tikzcd} \]
	canonically commutes.
	Here $\varepsilon^*_U$ denotes the functor introduced in \cref{subsec:analytification_Perf}.
\end{lem}

\begin{proof}
	Proceeding by induction on the cohomological amplitude as in the proof of \cite[Theorem 7.1]{Porta_DCAGI}, we see that it is enough to prove that the diagram
	\[ \begin{tikzcd}
		\Cohh(X \times \Spec(A_U)) \arrow{d}{p_{U*}} \arrow{r}{\varepsilon_{X,U}^*} & \Cohh(X\an \times U) \arrow{d}{p\an_{U*}} \\
		\Coh^-(\Spec(A_U)) \arrow{r}{\varepsilon_U^*} & \Coh^-(U)
	\end{tikzcd} \]
	commutes.
	
	We first deal with the case where $X$ is a proper derived $\C$-\emph{scheme}.
	Fix $\cF \in \Cohh(X \times \Spec(A_U))$.
	In this case, the \v{C}ech complex computing both $p_{U*}(\cF)$ and $p\an_{U*}(\cF\an)$ is cohomologically bounded.
	As \cref{prop:basic_properties_global_sections} shows that $\varepsilon^*_U$ is $t$-exact, we deduce that $\varepsilon^*_U(p_{U*}(\cF))$ is also cohomologically bounded.
	We are therefore left to check that the canonical map
	\[ \gamma_\cF \colon \varepsilon^*_U(p_{U*}(\cF)) \longrightarrow p\an_{U*}(\cF\an) \]
	between objects in $\Cohb(U)$ is an equivalence.
	Let $\cG \coloneqq \fib(\gamma_\cF)$.
	Equivalently, we have to prove that $\cG \simeq 0$.
	Since $\cG$ is cohomologically bounded, the cohomological Nakayama's lemma implies that it is enough to check that for every closed point $x \colon \Sp(\C) \to U$ one has $x^* \cG \simeq 0$.
	On the other hand, the derived base change and its analytic counterpart\footnote{Since $x \colon \Sp(\C) \hookrightarrow U$ is a closed immersion, the analytic base change follows from the unramifiedness of $\cTanc$. This can  be proved as in Lemma 6.4 of \cite{Porta_Yu_Mapping}, the key ingredients in the derived setting are \cite[Propositions 11.12(3) and 12.10]{DAG-IX}.} imply that the two diagrams
	\[ \begin{tikzcd}[column sep = large]
		\Coh^+(X \times \Spec(A_U)) \arrow{d}{p_{U*}} \arrow{r}{(\mathrm{id}_X \times x)^*} & \Coh^+(X) \arrow{d}{p_*} \\
		\Coh^+(A_U) \arrow{r}{x^*} & \Coh^+(\C)
	\end{tikzcd} \quad \text{and} \quad \begin{tikzcd}[column sep = large]
	\Coh^+(X\an \times U) \arrow{d}{p\an_{U*}} \arrow{r}{(\mathrm{id}_{X\an} \times x)^*} & \Coh^+(X\an) \arrow{d}{p\an_*} \\
	\Coh^+(U) \arrow{r}{x^*} & \Coh^+(\C)
	\end{tikzcd} \]
	are commutative.
	In this way, we can reduce ourselves to the case where $U = \Sp(\C)$, and in this case the statement follows from the equivalences
	\[ \Cohh(X \times \Spec(A_U)) \simeq \Cohh(\trunc(X) \times \Spec(\pi_0(A_U))) \quad \text{and} \quad \Cohh(X\an \times U) \simeq \Cohh(\trunc(X)\an \times \trunc(U)) \]
	and \cite[Theorem 7.1]{Porta_Yu_Higher_analytic_stacks_2014} (in fact, the classical GAGA theorem that can be found in \cite[Expos\'e XII, Th\'eor\`eme 4.4]{SGA1} is enough for this step).
	
	At this point, we proceed by induction on the geometric level of $X$.
	We notice that the same proof as in \cite[Theorem 7.1]{Porta_Yu_Higher_analytic_stacks_2014} applies.
	The reader should be wary that also in this case the noetherian induction has to be performed on $X$ (and not on $X \times \Spec(A_U)$).
	The reader should also be aware that in loc.\ cit.\ the cohomological convention was used, while in this paper we are following the homological one.
\end{proof}

\begin{prop} \label{prop:relative_analytification_coherent_sheaves_fully_faithful}
	Let $X$ be a proper derived geometric $\C$-stack.
	Let $U \in \dStn_{\mathbb C}$ be a derived Stein space.
	Then the analytification functor
	\[ \Coh(X \times \Spec(A_U)) \longrightarrow \Coh(X\an \times U) \]
	is fully faithful.
\end{prop}

\begin{proof}
	Fix $\cF, \cG \in \Coh(X \times \Spec(A_U))$.
	We have to prove that the natural morphism
	\[ \psi_{\cF, \cG} \colon \Map^{\mathrm{st}}_{X \times \Spec(A_U)}(\cF, \cG) \longrightarrow \Map^{\mathrm{st}}_{X\an \times U}(\cF\an, \cG\an) \]
	is an equivalence.
	Let $p \colon X \times \Spec(A_U) \to \Spec(\C)$ and $p\an \colon X\an \times U \to \Sp(\C)$ be the two canonical maps to the point.
	It follows from the definitions that we have natural equivalences
	\[ \Map^{\mathrm{st}}_{X \times \Spec(A_U)}(\cF, \cG) \simeq \tau_{\ge 0} p_* \cHom_{X \times \Spec(A_U)}(\cF, \cG) \quad , \quad \Map^{\mathrm{st}}_{X\an \times U}(\cF\an, \cG\an) \simeq \tau_{\ge 0} p\an_* \cHom_{X\an \times U}(\cF\an, \cG\an) . \]
	If $\cF, \cG \in \Cohb(X \times \Spec(A_U))$ then the same argument given in \cref{lem:analytification_Hom_sheaf} shows that the canonical map
	\[ \zeta_{\cF, \cG} \colon \cHom_{X \times \Spec(A_U)}(\cF, \cG)\an \longrightarrow \cHom_{X\an \times U}(\cF\an, \cG\an) \]
	is an equivalence.
	Furthermore, in this case $\cHom_{X \times \Spec(A_U)}(\cF, \cG)$ belongs to $\Coh^-(X \times \Spec(A_U))$ and therefore we can use \cref{lem:relative_GAGA_I} to deduce that the canonical map
	\[ \varepsilon^*_U \big( p_{U*} \big( \cHom_{X \times \Spec(A_U)}(\cF, \cG) \big) \big) \longrightarrow p\an_{U*} \big( \cHom_{X \times \Spec(A_U)}(\cF, \cG)\an \big)  \]
	is an equivalence.
	Composing with the equivalence $\zeta_{\cF, \cG}$, applying the global section functor $\Gamma(U;-)$ and using \cref{prop:basic_properties_global_sections}(4) we deduce that the canonical map
	\[ p_* \cHom_{X \times \Spec(A_U)}(\cF, \cG) \longrightarrow p\an_* \cHom_{X\an \times U}(\cF\an, \cG\an) \]
	is an equivalence.
	Therefore, the conclusion follows in the case where $\cF$ and $\cG$ are (locally) cohomologically bounded.
	At this point, the argument given in the proof of \cite[Theorem 7.2]{Porta_DCAGI} shows that the map $\psi_{\cF, \cG}$ is an equivalence whenever $\cF, \cG$ belong to $\Coh(X)$.
\end{proof}

For later use we record the following useful consequence:

\begin{cor} \label{cor:relative_algebraizability}
	Let $X$ be a proper derived geometric $\C$-stack and let $j \colon X_{\mathrm{red}} \to X$ be the canonical map.
	Let $U \in \dSt_\C$ be a derived affinoid.
	Then for $\cF \in \Coh(U \times X\an)$ the following conditions are equivalent:
	\begin{enumerate}
		\item the coherent sheaf $\cF$ is algebraizable, i.e.\ it belongs to the essential image of the functor $\Coh(\Spec(A_U) \times X) \to \Coh(U \times X\an)$;
		
		\item the discrete sheaf $\rH^i(\cF)$ is algebraizable for every $i \in \mathbb Z$.
	\end{enumerate}
	If furthermore $\cF \in \Coh^+(U \times X\an)$, then the above conditions are equivalent to:
	\begin{enumerate} \setcounter{enumi}{2}
		\item the pullback $(\mathrm{id}_{U} \times j\an)^* \cF \in \Coh(U \times X_{\mathrm{red}}\an)$ is algebraizable.
	\end{enumerate}
\end{cor}

\begin{proof}
	Since the analytification functor $(-)\an \colon \Coh(\Spec(A_U) \times X) \to \Coh(U \times X\an)$ is $t$-exact, it commutes with both the limit and the colimit of the Postnikov tower.
	This shows immediately that $\cF$ is algebraizable if and only if for every $n, m \in \mathbb Z$ the sheaf $\tau_{\le n}\tau_{\ge m}(\cF)$ is algebraizable.
	Moreover, \cref{prop:relative_analytification_coherent_sheaves_fully_faithful} shows that this functor is also fully faithful.
	A simple induction on the number of nonvanishing cohomology groups implies therefore the equivalence between (1) and (2).
	
	Assume now that $\cF \in \Coh^+(U \times X\an)$.
	Then the implication (1) $\Rightarrow$ (3) is clear.
	Let us prove that (3) implies (2).
	Since $\cF$ is eventually connective, we can choose the minimum integer $i$ such that $\rH^i(\cF)$ is non-zero.
	Using the fibre sequence
	\[ \tau_{\le i + 1} \cF \longrightarrow \cF \longrightarrow \rH^i(\cF) \]
	we see that it is enough to prove that $\rH^i(\cF)$ is algebraizable.
	We can furthermore replace both $U$ and $X$ by their truncations, and therefore assume that they are underived.
	Let $\cJ$ be the nilradical ideal sheaf of $X$, $\cJ\an$ its analytification and let $\cJ_U\an$ be pullback of $\cJ\an$ along $U \times X\an \to X\an$.
	Since $X$ is proper, there exists an integer $n$ such that $\cJ^n = 0$.
	We now observe that
	\[ \rH^i(\cF) / \cJ_U\an \rH^i(\cF) \simeq \rL^0 (\mathrm{id}_{U} \times j\an)^* \rH^i(\cF) \simeq \rH^i( (\mathrm{id}_{U} \times j\an)^* \cF ) . \]
	This implies that $\rH^i(\cF) / \cJ_U\an$ is algebraizable.
	Proceeding by induction on $m$ as in \cite[Theorem 5.13]{Porta_Yu_Higher_analytic_stacks_2014} we see that $\rH^i(\cF) / (\cJ_U\an)^m \rH^i(\cF)$ is algebraizable for every $m \ge 1$.
	Taking $m = n$ we conclude that $\rH^i(\cF)$ is algebraizable, thus completing the proof.
\end{proof}

At this point, the only missing piece needed for the proof of \cref{thm:relative_GAGA} is the following:

\begin{prop} \label{prop:relative_GAGA_complex_discrete}
	Let $X$ be a proper derived geometric $\C$-stack.
	Let $U \in \dStn_{\mathbb C}$ be a derived Stein space and let $W \Subset V \Subset U$ be a nested sequence of relatively compact open Stein subsets of $U$.
	Then there exists a functor $\Coh(X\an \times U) \to \Coh(X \times \Spec(A_W))$ making the diagram
	\[ \begin{tikzcd}
		{} & \Coh(X \times \Spec(A_W)) \arrow{d}{(-)\an} \\
		\Coh(X\an \times U) \arrow{ur} \arrow{r} & \Coh(X\an \times W)
	\end{tikzcd} \]
	commutative.
\end{prop}

\begin{proof}
	Using \cref{cor:relative_algebraizability} we see that it is enough to prove the same statement at the level of hearts.
	Using \cref{prop:relative_analytification_coherent_sheaves_fully_faithful}, we see that the relative analytification functor
	\begin{equation} \label{eq:relative_GAGA_complex_discrete_neighborhood}
		(-)\an \colon \Cohh(X \times \Spec(A_W)) \to \Cohh(X\an \times W)
	\end{equation}
	is fully faithful.
	Therefore, it is enough to prove that the restriction functor
	\[ \Cohh(X\an \times U) \longrightarrow \Cohh(X\an \times W) \]
	factors through the essential image of $(-)\an$.
	
	We first deal with the case where $X$ is a scheme.
	Notice that, using the equivalences
	\[ \Cohh(X \times \Spec(A_U)) \simeq \Cohh(\trunc(X) \times \Spec(\pi_0(A_U))) \quad , \quad \Cohh(X\an \times U) \simeq \Cohh(\trunc(X)\an \times \trunc(U)) , \]
	we can assume that both $X$ and $U$ (and hence $V$ and $W$) are underived.
	Under this hypothesis, we proceed by noetherian induction on the dimension of $X$.
	Using Chow's lemma as in \cite[Expos\'e XII, Th\'eor\`eme 4.4]{SGA1} we are readily reduced to the case of projective space, $X = \bbP^n_\C$.
	Write $\bP^n_\C \coloneqq (\bbP^n_\C)\an$.
	Let
	\begin{gather*}
		p_U \colon \bbP^n_\C \times \Spec(A_U) \longrightarrow \Spec(A_U) \quad , \quad p\an_U \colon \bP^n_\C \times U \longrightarrow U , \\
		q_U \colon \bbP^n_\C \times \Spec(A_U) \longrightarrow \bbP^n_\C \quad , \quad q\an_U \colon \bP^n_\C \times U \longrightarrow \bP^n_\C
	\end{gather*}
	be the natural projections.
	For $m \in \mathbb Z$, we write
	\[ \cO_{\bbP^n_\C \times \Spec(A_U)}(m) \coloneqq q_U^* \cO_{\bbP^n_\C}(m) \quad , \quad \cO_{\bP^n_\C \times U}(m) \coloneqq q^{\mathrm{an}*}_U \cO_{\bP^n_\C}(m) . \]
	Given $\cF \in \Cohh(\bP^n_\C \times U)$, we consider the canonical map
	\[ \phi_m \colon \cG \coloneqq \rL^0 p^{\mathrm{an}*}_U \rR^0 p\an_{U*}( \cF(-m) ) \otimes \cO_{\bP^n_\C \times U}(m) \longrightarrow \cF . \]
	For every point $x \in \overline{V}$ there is an integer $m_x \in \mathbb Z$ such that for $m \ge m_x$ the pullback of this morphism along $\mathrm{id}_{\bP^n_\C} \times x \colon \bP^n_\C \to \bP^n_\C \times U$ becomes surjective.
	As both $\cG$ and $\cF$ are coherent, Nakayama's lemma implies that there exists a neighbourhood $U_x$ of $x$ such that for $m \ge m_u$ the map $\phi_m$ becomes surjective when restricted to $U_x$.
	As $\overline{V}$ is compact, we can therefore find an open Stein subset $U'$ of $U$ containing $\overline{V}$ and an integer $m$ such that the restriction of $\phi_m$ to $U'$ is surjective.
	In particular, the restriction of $\phi_m$ to $V$ is surjective.
	Repeating the same reasoning for the kernel of $\phi$ on $U'$, we find a second open Stein subset $U''$ of $U$ containing $\overline{V}$ such that $\cF |_{U''}$ admits a presentation of the form
	\[ \begin{tikzcd}
		\cH \arrow{r}{f} & \cG \arrow{r} & \cF |_{U''} \arrow{r} & 0 ,
	\end{tikzcd} \]
	where $\cH$ and $\cG$ can be written as
	\[ \cH \simeq \rL^0 p^{\mathrm{an}*}_{U''}( \cH_0 ) \otimes \cO_{\bP^n_\C \times U''}(m_2) \quad , \quad \cG \simeq \rL^0 p^{\mathrm{an}*}_{U''}( \cG_0 ) \otimes \cO_{\bP^n_\C \times U''}(m_1)  \]
	for $\cH_0, \cG_0 \in \Cohh(U'')$ and $m_1, m_2 \gg 0$.
	In particular, the same remains true after restricting to $V$.
	Using \cref{thm:unbounded_coherent_pro_compact_Stein} we see that $\cH_0|_W$, $\cG_0 |_W$ come from objects $\cH_0\alg$ and $\cG_0\alg$ in $\Cohh(A_{W})$ via the functor $\varepsilon_W^*$.
	Since we already argued that the functor \eqref{eq:relative_GAGA_complex_discrete_neighborhood} is fully faithful, we can find a map
	\[ f\alg \colon p_{W}^*(\cH_0\alg) \otimes \cO_{\bbP^n_\C \times \Spec(A_W)}(m_2) \longrightarrow p_{W}^*(\cG_0\alg) \otimes \cO_{\bbP^n_\C \times \Spec(A_W)}(m_1) \]
	whose analytification coincides with the map $f \colon \cH \to \cG$.
	Set $\cF\alg \coloneqq \mathrm{coker}(f\alg)$.
	Then $(\cF\alg)\an \simeq \cF$, i.e.\ $\cF$ belongs to the essential image of the analytification functor.
	At this point, the extension to a generic proper geometric stack $X$ goes as in \cite[Theorem 7.4]{Porta_Yu_Higher_analytic_stacks_2014}.
\end{proof}

\subsection{More examples} \label{section:examples}

Building on \cref{thm:relative_GAGA} we can prove that a number of different stacks satisfy the universal GAGA property.
We start by discussing the case of formal completions.
Next we study three stacks coming from Hodge theory, the de Rham, Betti and Dolbeault stacks.
Although we briefly recall their definitions below, we refer the reader to \cite[\S 2.1]{PTVV_2013} for a more thorough discussion.
Finally we consider the case of $\rB G$ where $G$ is a complex reductive algebraic group.

\subsubsection{Formal completions} \label{eg:formal_completion_strong_GAGA}

Let $X$ be a derived geometric stack locally almost of finite presentation over $k$ and let $Y \hookrightarrow X$ be a closed immersion.
We also suppose that $Y$ is locally almost of finite presentation.

\begin{defin}
	Let $\mathrm{Nil}_{/X}(Y)$ be the full subcategory of $(\dSt_k)_{Y_{\mathrm{red}} // X}$ spanned by morphisms $Y_{\mathrm{red}} \to Z \to X$ where $Z$ is a derived geometric stack locally almost of finite presentation and the map $Y_{\mathrm{red}} \to Z$ induces an equivalence $Y_{\mathrm{red}} \simeq Z_{\mathrm{red}}$.
\end{defin}

\begin{defin}
	Let $X, Y$ be derived geometric stacks locally almost of finite presentation over $k$ and let $i \colon Y \hookrightarrow X$ be a closed immersion.
	We define the \emph{formal completion of $X$ along $Y$} as the colimit
	\[ X_Y^\wedge \coloneqq \colim_{Z \in \mathrm{Nil}_{/X}(Y)} Z . \]
\end{defin}

The same definitions can be applied in the analytic setting.
We have the following global analogue of \cite[Proposition 8.2]{Porta_DCAGI}

\begin{lem}
	The analytification functor induces an equivalence of $\infty$-categories
	\[ (-)\an \colon \mathrm{Nil}_{/X}(Y) \simeq \mathrm{Nil}_{/X\an}(Y\an) . \]
\end{lem}

\begin{proof}
	We first remark that for every derived geometric stack $Z$ locally almost of finite presentation over $k$, the canonical map
	\[ (Z_{\mathrm{red}})\an \longrightarrow (Z\an)_{\mathrm{red}} \]
	is an equivalence.
	This implies that the analytification functor induces a well defined map
	\[ (-)\an \colon \mathrm{Nil}_{/X}(Y) \longrightarrow \mathrm{Nil}_{/X\an}(Y\an) . \]
	Let $U_\bullet$ be an affine hypercover of $X$ and let $Y_\bullet \coloneqq Y \times_X U_\bullet$.
	Then for any $[n] \in \mathbf \Delta$, the map $Y_n \to U_n$ is a closed immersion, and in particular $Y_n$ is an affine derived scheme almost of finite presentation.
	We have canonical equivalences
	\[ \mathrm{Nil}_{/X}(Y) \simeq \lim_{\mathbf \Delta} \mathrm{Nil}_{/U_\bullet}(Y_\bullet) \quad , \quad \mathrm{Nil}_{/X\an}(Y\an) \simeq \lim_{\mathbf \Delta} \mathrm{Nil}_{/U_\bullet\an}(Y_\bullet\an) . \]
	We are therefore reduced to the case where $X$ itself is affine.
	
	We first prove that it is fully faithful.
	Let $Z, T \in \mathrm{Nil}_{/Y}(X)$.
	Since $X$ is affine, we have
	\[ \Map_{Y_{\mathrm{red}}\an // X\an}(Z\an, T\an) \simeq \Map_{(Y_{\mathrm{red}}\an)\alg // X}((Z\an)\alg, T) . \]
	Fix now $T \in \mathrm{Nil}_{/Y}(X)$ and consider the full subcategory $\cC$ of $\mathrm{Nil}_{/X}(Y)$ spanned by those $Z$ for which the canonical map
	\[ \Map_{Y_{\mathrm{red}} // X}( Z, T) \longrightarrow \Map_{(Y_{\mathrm{red}})\alg // X}( (Z\an)\alg, T ) \]
	is an equivalence.
	We observe that:
	\begin{enumerate}
		\item the object $Y_{\mathrm{red}}$ belongs to $\cC$;
		\item the category $\cC$ is closed under colimits in $\mathrm{Nil}_{/X}(Y)$.
	\end{enumerate}
	Proceeding by induction, we are therefore left to check that if $Z \in \cC$ and $M \in \Coh^{\ge 1}(Z)$, then $Z[M] \in \cC$.
	This follows at once from \cite[Theorem 5.21]{Porta_Yu_Representability}.
	This shows that the functor is fully faithful.
	For essential surjectivity, recall from \cite[Proposition 8.1]{Porta_DCAGI} that a derived geometric analytic stack $Z$ is algebraizable if and only if $\trunc(Z)$ is algebraizable.
	Since $\trunc(Z)$ is algebraizable if and only if $\trunc(Z)_{\mathrm{red}} = Z_{\mathrm{red}}$ is algebraizable, we conclude that the above functor is essentially surjective.
\end{proof}

\begin{cor} \label{cor:analytification_formal_completions}
	Let $X, Y$ be derived geometric stacks locally almost of finite presentation over $k$.
	The canonical morphism
	\[ (X_Y^\wedge)\an \longrightarrow (X\an_{Y\an})^\wedge \]
	is an equivalence.
\end{cor}

\begin{proof}
	This follows from the fact that $(-)\an$ commutes with colimits by construction and from the above lemma.
\end{proof}

\begin{prop} \label{prop:analytification_formal_completion_conservative_t_exact}
	Let $X, Y$ be derived geometric stacks locally almost of finite presentation over $k$.
	Let $U \in \dAfd_k$ be a derived $k$-affinoid (resp.\ Stein) space.
	The canonical map
	\begin{equation} \label{eq:analytification_formal_completion_conservative_t_exact}
		(-)\an \colon \QCoh(\Spec(A_U) \times X_Y^\wedge) \longrightarrow \cO_{U \times (X_Y^\wedge)\an} \Mod
	\end{equation}
	is conservative and $t$-exact.
\end{prop}

\begin{proof}
	For every $Z \in \mathrm{Nil}_{/X}(Y)$ we let
	\[ j_Z \colon Z \longrightarrow X_Y^\wedge \quad , \quad j_{Z\an} \colon Z\an \longrightarrow (X_Y)^\wedge \]
	be the two canonical maps.
	The morphisms
	\[ j_Z^* \colon \QCoh(\Spec(A_U) \times X_Y^\wedge) \longrightarrow \QCoh(\Spec(A_U) \times Z) \quad , \quad j_{Z\an}^* \colon \cO_{U \times (X_Y^\wedge)\an} \Mod \longrightarrow \cO_{U \times Z\an} \Mod \]
	are jointly conservative as $Z$ varies in $\mathrm{Nil}_{/X}(Y)$.
	Since $Z$ is a derived geometric stack, \cref{eg:Artin_strong_GAGA} implies that the map
	\[ (-)\an \colon \QCoh(\Spec(A_U) \times Z) \longrightarrow \cO_{U \times Z\an} \Mod \]
	is conservative.
	As
	\[ \QCoh(\Spec(A_U) \times X_Y^\wedge) \simeq \lim_{Z \in \mathrm{Nil}_{/X}(Y)} \QCoh(\Spec(A_U) \times Z) \quad , \quad \cO_{U \times (X_Y^\wedge)\an} \Mod \simeq \lim_{Z \in \mathrm{Nil}_{/X}(Y)} \cO_{U \times Z\an} \Mod , \]
	it follows that \eqref{eq:analytification_formal_completion_conservative_t_exact} is conservative.
	
	It is also clear that the functor \eqref{eq:analytification_formal_completion_conservative_t_exact} is left $t$-exact.
	Let now $\cF \in \QCoh(\Spec(A_U) \times X_Y^\wedge)^{\le 0}$.
	We have to prove that for every $\cG \in \cO_{U \times (X_Y^\wedge)\an} \Mod^{\ge 1}$, we have
	\[ \Map_{\cO_{U \times (X_Y^\wedge)\an} \Mod}( \cG, \cF\an) \simeq 0 . \]
	By definition, we have
	\[ \Map_{\cO_{U \times (X_Y^\wedge)\an} \Mod}( \cG, \cF\an) \simeq \lim_{Z \in \mathrm{Nil}_{/X}(Y)} \Map_{\cO_{U \times Z\an}}( j_{Z\an}^* \cG, j_{Z\an}^* \cF\an ) . \]
	However $j_{Z\an}^* \cF\an \simeq ( j_Z^* \cF )\an$, and the analytification functor
	\[ (-)\an \colon \QCoh(\Spec(A_U) \times Z) \longrightarrow \cO_{U \times Z\an} \Mod \]
	is $t$-exact.
	Since moreover $j_{Z\an}^* \cG \in \cO_{U \times Z\an} \Mod^{\ge 1}$ by the definition of the $t$-structure on $\cO_{U \times (X_Y^\wedge)\an} \Mod$, the conclusion follows.
\end{proof}

\begin{prop} \label{prop:formal_completions_relative_GAGA}
	Let $X, Y$ be derived geometric stacks locally almost of finite presentation over $k$.
	Suppose furthermore that $Y$ is proper.
	Then the formal completion $X_Y^\wedge$ satisfies the universal GAGA property.
\end{prop}

\begin{proof}
	Since the analytification is defined as a left Kan extension, we have a natural equivalence
	\[ (X_Y^\wedge)\an \simeq (X\an)_{Y\an}^\wedge . \]
	We have
	\[ \Perf( X_Y^\wedge \times \Spec(A_U) ) \simeq \lim_{Z \to X} \Perf(Z \times \Spec(A_U)) , \]
	and similarly
	\[ \Perf( (X\an)_{Y\an}^\wedge \times U ) \simeq \lim_{Z\to X} \Perf(Z\an \times U) \]
	Notice that each $Z$ is still a proper derived geometric stack locally almost of finite presentation.
	In the non-archimedean setting, \cref{thm:relative_GAGA} shows that the analytification functor
	\[ \Perf(Z \times \Spec(A_U)) \longrightarrow \Perf(Z\an \times U) \]
	is an equivalence.
	In the \canal case, as always \cref{thm:relative_GAGA} shows that for every compact Stein subset $K \subset U$ and every open Stein neighbourhood $V$ of $K$ the functor
	\[ \Perf(Z \times \Spec(A_V)) \longrightarrow \Perf(Z\an \times V) \]
	is fully faithful.
	In particular, we obtain that
	\[ \Perf( X^\wedge_Y \times \Spec(A_U) ) \longrightarrow \Perf( (X\an)^\wedge_{Y\an} \times U ) \]
	is fully faithful as well.
	In order to prove that
	\[ \fcolim_{K \subset V \subset U} \Perf( X_Y^\wedge \times \Spec(A_V)) \longrightarrow \fcolim_{K\subset V \subset U} \Perf( (X\an)^\wedge_{Y\an} \times V) \]
	is an equivalence in $\Ind(\Cat_\infty^{\mathrm{st}, \otimes})$ it is therefore enough to prove that given $W \Subset V$ a relatively compact Stein open neighbourhood of $K$ in $V$ the restriction functor
	\[ \Perf( (X\an)^\wedge_{Y\an} \times V ) \longrightarrow \Perf( (X\an)^\wedge_{Y\an} \times W ) \]
	factors through $\Perf( X^\wedge_Y \times \Spec(A_W))$.
	This fact follows directly from \cref{prop:relative_GAGA_complex_discrete}.
	Notice that the choice of $W$ does not depend on $Z \in \mathrm{Nil}_{/X}(Y)$.
\end{proof}

\subsubsection{De Rham stacks} \label{eg:derhamstack}

The de Rham stack can be defined in both the algebraic and analytic setting as follows.
Let
\[ j \colon \Aff^{\mathrm{red}}_k \longrightarrow \dAff_k \]
be the natural inclusion.
This is a continuous and cocontinuous morphism of sites with respect to the \'etale topology on both sides.
In particular, the functor
\[ j^s \colon \dSt_k \longrightarrow \St( \Aff_k^{\mathrm{red}}, \tauet ) \]
admits both a left adjoint $j_s$ and a right adjoint ${}_s j$.
We set
\[ (-)_{\mathrm{dR}} \coloneqq {}_s j \circ j^s \quad , \quad (-)_{\mathrm{red}} \coloneqq j_s \circ j^s . \]
It can be shown that when $X = \Spec(A)$ then $X_{\mathrm{red}} \simeq \Spec( \pi_0(A)_{\mathrm{red}} )$, while it is always true that
\[ X_{\mathrm{dR}}( \Spec(A) ) \simeq X( \Spec( \pi_0(A) )_{\mathrm{red}} ) . \]
The same definitions can be carried over in the analytic setting, using $\Afd_k^{\mathrm{red}}$ instead of $\Aff_k^{\mathrm{red}}$.
See \cite[\S 3]{Porta_Derived_Riemann_Hilbert} for the \canal case.

\begin{lem}
	Let $X \in \dSt_k^{\mathrm{afp}}$ be a derived stack locally almost of finite presentation.
	Then there is a canonical map
	\[ (X_{\mathrm{dR}})\an \longrightarrow (X\an)_{\mathrm{dR}} \]
	which is furthermore an equivalence when $X$ is a smooth geometric stack.
\end{lem}

\begin{proof}
	We observe that the analytification functor takes $\Aff_k^{\mathrm{red,afp}}$ to $\Afd_k^{\mathrm{red}}$.
	In particular, the natural transformation $(-)\an \circ (-)_{\mathrm{dR}} \to (-)_{\mathrm{dR}} \circ (-)\an$ is simply a Beck-Chevalley transformation.
	When $X$ is smooth and geometric, we observe that the maps
	\[ X \longrightarrow X_{\mathrm{dR}} \quad , \quad X\an \longrightarrow (X\an)_{\mathrm{dR}} \]
	are effective epimorphism.
	Their \v{C}ech nerves can be identified with the simplicial objects $(X^\bullet)^{\wedge}_X$ and $((X\an)^\bullet)^{\wedge}_{X\an}$ given by the formal completions of $X^n$ and $(X\an)^n$ along the small diagonals.
	The conclusion now follows from \cref{cor:analytification_formal_completions}.
\end{proof}

\begin{prop}
	Let $X$ be a smooth geometric stack locally almost of finite presentation over $k$.
	Then for any $U \in \dAfd_k$ the canonical map
	\[ (-)\an \colon \QCoh( \Spec(A_U) \times X_{\mathrm{dR}} ) \longrightarrow \cO_{U \times X_{\mathrm{dR}}\an} \Mod  \]
	is conservative and $t$-exact.
\end{prop}

\begin{proof}
	Write $X^\bullet / X_{\mathrm{dR}}$ and $(X\an)^\bullet / X_{\mathrm{dR}}\an$ to denote the \v{C}ech nerves of the maps $X \to X_{\mathrm{dR}}$ and $X\an \to X_{\mathrm{dR}}\an$.
	Since $X$ is smooth, we have
	\[ X_{\mathrm{dR}} \simeq | X^\bullet / X_{\mathrm{dR}} | \quad , \quad X_{\mathrm{dR}}\an \simeq | (X\an)^\bullet / X_{\mathrm{dR}}\an | . \]
	In turn, this provides canonical equivalences
	\[ \QCoh( \Spec(A_U) \times X_{\mathrm{dR}} ) \simeq \lim_{\mathbf \Delta} \QCoh( \Spec(A_U) \times ( X^\bullet / X_{\mathrm{dR}} ) ) \quad , \quad \cO_{U \times X_{\mathrm{dR}}\an} \Mod \simeq \lim_{\mathbf \Delta} \cO_{U \times ( (X\an)^\bullet / X_{\mathrm{dR}}\an)} \Mod . \]
	The argument given in \cite[Proposition 5.1]{Porta_Derived_Riemann_Hilbert} shows that we can endow $\QCoh( \Spec(A_U) \times X_{\mathrm{dR}})$ and $\cO_{U \times X_{\mathrm{dR}}\an} \Mod$ with $t$-structures characterized by the fact that the forgetful functors
	\[ \QCoh(\Spec(A_U) \times X_{\mathrm{dR}}) \longrightarrow \QCoh(\Spec(A_U) \times X) \quad , \quad \cO_{U\times X_{\mathrm{dR}}\an} \Mod \longrightarrow \cO_{U \times X} \Mod \]
	are $t$-exact.
	Moreover, the analytification functor is obtained by passing to the limit from the analytification functors
	\[ (-)\an \colon \QCoh( \Spec(A_U) \times (X^\bullet / X_{\mathrm{dR}})) \to \cO_{U \times ( (X\an)^\bullet / X_{\mathrm{dR}}\an )} \Mod . \]
	Since we can identify $X^n / X_{\mathrm{dR}}$ and $(X\an)^n / X_{\mathrm{dR}}\an$ with the formal completion of the small diagonal in $X^n$ and in $(X\an)^n$ respectively, the conclusion now follows from \cref{prop:analytification_formal_completion_conservative_t_exact}.
\end{proof}

\begin{prop} \label{prop:de_Rham_universal_GAGA}
	Let $X$ be a smooth and proper geometric stack locally almost of finite presentation over $k$.
	Then $X_{\mathrm{dR}}$ satisfies the universal GAGA property.
\end{prop}

\begin{proof}
	Let us first assume that $k$ is a non-archimedean field.
	Then thanks to the previous proposition, we only have to check that the canonical map
	\[ \Perf( \Spec(A_U) \times X_{\mathrm{dR}} ) \longrightarrow \Perf( U \times X_{\mathrm{dR}}\an ) \]
	is an equivalence for every $U \in \dAfd_k$.
	Let $X^\bullet / X_{\mathrm{dR}}$ and $(X\an)^\bullet / X_{\mathrm{dR}}\an$ be the \v{C}ech nerves of the maps $X \to X_{\mathrm{dR}}$ and $X\an \to X_{\mathrm{dR}}\an$, respectively.
	Then
	\[ \Perf( \Spec(A_U) \times X_{\mathrm{dR}} ) \simeq \lim_{\mathbf \Delta} \Perf( \Spec(A_U) \times ( X^\bullet / X_{\mathrm{dR}} ) ) \]
	and
	\[ \Perf( U \times X_{\mathrm{dR}}\an ) \simeq \lim_{\mathbf \Delta} \Perf( U \times ((X\an)^\bullet / X_{\mathrm{dR}}\an) ) . \]
	The conclusion now follows directly from \cref{prop:formal_completions_relative_GAGA}.
	
	We now turn to the \canal case.
	Fix therefore a compact Stein subset $K \subset U$.
	Since we can identify both $X^n / X_{\mathrm{dR}}$ and $(X\an)^n / X_{\mathrm{dR}}\an$ with the formal completion of the small diagonal in $X^n$ and in $(X\an)^n$, we can use \cref{prop:formal_completions_relative_GAGA} to deduce that for every open Stein neighbourhood $V$ of $K$ in $U$ the map
	\[ \Perf( \Spec(A_V) \times ( X^n / X_{\mathrm{dR}} ) ) \longrightarrow \Perf( V \times ( (X\an)^n / X_{\mathrm{dR}}\an ) ) \]
	is fully faithful.
	Therefore, to prove that the map
	\[ \fcolim_{K \subset V \subset U} \Perf(\Spec(A_V) \times X_{\mathrm{dR}}) \longrightarrow \fcolim_{K \subset V \subset U} \Perf(V \times X_{\mathrm{dR}}\an) \]
	is an equivalence in $\Ind(\Cat_\infty^{\mathrm{st}, \otimes})$ it is enough to check that if $W \Subset V$ is a relatively compact open Stein neighbourhood of $K$ in $V$ then the restriction functor
	\[ \Perf(V \times X_{\mathrm{dR}}\an) \longrightarrow \Perf( W \times X_{\mathrm{dR}}\an ) \]
	factors through $\Perf(\Spec(A_W) \times X_{\mathrm{dR}}\an)$.
	This follows immediately from \cref{prop:relative_GAGA_complex_discrete}.
\end{proof}

\subsubsection{Betti stacks} \label{eg:constantstack}

The canonical functor $\pi \colon \dAff_k^{\mathrm{afp}} \to \{*\}$ induces an adjunction
\[ \pi_s \colon \dSt_k^{\mathrm{afp}} \leftrightarrows \cS \colon \pi^s .  \]
Given a space $K \in \cS$, we set
\[ K_\rB \coloneqq \pi^s(K) . \]
We refer to $K_\rB$ as the \emph{Betti stack associated to $B$}.
In other words, $K_\rB$ is the sheafification of the constant presheaf with values $K$.
We similarly define $K_\rB\an$ as the constant analytic stack associated to $K$.

\begin{lem}
	There is a canonical equivalence $(K_\rB)\an \simeq K_\rB\an$.
\end{lem}

\begin{proof}
	Let us denote by $\varphi$ the derived analytification functor
	\[ \varphi \coloneqq (-)\an \colon \dAff_k^{\mathrm{afp}} \longrightarrow \dAn_k . \]
	Then $\varphi^s( K_\rB\an ) \simeq \pi^s( K ) \simeq K_\rB$.
	Therefore,
	\[ \Map_{\dAnSt_k}( (K_\rB)\an , K_\rB\an ) \simeq \Map_{\dSt_k^{\mathrm{afp}}}( K_\rB, \varphi^s( K_\rB\an ) ) . \]
	Therefore the identity of $K_\rB$ corresponds to a canonical morphism $(K_\rB)\an \to K_\rB\an$.
	We now observe that this morphism is an equivalence when $K \simeq *$, and moreover both $(-)\an$ and the formation of $K_\rB\an$ commute with arbitrary colimits.
	The conclusion therefore follows.
\end{proof}

\begin{prop} \label{prop:Betti_universal_GAGA}
	Let $K \in \cS$ be a space.
	Then for any $U \in \dAfd_k$, $K_\rB$ satisfies the universal GAGA property.
\end{prop}

\begin{proof}
	We first observe that
	\begin{equation} \label{eq:Betti_t_exact_conservative}
		\QCoh( K_\rB \times \Spec(A_U) ) \simeq \Fun(K, A_U \Mod) \quad , \quad \cO_{K_\rB\an \times U} \Mod \simeq \Fun(K, \cO_U \Mod) .
	\end{equation}
	Moreover, the analytification functor is simply obtained by composition with the analytification functor
	\[ \varepsilon^*_U \colon A_U \Mod \longrightarrow \cO_U \Mod . \]
	As \cref{prop:basic_properties_global_sections} guarantees that $\varepsilon_U^*$ is $t$-exact and conservative, we deduce that the same goes for the functor \eqref{eq:Betti_t_exact_conservative}.
	Next, in the non-archimedean case the equivalence $\Perf(A_U) \simeq \Perf(U)$ immediately implies that the analytification
	\[ \Perf( K_\rB \times \Spec(A_U) ) \longrightarrow \Perf( K_\rB\an \times U ) \]
	is an equivalence.
	In the \canal case, fix a compact Stein subset $K \subset U$.
	Then \cref{prop:relative_analytification_coherent_sheaves_fully_faithful} implies that for each open Stein neighbourhood $V$ of $K$ inside $U$ the canonical map
	\[ \Perf(K_\rB \times \Spec(A_V)) \longrightarrow \Perf( K_\rB\an \times V ) \]
	is fully faithful, while \cref{lem:global_section_relatively_compact} implies that if $W \Subset V$ is a relatively compact open Stein neighbourhood of $K$ inside $V$ then the restriction
	\[ \Perf( K_\rB\an \times V ) \longrightarrow \Perf( K_\rB\an \times W ) \]
	factors through $\Perf( K_\rB \times \Spec(A_W) )$.
	This implies that the canonical map
	\[ \fcolim_{K \subset V \subset U} \Perf( K_\rB \times \Spec(A_V) ) \longrightarrow \fcolim_{K \subset V \subset U} \Perf( K_\rB\an \times V ) \]
	is an equivalence.
\end{proof}

\subsubsection{Dolbeault stacks}

The Dolbeault stack of a derived formal stack $X$ appears in Simpson's nonabelian Hodge theory in dealing with Higgs bundles.
It is defined as follows: let
\[ \rT X \coloneqq \Spec_X( \Sym_{\cO_X}( \mathbb L_X ) ) \]
be the derived tangent bundle to $X$.
Let $\widehat{\rT X}$ be the formal completion of $\rT X$ along the zero section.
Using \cite[4.2.2.9]{Lurie_Higher_algebra} we can convert the natural commutative group structure of $\rT X$ relative to $X$ (seen as an associative one) into a simplicial diagram $\rT^\bullet X \colon \mathbf \Delta\op \to (\dSt_k)_{/X}$.
Unwinding the definitions, we see that $\rT^\bullet X$ can be identified with the $n$-fold product $\rT X \times_X \cdots \times_X \rT X$.
The zero section $X \to \rT X$ allows to promote $\rT^\bullet X$ to a simplicial diagram
\[ \rT^\bullet X \colon \mathbf \Delta\op \longrightarrow (\dSt_k)_{X // X} . \]
Formal completion along the natural maps $X \to \rT^n X$ provides us with a new simplicial object
\[ \widehat{\rT^\bullet X} \colon \mathbf \Delta\op \longrightarrow (\dSt_k)_{/X} . \]

\begin{defin}
	The \emph{Dolbeault stack of $X$} is the geometric realization
	\[ X_{\mathrm{Dol}} \coloneqq \left| \widehat{\rT^\bullet X} \right| \in (\dSt_k)_{/X} . \]
\end{defin}

The Dolbeault stack can be defined directly at the analytical level by the exact same procedure.
We have:

\begin{lem} \label{lem:analytification_Dolbeault_shape}
	Let $X$ be a derived geometric $k$-stack.
	Then there is a natural equivalence $(X_{\mathrm{Dol}})\an \simeq (X\an)_{\mathrm{Dol}}$.
\end{lem}

\begin{proof}
	Using \cite[Theorem 5.21]{Porta_Yu_Representability} we see that $(\mathbb L_X)\an \simeq \anL_{X\an}$.
	From here, the conclusion follows directly from \cref{cor:analytification_formal_completions} and from the fact that the analytification functor $(-)\an \colon \dSt_k^{\mathrm{afp}} \to \dAnSt_k$ commutes with finite limits and arbitrary colimits.
\end{proof}

\begin{prop} \label{prop:Dolbeault_universal_GAGA}
	Let $X$ be a proper derived geometric $k$-stack.
	For any $U \in \dAfd_k$, the Dolbeault stack $X_{\mathrm{Dol}}$ satisfies the universal GAGA property.
\end{prop}

\begin{proof}
	Since the face maps in the simplicial diagram $\widehat{\rT^\bullet X}$ are flat, we deduce directly from \cref{prop:analytification_formal_completion_conservative_t_exact} that the canonical map
	\[ \QCoh( X_{\mathrm{Dol}} \times \Spec(A_U) ) \longrightarrow \cO_{X_{\mathrm{Dol}}\an \times U} \Mod \]
	is conservative and $t$-exact.
	
	In the non-archimedean case, \cref{prop:formal_completions_relative_GAGA} implies that the analytification functor induces an equivalence
	\[ \Perf( \widehat{\rT^n X} \times \Spec(A_U) ) \simeq \Perf( \widehat{\rT^n X\an} \times U)  \]
	for every $U \in \dAfd_k$ and every $n \ge 0$.
	Therefore we deduce that the canonical map
	\[ \Perf( X_{\mathrm{Dol}} \times \Spec(A_U) ) \longrightarrow \Perf( X_{\mathrm{Dol}}\an \times U ) \]
	is an equivalence as well.
	In the \canal case, we deduce from \cref{prop:formal_completions_relative_GAGA} that each map
	\[ \Perf( \widehat{\rT^n X} \times \Spec(A_U) ) \longrightarrow \Perf( \widehat{ \rT^n X\an } \times U ) \]
	is fully faithful, and therefore that for every $U \in \dSt_\C$ the functor
	\[ \Perf( X_{\mathrm{Dol}} \times \Spec(A_U) ) \longrightarrow \Perf( X_{\mathrm{Dol}}\an \times U ) \]
	is fully faithful.
	Let now $K \subset U$ be a compact Stein subset.
	It is enough to prove that if $W \Subset V$ are two open Stein neighbourhoods of $K$ inside $U$, with $W$ relatively compact inside $V$, then the canonical map
	\[ \Perf( X_{\mathrm{Dol}}\an \times V ) \longrightarrow \Perf( X_{\mathrm{Dol}}\an \times W ) \]
	factors through $\Perf( X_{\mathrm{Dol}} \times \Spec(A_W) )$.
	This follows once again from \cref{prop:relative_GAGA_complex_discrete}.
\end{proof}

\subsubsection{Classifying stack of a complex reductive group} \label{eg:reductive_strong_GAGA}

In the previous sections we discussed several examples of derived stacks satisfying the universal GAGA property.
All the examples we considered so far are consequences of the analysis carried out in order to prove \cref{thm:relative_GAGA}.
We now consider a different kind of example.

Let $G$ be a connected reductive group over $\mathbb C$.
Then $\rB G$ is a smooth geometric stack, but it is not proper in the sense of \cite[Definition 4.8]{Porta_Yu_Higher_analytic_stacks_2014}.
We nevertheless can prove the following result:

\begin{prop}
	If $G$ is a connected reductive group over $\mathbb C$, then $\rB G$ satisfies the GAGA property.
\end{prop}

\begin{proof}
	Let us start by remarking that since the analytification functor commutes with colimits, we have a canonical equivalence $(\rB G)\an \simeq \rB( G\an )$.
	We will therefore use the notation $\rB G\an$, since no confusion can arise.
	Next, we observe that the argument given in \cref{eg:Artin_strong_GAGA} shows that the canonical map
	\[ \QCoh(\rB G) \longrightarrow \cO_{\rB G\an} \Mod \]
	is $t$-exact and conservative.
	All we are left to check is therefore that the canonical functor
	\[ \Perf(\rB G) \longrightarrow \Perf(\rB G\an) \]
	is an equivalence.
	We will prove more generally that the morphism
	\[ \Coh( \rB G ) \longrightarrow \Coh( \rB G\an ) \]
	is an equivalence of stable $\infty$-categories.
	Since both sides are equipped with complete $t$-structures and the functor between them is $t$-exact, we reduce ourselves to the following two statements:
	\begin{enumerate}
		\item The analytification functor
		\[ \Cohh(\rB G) \longrightarrow \Coh( \rB G\an ) \]
		is fully faithful.
		\item The functor
		\[ \Cohh(\rB G) \longrightarrow \Cohh(\rB G\an) \]
		is essentially surjective.
	\end{enumerate}
	Notice that the second statement is entirely classical and it is, in fact, well known.
	In the case where $G$ is semi-simple, it is proven for instance in \cite[Corollary 15.8.7]{Taylor_Several_complex}.
	For tori it is well known.
	Finally, a general reductive group $G$ admits a finite cover by a product of a torus and a semi-simple group, and from here it is straightforward to obtain the statement for $G$.
	
	It is therefore enough to prove the first statement.
	Since we work over $\mathbb C$, \cite{Hall_Rydh_Compact_Generation} implies that $\QCoh(\rB G)$ is compactly generated and in particular that
	\[ \QCoh(\rB G) \simeq \mathrm D( \QCoh^\heartsuit(\rB G) ) . \]
	Since $G$ is reductive, for $M, N \in \QCoh^\heartsuit(\rB G)$, the mapping space $\Map_{\QCoh(\rB G)}(M, N)$ is discrete and coincides with the hom set $\Hom_G(M,N)$.
	We let $M\an$ and $N\an$ denote the analytic representations of $G\an$ associated to $M$ and $N$.
	Since $M$ and $N$ are coherent, we have an equivalence
	\[ \Hom_G(M, N) \simeq \Hom_{G\an}(M\an, N\an) \simeq \pi_0 \Map_{\QCoh(\rB G\an)}(M\an, N\an) , \]
	which readily follows from \cite[Corollary 15.8.7]{Taylor_Several_complex}.
	In order to complete the proof, we have to check that $\pi_i \Map_{\QCoh( \rB G\an)}( M\an, N\an) \simeq 0$ for $i \ne 0$.
	We denote by $\cH(G\an)$ the category of holomorphic representations of $G\an$ on topological vector spaces with continuous $G\an$-invariant maps between them.
	We now invoke the results of \cite{Hochschild_Mostow_Holomorphic}.
	Since $G$ is reductive, it has a maximal compact subgroup, which is ample in the sense of loc.\ cit.
	Moreover, any finite dimensional representation of $G\an$ is complete and locally convex.
	Therefore, Proposition 2.3 in loc.\ cit.\ implies that $N\an$ is holomorphically injective, and therefore that
	\[ \Hom_{G\an}(M\an, N\an) \simeq \Map_{\mathrm D(\cH(G\an))}( M\an, N\an ) . \]
	
	We are now reduced to prove that there is an equivalence
	\[ \Map_{\mathrm D(\cH(G\an))}( M\an, N\an) \simeq \Map_{\cO_{\rB G\an \Mod}}(M\an, N\an) . \]
	We will use the completed tensor product $\hat \otimes$ of locally convex spaces. 
	As we are working with global sections over Stein spaces the locally convex spaces are nuclear, see \cite[Proposition IX.5.18]{Demailly_Complex_analytic}, and there is no distinction between projective and injective tensor product.
	In particular $\hat \otimes$ preserves subspaces, see \cite[Proposition IX.5.6]{Demailly_Complex_analytic}.
	
	Let $(G\an)^\bullet$ be the \v{C}ech nerve of the map $p \colon \Sp(\C) \to \rB G\an$.
	We can compute the Exts in $\cO_{\rB G\an} \Mod$ by means of the equivalence
	\[ \cO_{\rB G\an} \Mod \simeq \lim_{[n] \in \mathbf \Delta} \cO_{(G\an)^{\times n}} \Mod . \]
	Write $q_n \colon (G\an)^{\times n} \to \Sp(\C)$ and $p_n \colon (G\an)^{\times n} \to \rB G$ for the standard projection maps.
	We have canonical identifications $p_n \simeq p \circ q_n$.
	The cosimplicial object computing $\Map_{\rB G\an}(M\an, N\an)$ has in degree $n$
	\[ \Map_{\Cohh((G\an)^{\times n})}(p_n^* M\an, p_n^* N\an) \simeq \Map_{\C}( p^* M\an, q_{n*} p_n^* N\an ) . \]
	As $(G\an)^{\times n}$ is Stein and $p^* M\an$ and $p^* N\an$ are globally finitely generated in the sense of \cite[Lemma 8.11]{Porta_Yu_Higher_analytic_stacks_2014}, we have an equivalence
	\[ \Hom_{\Coh((G\an)^{\times n})}( p_n^* M\an, p_n^* N\an) \simeq \Hom_{\cO( (G\an)^{\times n})}^{\mathrm{cts}}( q_{n*} p_n^* M\an, q_{n*} p_n^* N\an ) . \]
	Here the superscript $\mathrm{cts}$ denotes the subset of continuous maps for the unique complete topology on the global sections of a coherent sheaf over a Stein space.
	We further have the following equivalence:
	\[ \Hom_{\cO( (G\an)^{\times n})}^{\mathrm{cts}}( q_{n*} p_n^* M\an, q_{n*} p_n^* N\an ) \simeq \Hom_\C^{\mathrm{cts}}( p^* M\an, p^* N\an \cotimes \cO((G\an)^{\times n}) ) . \]
	We claim that there is a further isomorphism
	\[ \Hom_\C^{\mathrm{cts}}( p^* M\an, p^* N\an \cotimes \cO((G\an)^{\times n}) ) \simeq \Hom_{\cH(G\an)}(M\an , N\an \cotimes \mathcal O((G\an)^{\times n+1})) . \]
	This would follow if we could show that the forgetful functor from $\mathcal H(G\an)$ to topological vector spaces has a right adjoint given by $-\hat \otimes \mathcal O(G\an)$, equivalently that pushforward from the point to $\rB G\an$ is given by $-\hat \otimes \mathcal O(G\an)$. 
	Unfortunately this situation does not seem to be treated in the literature and leads beyond the scope of this article. 
	We will prove a weaker statement that is sufficient for our purposes, using the fact that $M$ is finite dimensional.
	
	Firstly, observe that $\Hom_{\mathcal H(G\an)}(M, \mathcal O(G\an)) \cong \Hom(M,\C)$. 
	There is clearly an embedding of the right hand side into the left hand side as $\mathcal O(G\an)$ contains the regular functions on $G$. 
	But this embedding is surjective as any $G$-equivariant map from $M$ to $\mathcal O(G\an)$ is determined by the image of a basis of $M$ evaluated at the identity, i.e.\ the space of such maps has dimension $\dim M$.
	
	Now we consider $\Hom_{\mathcal H(G\an)}(M, V \hat \otimes \mathcal O(G\an))$ where we write $V$ for the trivial $G$-representation $N \hat \otimes \mathcal O(G\an)^{\times n}$. 
	Any function from $M$ factors through some $V_{i} \hat\otimes \mathcal O(G\an)$, where $V_{i}$ is a finite dimensional subspace of $V$. 
	Thus we need to compute $\Hom_{\mathcal H(G\an)}(M, \colim V_{i} \hat \otimes \mathcal O(G\an))$ where we take the colimit over all finite-dimensional subspaces.
	Now Hom out of $M$ commutes with filtered colimits over admissible maps and we are reduced to showing that $\Hom_{\mathcal H(G\an)}(M, V' \otimes \mathcal O(G\an)) \cong \Hom(M, V')$ for $V'$ finite-dimensional, which readily follows from the case $V'$ that is 1-dimensional. 
	Then $\Hom_{\mathcal H(G\an)}(M, \colim V_{i} \otimes \mathcal O(G\an)) \cong \Hom^{\mathrm{cts}}(M, \colim V_{i}) \cong \Hom^{cts}(M, V)$.
	This proves the claim.
	
	Now we have a complex of holomorphic representations of $G\an$ which computes cohomology in $\mathcal H(G\an)$. 
	The complex $(\mathcal O(G\an)^{\otimes \bullet+1} \otimes N)$ is quasi-isomorphic to $N$ by the usual bar complex arguments, and by \cite[Proposition 2.1]{Hochschild_Mostow_Holomorphic} it is levelwise injective, thus injective as it is bounded below. 
	We note that the $G\an$ action on $\mathcal O(G\an)^{\otimes n+1}$ is on the last factor since $(q_{n})_{*}p_{n}^{*}\mathcal N \cong \mathcal O(G\an)^{\otimes n} \otimes \mathcal N)$ has the trivial $G\an$-action as $p^{*}\mathcal N$ forgets the $G\an$-action.
	This gives $\Hom_{\mathcal H(G\an)}(M, N\hat \otimes \mathcal O(G\an)^{\times \bullet + 1}) \simeq \Hom_{\mathcal H(G\an)}(M, N)$
	
	This shows that the \v Cech complex computation recovers the (trivial) holomorphic group cohomology of $G\an$.
\end{proof}

\begin{rem}
	In the above example we used in an essential way that $G$ is a reductive group over the field of complex numbers.
	We do not know what happens if $G$ is a reductive group over a non-archimedean field, but it would be interesting to know if the same property holds.
\end{rem}

\section{Tannaka duality} \label{sec:Tannaka_duality}

In this section we prove the main theorem of this paper.
Our goal is to find sufficient conditions on a geometric derived stack $Y$ and a (not necessarily geometric) derived stack $X$ so that the canonical morphism
\[ \bfMap_{\dSt_k}(X,Y)\an \longrightarrow \bfMap_{\dAnSt_k}(X\an, Y\an) \]
is an equivalence.
When both $X$ and $Y$ are proper underived schemes, Serre's GAGA theorem and a simple graph argument imply that
\[ \Hom_{\mathrm{Sch}_k}(X,Y) \simeq \Hom_{\An_k}(X\an, Y\an) . \]
The relative version of the GAGA theorem implies that the same holds true for the hom schemes.
Unfortunately, this argument breaks down when $Y$ is no longer proper, or when $Y$ is taken to be a stack (in both cases, it is the graph argument which fails).
The idea to fix this problem was first introduced by Lurie \cite{Lurie_Tannaka_duality}. He allows $X$ to be a proper \DM stack over $\C$ and $Y$ to be a geometric stack satisfying several conditions making tannakian reconstruction for $Y$ possible.
Lurie contents himself with proving that under these assumptions the canonical map
\[ \Map_{\St_k}(X,Y) \longrightarrow \Map_{\mathrm{AnSt}_k}(X\an, Y\an) \]
is an equivalence.
Our goal is to generalize this result in several directions: Firstly, we aim to prove a relative version of the result, i.e.\ we consider mapping stacks rather than just mapping spaces. Secondly, we want to allow $X$ to be a more general object than a \DM stack.
For instance, in \cref{sec:applications} we will be interested in the situation where $X$ is $S_{\mathrm{dR}}$ for $S$ a smooth and proper $k$-scheme.
Finally, we want to allow $X$, $Y$ and the mapping stacks to be derived.

The general strategy for the proof of our main theorem is the same as the one employed in \cite{Lurie_Tannaka_duality}.
However, the technical tools used in loc.\ cit.\ needed to be generalized and sharpened in order to apply to the situations we are concerned with.
These improved tools also form the other main theorems of this paper.
Notably, we are referring to Theorems \ref{thm:generalized_adjunction}, \ref{cor:perfect_complexes_compact_Stein} and \ref{thm:relative_GAGA}.

\subsection{Analytification and Tannaka duality}

We start this section by briefly reviewing the notion and the machinery of Tannaka duality.
The main references for the algebraic theory are \cite{DAG-VIII} and \cite[\S III.9]{Lurie_SAG}.
Our goal is to study how the Tannaka property interacts with the analytification functor.
Recall that we have an $\infty$-functor
\[ \Perf_k \colon \dAff_k\op \longrightarrow \Cat_\infty^{\mathrm{st}, \otimes} \]
with values in stably symmetric monoidal $\infty$-categories that sends $\Spec(A)$ to the $\infty$-category $\Perf(A)$, equipped with its canonical symmetric monoidal structure.
We also have at our disposal an $\infty$-functor
\[ \QCoh_k \colon \dAff_k\op \longrightarrow \Cat_\infty^{\mathrm{st}, \otimes} \]
sending $\Spec(A)$ to the $\infty$-category $\QCoh(A) \simeq A\Mod$ equipped with its natural symmetric monoidal structure.
Given stably symmetric monoidal $\infty$-categories $\cC$ and $\cD$, we denote by $\Fun^\otimes_{\mathrm{ex}}(\cC, \cD)$ the $\infty$-category of symmetric monoidal exact functors from $\cC$ to $\cD$.
Recall that both $\Perf_k$ and $\QCoh_k$ satisfy \'etale descent.
In particular, they extend to functors
\[ \Perf_k, \QCoh_k \colon \dSt_k\op \longrightarrow \Cat_\infty^{\mathrm{st}, \otimes} , \]
and in particular for every pair of derived stacks locally almost of finite presentation, $X$ and $Y$, we obtain morphisms
\[ P \colon \Map_{\dSt_k}(X,Y) \longrightarrow \Fun^\otimes(\Perf(Y), \Perf(X)) \]
and
\[ \widehat{P} \colon \Map_{\dSt_k}(X,Y) \longrightarrow \Fun^\otimes(\QCoh(Y), \QCoh(X)) . \]

\begin{defin} \label{def:tannakian_stacks}
	We say a derived stack $Y \colon \dAff_k\op \to \cS$ is \emph{weakly tannakian} (or that $Y$ \emph{satisfies the weak Tannaka property}) if it satisfies the following condition:
	\begin{enumerate}
		\item for any derived stack $X \colon \dAff_k\op \to \cS$, the $\infty$-functor
			\[ \widehat{P} \colon \Map_{\mathrm{dSt}_k}(X,Y) \longrightarrow \Fun^\otimes(\QCoh(Y), \QCoh(X)) \]
		is fully faithful.
	\end{enumerate}
	We say that a derived stack $Y \colon \dAff_k\op \to \cS$ is \emph{tannakian} (or that $Y$ \emph{satisfies the Tannaka property}) if it is weakly tannakian and it satisfies the following supplementary condition:
	\begin{enumerate} \setcounter{enumi}{1}
		\item the essential image of $\widehat{P}$ consists of exact symmetric monoidal functors $\QCoh(Y) \to \QCoh(X)$ that commute with colimits and preserve both connective objects and flat objects.
	\end{enumerate}
\end{defin}

Notice that a symmetric monoidal functor $F \colon \QCoh(Y) \to \QCoh(X)$ preserves dualizable objects.
As $\Perf(Y)$ and $\Perf(X)$ coincide with the full subcategories of $\QCoh(Y)$ and $\QCoh(X)$ spanned exactly by dualizable objects, we conclude that each such functor gives rise to a functor $F \colon \Perf(Y) \to \Perf(X)$, which is again symmetric monoidal.
This argument also shows that the inclusion $\Perf(X) \hookrightarrow \QCoh(X)$ induces an equivalence
\[ \Fun^\otimes(\Perf(Y), \Perf(X)) \stackrel{\sim}{\longrightarrow} \Fun^\otimes(\Perf(Y), \QCoh(X)) . \]

\begin{lem}
	Suppose that $\QCoh(Y)$ is compactly generated by perfect complexes, i.e.\ $\QCoh(Y) \simeq \Ind(\Perf(Y))$.
	Then $Y$ is weakly tannakian if and only if the functor $P \colon \Map_{\dSt_k}(X,Y) \to \Fun^\otimes(\Perf(Y), \Perf(X))$ is fully faithful.
\end{lem}

\begin{proof}
	Under the assumption $\QCoh(Y) \simeq \Ind(\Perf(Y))$, we have
	\[ \Fun^\otimes(\Perf(Y), \Perf(X)) \simeq \Fun^\otimes(\Perf(Y), \QCoh(X)) \simeq \Fun^{\otimes}_L(\QCoh(Y), \QCoh(X)) , \]
	where the subscript $L$ denotes the full subcategory spanned by those functors that commute with arbitrary colimits.
	The conclusion follows.
\end{proof}

\begin{eg}
	\cite[Theorem 9.3.0.3]{Lurie_SAG} guarantees that if $X$ is a geometric stack with affine diagonal then it is tannakian.
\end{eg}

\Cref{def:tannakian_stacks} is difficult to export to the analytic setting.
The main obstruction is that there is no truly satisfactory notion of quasi-coherent sheaf in the analytic setting.
However, perfect complexes make perfect sense, and in particular for every pair of derived analytic stacks $X, Y \colon \dAfd_k\op \to \cS$ we have a natural map
\[ P \colon \Map_{\dAnSt_k}(X,Y) \longrightarrow \Fun^\otimes(\Perf(Y), \Perf(X)) . \]
In virtue of the above lemma we might be tempted to use the map $P$ to define at least the weak tannakian property in the analytic setting.
However, this would still not be a satisfactory notion: indeed, the condition $\QCoh(Y) \simeq \Ind(\Perf(Y))$ is a rather strong one, which in particular implies that $\Perf(Y)$ is a saturated stable $\infty$-category in the sense of \cite{Toen_Algebrisation_2008}. However, Theorem 1.1 in loc.\ cit.\ implies that an analytic space $X$ is algebraizable if and only if $\Perf(X)$ is saturated.
However, in the case where $Y$ actually comes from analytification, we can prove the following result:

\begin{thm}\label{thm:refined_tannakian}
	Let $Y \in \dSt_k\afp$ be a derived stack locally almost of finite presentation.
	Suppose that:
	\begin{enumerate}
		\item the stable $\infty$-category $\QCoh(Y)$ satisfies $\QCoh(Y) \simeq \Ind(\Perf(Y))$;
		\item $Y$ is tannakian.
	\end{enumerate}
	Then for every $X \in \dAnSt_k$ the composition
	\begin{equation} \label{eq:refined_tannakian}
		\Map_{\dAnSt_k}( X, Y\an ) \stackrel{P}{\longrightarrow} \Fun^\otimes( \Perf(Y\an), \Perf(X) ) \longrightarrow \Fun^\otimes(\Perf(Y), \Perf(X))
	\end{equation}
	is fully faithful.
	Here the second map is the one induced by the analytification functor $\Perf(Y) \to \Perf(Y\an)$.
\end{thm}

\begin{proof}
	We adapt the strategy of \cref{prop:general_strategy_analytification} in the current setting.	
	Let us denote the composite functor \eqref{eq:refined_tannakian} by
	\[ \tau_{X,Y} \colon \Map_{\dAnSt_k}(X, Y\an) \longrightarrow \Fun^\otimes(\Perf(Y), \Perf(X)) . \]
	This map is functorial in both $X$ and $Y$.
	Notice that left and right hand sides commute separately with colimits in $X \in \dAnSt_k$.
	We can therefore reduce ourselves to the case where $X$ is a derived $k$-affinoid (resp.\ Stein) space to begin with.
	Introduce the presheaf
	\[ T_Y \colon \dAfd_X\op \longrightarrow \Cat_\infty \]
	sending a map $U \to X$ to the $\infty$-category $\Fun^\otimes(\Perf(Y), \Perf(U))$.
	Notice that $T_Y$ is in fact a sheaf, and the maps $\tau_{U,Y}$ assemble into a natural transformation
	\[ \tau_Y \colon F^s_X(Y\an) \longrightarrow T_Y . \]
	Let $\overline{\tau}_Y$ denote the composition
	\[ \overline{\tau}_Y \colon G^p_X(Y) \longrightarrow G^s_X(Y) \stackrel{\alpha_Y}{\longrightarrow} F^s_X(Y) \stackrel{\tau_Y}{\longrightarrow} T_Y , \]
	where $\alpha_Y$ is the map from \cref{thm:generalized_adjunction}.
	Now recall that the $\infty$-topos underlying $X$ has enough points.
	Since full faithfulness can be tested on stalks and since $G^s_X(Y)$ is the sheafification of $G^p_X(Y)$ (hence they have the same stalks), we are reduced to proving that $\overline{\tau}_Y$ is fully faithful.
	
	We start dealing with the \kanal case.
	Fix $U \in \dAfd_X$ and let $A_U \coloneqq \Gamma(U; \cO_U\alg)$.
	Then, unwinding the definitions, we have
	\[ G^p_X(Y)(U) \simeq \Map_{\dSt_k}(\Spec(A_U), Y) . \]
	Since $Y$ is tannakian and $\QCoh(Y) \simeq \Ind(\Perf(Y))$, the canonical map
	\[ P \colon \Map_{\dSt_k}(\Spec(A_U), Y) \longrightarrow \Fun^\otimes(\Perf(Y), \Perf(A_U)) \]
	is fully faithful.
	On the other hand,
	\[ T_Y(U) \simeq \Fun^\otimes(\Perf(Y), \Perf(U)) . \]
	Using \cref{lem:perfect_complexes_affinoid}, we see that the global section functor $\Gamma(U;-)$ induces an equivalence
	\[ \Perf(U) \stackrel{\sim}{\longrightarrow} \Perf(A_U) , \]
	whence the full faithfulness of $\overline{\tau}_Y$.
	
	We now turn to the \canal case.
	Using the correspondence provided by \cref{thm:sheaves_compact_subsets} we extend both $G^p_X(Y)$ and $T_Y$ to compact subsets of $X$.
	We let once again $\overline{\tau}_Y$ denote the induced natural transformation between them.
	Using \cref{lem:sheaves_open_sheaves_compact} we see that it is enough to prove that for every compact Stein subset $K \subset X$ of $X$, the natural map
	\[ \overline{\tau}_{Y, (K)_X} \colon G^p_X(Y)((K)_X) \longrightarrow T_Y((K)_X) \]
	is fully faithful.
	Unravelling the definitions, we see that we have to prove that the functor
	\begin{equation} \label{eq:refined_tannakian_key_point}
		\colim_{K \subset U \subset X} \Map_{\dSt_k}(\Spec(A_U), Y) \longrightarrow \colim_{K\subset U \subset X} \Fun^\otimes(\Perf(Y), \Perf(U))
	\end{equation}
	is fully faithful.
	Here the colimit is taken over all Stein open neighbourhoods $U$ of $K$ inside $X$.
	We start dealing with the left hand side.
	Since $Y$ is tannakian and $\QCoh(Y) \simeq \Ind(\Perf(Y))$, we have fully faithful embeddings
	\[ \Map_{\dSt_k}(\Spec(A_U), Y) \longhookrightarrow \Fun^\otimes(\Perf(Y), \Perf(A_U)) . \]
	Let $\Fun^\otimes_{\mathrm{ind}}$ denote the mapping space in the $\infty$-category $\Ind(\Cat_\infty^{\mathrm{st}, \otimes})$.
	Then we have a tautological equivalence
	\[ \colim_{K \subset U \subset X} \Fun^\otimes(\Perf(Y), \Perf(A_U)) \simeq \Fun^\otimes_{\mathrm{ind}}\Big( \Perf(Y), \fcolim_{K \subset U \subset X} \Perf(A_U) \Big) . \]
	On the other hand, we also have
	\[ \colim_{K \subset U \subset X} \Fun^\otimes(\Perf(Y), \Perf(U)) \simeq \Fun^\otimes_{\mathrm{ind}} \Big(\Perf(Y), \fcolim_{K \subset U \subset X} \Perf(U) \Big) . \]
	Notice that the functor $\Gamma_{(K)}$ of \cref{thm:unbounded_coherent_pro_compact_Stein} induces an equivalence in $\Ind(\Cat_\infty^{\mathrm{st}, \otimes})$
	\[ \Gamma_{(K)} \colon \fcolim_{K \subset U \subset X} \Perf(U) \stackrel{\sim}{\longrightarrow} \fcolim_{K \subset U \subset X} \Perf(A_U) , \]
	making the diagram
	\[ \begin{tikzcd}
		\colim_{K \subset U \subset X} \Map_{\dSt_\C}(\Spec(A_U), Y) \arrow{r} \arrow{d} & \colim_{K \subset U \subset X} \Map_{\dAnSt_\C}(U, Y\an) \arrow{d} \\
		\Fun^\otimes_{\mathrm{ind}}\Big(\Perf(Y), \fcolim_{K \subset U \subset X} \Perf(A_U) \Big) \arrow{r}{\Gamma_{(K)}} & \Fun^\otimes_{\mathrm{ind}}\Big(\Perf(Y), \fcolim_{K \subset U \subset X} \Perf(U) \Big)
	\end{tikzcd} \]
	commutative.
	Notice that the composition of the right vertical functor with the top horizontal one coincides with \eqref{eq:refined_tannakian_key_point}.
	As the left vertical arrow is fully faithful and the bottom horizontal one is an equivalence, we can therefore conclude that \eqref{eq:refined_tannakian_key_point} is fully faithful, thus completing the proof.
\end{proof}

\begin{cor}
	Let $Y \in \dSt_k\afp$ be a derived stack locally almost of finite presentation.
	Suppose that
	\begin{enumerate}
		\item the stable $\infty$-category $\QCoh(Y)$ is compactly generated by perfect complexes, i.e.\ $\QCoh(Y) \simeq \Ind(\Perf(Y))$;
		\item $Y$ satisfies the GAGA property (cf.\ \cref{defin:strong_GAGA}(1));
		\item $Y$ is tannakian.
	\end{enumerate}
	Then for every derived analytic stack $X \colon \dAfd_k\op \to \cS$ the functor
	\[ P \colon \Map_{\dAnSt_k}(X,Y\an) \longrightarrow \Fun^{\otimes}(\Perf(Y\an), \Perf(X)) \]
	is fully faithful.
\end{cor}

\begin{proof}
	Since $Y$ satisfies the GAGA property, we see that the analytification functor
	\[ \Perf(Y) \longrightarrow \Perf(Y\an) \]
	is an equivalence of $\infty$-categories.
	In particular, we are reduced to check that the composition
	\[ \Map_{\dAnSt_{k}}(X, Y\an) \stackrel{P}{\longrightarrow} \Fun^\otimes(\Perf(Y\an), \Perf(X)) \longrightarrow \Fun^\otimes(\Perf(Y), \Perf(X)) \]
	is fully faithful.
	We have shown in \cref{thm:refined_tannakian} that this is true even without the assumption that $Y$ satisfies the GAGA property.
\end{proof}

In the proof of \cref{thm:tannakian_target} we will need some control on the essential image of the functor
\[ \Map_{\dAnSt_k}(X, Y\an) \longrightarrow \Fun^\otimes(\Perf(Y), \Perf(X)) \]
that we just proved is fully faithful.
Already in the algebraic case it is unreasonable to expect a characterization of the essential image (unless more restrictive hypotheses are formulated on $Y$, see \cite[Theorem 2.1]{Bhatt_algebraization_2014}).
Under the assumption $\QCoh(Y) \simeq \Ind(\Perf(Y))$, we have an equivalence
\[ \Fun^\otimes_L(\Perf(Y), \cO_X \Mod) \simeq \Fun^\otimes_L(\QCoh(Y), \cO_X \Mod) . \]
It is then much more reasonable to expect to be able to characterize the essential image of the fully faithful functor
\[ \Map_{\dAnSt_k}(X, Y\an) \longrightarrow \Fun^\otimes_L(\QCoh(Y), \cO_X \Mod) . \]
As usual, there is a difference between the \kanal case and the \canal one: in the former, we do obtain a characterization of the above functor whenever $X \in \dAfd_k$; in the latter however we are forced to replace $X$ by a compact Stein subspace.\\

We start with the following definition:

\begin{defin}
	Let $X \in \dAfd_k$ be a derived $k$-affinoid (resp.\ Stein) space.
	We say that an object $\cF \in \cO_X \Mod$ is \emph{flat} if for every $\cG \in \Cohh(X)$ the tensor product $\cF \otimes \cG$ belongs to $\cO_X \Mod^\heartsuit$.
\end{defin}

\begin{lem} \label{lem:refined_tannakian_image_necessary_conditions}
	Let $Y \in \dSt_k^{\mathrm{afp}}$ be a derived stack locally almost of finite presentation.
	Suppose that $Y$ satisfies the same assumptions as in \cref{thm:refined_tannakian} and that it is moreover geometric.
	Then for every derived $k$-affinoid (resp.\ Stein) space $X \in \dAfd_k$, the essential image of the functor
	\[ \Map_{\dAnSt_k}(X, Y\an) \longrightarrow \Fun^\otimes_L(\QCoh(Y), \cO_X \Mod) \]
	factors through the full subcategory spanned by those functors $F \in \Fun^\otimes_L(\QCoh(Y), \cO_X\Mod)$ preserving flat objects, connective objects and taking the full subcategory $\Perf(Y)$ of $\QCoh(Y)$ to $\Perf(X)$.
\end{lem}

\begin{proof}
	Let $f \colon X \to Y\an$ be a given morphism.
	Then its image $F$ in $\Fun^\otimes_L(\QCoh(Y), \cO_X\Mod)$ is obtained by extending by colimits along $\Perf(Y) \hookrightarrow \QCoh(Y)$ the composition
	\[ \Perf(Y) \longrightarrow \Perf(Y\an) \stackrel{f^*}{\longrightarrow} \Perf(X) \longrightarrow \cO_X \Mod . \]
	Since $\Perf(Y) \hookrightarrow \QCoh(Y)$ is fully faithful, we see that $F$ takes $\Perf(Y)$ to $\Perf(X)$ by construction.
	
	For the other statements we proceed by induction on the geometric level of $Y$.
	Suppose first that $Y = \Spec(A)$ is affine.
	Let $\cF \in \QCoh(Y)$ be a flat object.
	In this case, Lazard's theorem \cite[7.2.2.15(1)]{Lurie_Higher_algebra} implies that $\cF$ can be written as filtered colimit of finitely generated free $A$-modules.
	In particular, $f^*(\cF)$ can also be written as filtered colimit of free $\cO_X$-modules, and hence it is flat.
	Similarly, \cite[1.4.4.11]{Lurie_Higher_algebra} implies that $\QCoh(Y)^{\ge 0}$ coincides with the smallest full subcategory of $\QCoh(Y)$ closed under colimits and extensions and containing $A$.
	Since the functor $f^* \colon \QCoh(Y) \to \cO_X \Mod$ is exact and commutes with filtered colimits, it commutes with arbitrary colimits.
	The conclusion now follows from the fact that connective objects in $\cO_X \Mod$ are stable under colimits.
	
	Assume now that the statements have been proven for $n$-geometric derived stacks and let $Y$ be an $(n+1)$-geometric derived stack.
	Let $u \colon U \to Y$ be a smooth atlas and let $U_\bullet$ be its \v{C}ech nerve.
	Then
	\[ | U_\bullet\an | \simeq Y\an . \]
	Given $f \colon X \to Y\an$ we therefore see that, up to a cover $V \to X$, we can suppose that $f$ factors through $U\an$:
	\[ \begin{tikzcd}
		{} & & U\an \arrow{d}{u\an} \\
		V \arrow{r} \arrow[dashed, bend left = 15pt]{urr}{g} & X \arrow{r}{f} & Y\an .
	\end{tikzcd} \]
	Since we can check that $f^*$ commutes with flat objects and connective objects locally on $X$, we can assume from the very beginning that $f$ factors as $f \simeq u\an \circ g$.
	The inductive hypothesis guarantees that $g^*$ commutes with flat objects and connective objects.
	The same is true for $(u\an)^*$ because $u$ is a smooth atlas and the analytification functor $\QCoh(U) \to \cO_{U\an}\Mod$ commutes with flat objects and connective objects.
	Therefore, the conclusion follows.
\end{proof}

Our goal is to prove that the converse to \cref{lem:refined_tannakian_image_necessary_conditions} holds.
We start by dealing with the non-archimedean case, where the converse holds literally:

\begin{prop} \label{thm:refined_tannakian_image_kanal}
	Let $k$ be a non-archimedean field and let $Y \in \dSt_k^{\mathrm{afp}}$ be a derived $k$-stack locally almost of finite presentation.
	Suppose that $Y$ satisfies the same assumptions as in \cref{thm:refined_tannakian}.
	Then for every derived $k$-affinoid space $X \in \dAfd_k$ the essential image of the functor
	\[ \Map_{\dAnSt_k}(X, Y\an) \longrightarrow \Fun^\otimes_L(\QCoh(Y), \cO_X \Mod) \]
	contains all those functors that preserve flat objects, connective objects and take the full subcategory $\Perf(Y)$ of $\QCoh(Y)$ to $\Perf(X)$.
\end{prop}

\begin{proof}
	Let $A \coloneqq \Gamma(X; \cO_X\alg)$.
	\Cref{lem:perfect_complexes_affinoid} provides us with a canonical equivalence
	\[ \Perf(X) \simeq \Perf(A) . \]
	Let $F \in \Fun^\otimes_L(\QCoh(Y), \cO_X \Mod)$ be a functor satisfying the conditions in the statement of the proposition.
	By assumption, $F$ restricts to a symmetric monoidal functor
	\[ \overline{F} \colon \Perf(Y) \longrightarrow \Perf(X) . \]
	Using the above equivalence, we can redefine $\overline{F}$ as a functor
	\[ \overline{F} \colon \Perf(Y) \longrightarrow \Perf(A) . \]
	Consider the extension
	\[ \widetilde{F} \colon \QCoh(Y) \longrightarrow A \Mod . \]
	The functor
	\[ \varepsilon_X^* \colon A \Mod \longrightarrow \cO_X \Mod \]
	which we defined in \cref{subsec:analytification_Perf} commutes with filtered colimits.
	As a consequence, we can identify the composition
	\[ \begin{tikzcd}
		\QCoh(Y) \arrow{r}{\widetilde{F}} & A \Mod \arrow{r}{\varepsilon_X^*} & \cO_X \Mod
	\end{tikzcd} \]
	with the original functor $F$.
	Since $\varepsilon^*_X$ is conservative, strong monoidal, $t$-exact and preserves flat and coherent objects, we see that $\widetilde{F}$ preserves flat objects and connective objects.
	Therefore we can use the tannakian property of $Y$ to see that $\widetilde{F}$ comes from a map
	\[ \widetilde{f} \colon \Spec(A) \longrightarrow Y . \]
	Let $f \colon X \to Y\an$ be the image of $\widetilde{f}$ via the canonical map
	\[ \Map_{\dSt_k}(\Spec(A),Y) = G^p_X(Y)(X) \longrightarrow F^s_X(Y\an) = \Map_{\dAnSt_k}(X, Y\an) . \]
	Since the diagram
	\[ \begin{tikzcd}
		\Map_{\dSt_k}(\Spec(A), Y) \arrow{r} \arrow{d} & \Map_{\dAnSt_k}(X, Y\an) \arrow{d} \\
		\Fun^\otimes(\Perf(Y), \Perf(A)) \arrow{r} \arrow{d} & \Fun^\otimes(\Perf(Y), \Perf(X)) \arrow{d} \\
		\Fun^\otimes_L(\QCoh(Y), A\Mod) \arrow{r} & \Fun^\otimes_L(\QCoh(Y), \cO_X \Mod)
	\end{tikzcd} \]
	commutes, the conclusion follows.
\end{proof}

\begin{prop} \label{thm:refined_tannakian_image_canal}
	Let $Y \in \dSt_{\mathbb C}^{\mathrm{afp}}$ be a derived stack locally almost of finite presentation satisfying the same assumptions as in \cref{thm:refined_tannakian}.
	Let $X \in \dAfd_{\mathbb C}$ be a derived Stein space and let $V \Subset U \Subset X$ be a nested sequence of relatively compact Stein subspaces.
	Let $F \in \Fun^\otimes_L(\QCoh(Y), \cO_X\Mod)$ be a functor preserving perfect complexes, flat objects and connective objects.
	Then there exists a map $f \colon V \to Y\an$ so that the diagram
	\[ \begin{tikzcd}
		\QCoh(Y) \arrow{d} \arrow{r}{F} & \cO_X \Mod \arrow{d} \\
		\cO_{Y\an} \Mod \arrow{r}{f^*} & \cO_V \Mod
	\end{tikzcd} \]
	commutes.
\end{prop}

\begin{proof}
	Since $F \colon \QCoh(Y) \to \cO_X \Mod$ preserves perfect complexes, it restricts to a functor $\Perf(Y) \to \Perf(X)$.
	Furthermore, since $\QCoh(Y) \simeq \Ind(\Perf(Y))$ by assumption, we see that the extension of
	\[ \Perf(Y) \longrightarrow \Perf(X) \longrightarrow \cO_X \Mod \]
	along $\Perf(Y) \hookrightarrow \QCoh(Y)$ coincides with $F$.
	Let $A_X \coloneqq \Gamma(X; \cO_X\alg)$, $A_U \coloneqq \Gamma(U;\cO_U\alg)$ and $A_V \coloneqq \Gamma(V; \cO_V\alg)$.
	Using \cref{lem:global_section_relatively_compact}, we obtain a well defined functor
	\[ \overline{F} \colon \Perf(Y) \longrightarrow \Perf(X) \longrightarrow \Perf(A_U) . \]
	Consider the functor
	\[ \widetilde{F} \colon \QCoh(Y) \longrightarrow A_U \Mod \]
	obtained by extending by filtered colimits the composition
	\[ \Perf(Y) \stackrel{\overline{F}}{\longrightarrow} \Perf(A_U) \longhookrightarrow A_U \Mod \]
	along the inclusion $\Perf(Y) \hookrightarrow \QCoh(Y)$.
	\Cref{lem:mock_counit} implies that the composition
	\[ \begin{tikzcd}
		\Perf(X) \arrow{r}{\Gamma(U;-)}  &\Perf(A_U) \arrow{r} & \Perf(A_V) \arrow{r}{\varepsilon_V^*} & \Perf(V)
	\end{tikzcd} \]
	coincides with the restriction functor $\Perf(X) \to \Perf(V)$.
	It follows that the outer diagram in
	\[ \begin{tikzcd}
		\Perf(Y) \arrow{r} \arrow{d} & \Perf(X) \arrow{d} \\
		\Perf(A_U) \arrow{d} \arrow{r}{\varepsilon_U^*} & \Perf(U) \arrow{d} \\
		\Perf(A_V) \arrow{r}{\varepsilon_V^*} & \Perf(V)
	\end{tikzcd} \]
	commutes.
	In turn, this implies that the diagram
	\begin{equation} \label{eq:refined_tannakian_I}
		\begin{tikzcd}
			 \QCoh(Y) \arrow{r}{F} \arrow{d}{\widetilde{F}} & \cO_X \Mod \arrow{d} \\
			 A_V \Mod \arrow{r}{\varepsilon^*_V} & \cO_V \Mod
		\end{tikzcd}
	\end{equation}
	commutes.
	Recall now from \cref{prop:basic_properties_global_sections}  that the functor $\varepsilon_V^*$ commutes with filtered colimits, and it is conservative, $t$-exact and preserves flat, coherent and connective objects. 
	Therefore, we deduce that the symmetric monoidal functor $\widetilde{F} \colon \QCoh(Y) \to A_V \Mod$ commutes with colimits and preserves flat objects and connective objects.
	In particular, \cite[Theorem 9.3.0.3]{Lurie_SAG} implies the existence of a map $g \colon \Spec(A_V) \to Y$ so that $\widetilde{F} \simeq g^*$.
	The map $g$ defines an element in $G^p_X(Y)(V)$.
	Let $f$ be the image of $g$ via the canonical map
	\[ G^p_X(Y)(V) \longrightarrow G^s_X(Y)(V) \simeq F^s_X(Y\an)(V) \simeq \Map_{\dAnSt_k}(V, Y\an) . \]
	Then unravelling the definitions, we see that the diagram
	\[ \begin{tikzcd}
		\QCoh(Y) \arrow{r}{\widetilde{F}} \arrow{d} & A_V \Mod \arrow{d}{\widetilde{\varepsilon_V^*}} \\
		\cO_{Y\an} \Mod \arrow{r}{f^*} & \cO_V \Mod
	\end{tikzcd} \]
	commutes.
	Combining this with the commutativity of the diagram \eqref{eq:refined_tannakian_I}, we see that the conclusion follows.
\end{proof}

\subsection{Mapping stacks with Tannakian target}

We now turn to the main result of this paper:

\begin{thm} \label{thm:tannakian_target}
	Let $X, Y \in \mathrm{dSt}^{\mathrm{afp}}_k$ be derived stacks locally almost of finite presentation.
	Assume that:
	\begin{enumerate}
		\item the stack $Y$ is a geometric, tannakian stack such that $\QCoh(Y) \simeq \mathrm{Ind}(\Perf(Y))$;
		\item the mapping stack $\bfMap(X,Y)$ is geometric.
	\end{enumerate}
	If $X$ satisfies the GAGA property, then the canonical map
	\[ \Map(X,Y) \longrightarrow \Map(X\an, Y\an) \]
	is an equivalence.
	Furthermore, if $X$ satisfies the universal GAGA property, then the canonical map
	\[ \bfMap(X,Y)\an \longrightarrow \bfAnMap(X\an, Y\an) \]
	is an equivalence.
\end{thm}

\begin{proof}
	We first observe that when $U = \Sp(k)$ we have
	\[ \bfMap(X,Y)\an(U) \simeq \Map(X,Y) , \]
	so the first statement follows from the proof of the second.\footnote{Note that since $\Sp(\C)$ is a compact Stein space (and may be considered as a constant pro-object) this is indeed a special case of our result. Of course a direct proof in this case would be significantly easier than in the relative case that is our main interest.}
	
	Let now $U \in \dAfd_k$ be a derived $k$-affinoid (resp.\ Stein) space.
	Assume that $X$ satisfies the GAGA property relative to $U$.
	We will prove that the natural morphism
	\[ \bfMap(X,Y)\an(U) \longrightarrow \bfAnMap(X\an, Y\an)(U) \]
	is an equivalence.
	We work on $\cX_U$.
	Using \cref{lem:F_continuous_and_cocontinuous} we are reduced to checking that the natural morphism
	\[ F^s_U( \bfMap(X,Y)\an ) \longrightarrow F^s_U( \bfAnMap(X\an, Y\an) ) \]
	is an equivalence.
	Using \cref{thm:generalized_adjunction} and the fact that $\bfMap(X,Y)$ is geometric, we obtain a natural equivalence
	\[ G^s_U( \bfMap(X,Y) ) \longrightarrow F^s_U(\bfMap(X,Y)\an) . \]
	We consider as in \cref{prop:general_strategy_analytification} the induced map
	\begin{equation} \label{eq:tannakian_target_generalized_adjunction}
		G^p_U( \bfMap(X,Y) ) \longrightarrow F^s_U( \bfAnMap(X\an, Y\an) ) .
	\end{equation}
	
	We first deal with the non-archimedean case.
	In this case, we claim that the map \eqref{eq:tannakian_target_generalized_adjunction} is an equivalence.
	Fix an \'etale map $V \to U$ from a derived $k$-affinoid space $V$.
	Let
	\[ A_V \coloneqq \Gamma(V; \cO_V\alg) . \]
	Then
	\[ G^p_U( \bfMap(X,Y) )(V) \simeq \Map_{\dSt_k}( \Spec(A_V), \bfMap(X,Y) ) \simeq \Map_{\dSt_k}( \Spec(A_V) \times X, Y ) . \]
	Since $Y$ is tannakian, we have a fully faithful embedding
	\[ Q \colon \Map_{\dSt_k}( \Spec(A_V) \times X, Y ) \longhookrightarrow \Fun^\otimes( \Perf(Y), \Perf(\Spec(A_V) \times  X) . \]
	On the other hand, \cref{thm:refined_tannakian} provides us with a fully faithful embedding
	\[ Q\an \colon F^s_U(\bfAnMap(X\an,Y\an))(V) \simeq \Map_{\dAnSt_k}(V \times X\an, Y\an) \longhookrightarrow \Fun^\otimes( \Perf(Y), \Perf(V \times X\an) ) . \]
	Now, since $X$ satisfies the universal GAGA property, the functor
	\[ \varepsilon_{X,V}^* \colon \Perf(\Spec(A_V) \times  X) \simeq \Perf(V \times X\an) \]
	is an equivalence.
	We therefore obtain the following diagram:
	\[ \begin{tikzcd}
		\Map_{\dSt_k}( \Spec(A_V) \times X, Y ) \arrow{r} \arrow[hook]{d}{Q} & \Map_{\dAnSt_k}( V \times X\an, Y\an ) \arrow[hook]{d}{Q\an} \\
		\Fun^\otimes( \Perf(Y), \Perf(\Spec(A_V) \times X) ) \arrow{r}{\sim}[swap]{\varepsilon_{X,V}^*} & \Fun^\otimes(\Perf(Y), \Perf(V \times X\an)) .
	\end{tikzcd} \]
	This immediately implies that the top horizontal map is fully faithful.
	We are therefore left to check that it is essentially surjective.
	To do so, we fix a morphism
	\[ f \colon V \times X\an \longrightarrow Y\an . \]
	Consider the extended diagram
	\[ \begin{tikzcd}
		\Map_{\dSt_k}( \Spec(A_V) \times X, Y ) \arrow{r} \arrow[hook]{d}{Q} & \Map_{\dAnSt_k}( V \times X\an, Y\an ) \arrow[hook]{d}{Q\an} \\
		\Fun^\otimes( \Perf(Y), \Perf(\Spec(A_V) \times X) ) \arrow{d} \arrow{r}{\sim}[swap]{\varepsilon_{X,V}^*} & \Fun^\otimes(\Perf(Y), \Perf(V \times X\an)) \arrow{d} \\
		\Fun^\otimes_L( \QCoh(Y), \QCoh( \Spec(A_V) \times X ) ) \arrow{r}[swap]{\varepsilon_{X,V}^*} & \Fun^\otimes_L( \QCoh(Y), \cO_{V \times X\an} \Mod ) .
	\end{tikzcd} \]
	Let
	\[ F \colon \Perf(Y) \longrightarrow \Perf(\Spec(A_V) \times X) \]
	be the functor corresponding to $f^* \colon \Perf(Y) \to \Perf(V \times X\an)$ via the equivalence $\varepsilon_{X,V}^*$, and let
	\[ \widetilde{F} \colon \QCoh(Y) \longrightarrow \QCoh(\Spec(A_V) \times X) \]
	be the functor obtained by extension along $\Perf(Y) \hookrightarrow \QCoh(Y) \simeq \mathrm{Ind}(\Perf(Y))$.
	Since $Y$ is tannakian, it is enough to check that $\widetilde{F}$ commutes with flat objects and connective objects.
	We observe that
	\[ \varepsilon_{X,V}^* \circ \widetilde{F} \simeq \widetilde{f}^* , \]
	where
	\[ \widetilde{f}^* \colon \QCoh(Y) \longrightarrow \cO_{V \times X\an} \Mod \]
	is the functor obtained by left Kan extension of $f^*$ along $\Perf(Y) \hookrightarrow \mathrm{Ind}(\Perf(Y)) \simeq \QCoh(Y)$.
	As such, \cref{lem:refined_tannakian_image_necessary_conditions} implies that it preserves flat objects and connective objects.
	We now observe that since $X$ satisfies GAGA relative to $V$ the functor
	\[ \varepsilon_{X,V}^* \colon \QCoh(\Spec(A_V) \times X) \longrightarrow \cO_{V \times X\an} \Mod \]
	is conservative and $t$-exact.
	As flatness and connectivity of an object $\cF \in \QCoh(\Spec(A_V) \times X)$ can be checked locally with respect to $X$, we conclude that $\cF$ is flat (resp.\ connective) if and only if $\varepsilon_{X,V}^*(\cF)$ is flat (resp.\ connective).
	Since we know that $\widetilde{f}^*$ preserves flat objects and connective objects, it follows that the same is true for $\widetilde{F}$.
	Since $Y$ is tannakian, we see that $F$ comes from a morphism $\overline{f} \colon \Spec(A_V) \times X \to Y$.
	Moreover the commutativity of the diagram implies that
	\[ \overline{f}\an \simeq f . \]
	This completes the proof in the non-archimedean case.\\
	
	We now deal with the \canal case.
	We adapt the idea of \cref{prop:general_strategy_analytification} to this context.
	As usual, the overall strategy stays the same, but instead of proving that the map \eqref{eq:tannakian_target_generalized_adjunction} is an equivalence on the spot, we prove that it becomes an equivalence after passing (via \cref{thm:sheaves_compact_subsets}) to compact Stein subsets of $U$.
	We will conclude by \cref{lem:sheaves_open_sheaves_compact}.
	Let therefore $K \subset U$ be a compact Stein in $U$ and let $(K)_U$ be the associated pro-object in $\dAnSt_\C$.
	If $V \subset U$ is an open Stein subspace, we let
	\[ A_V \coloneqq \Gamma(V; \cO_V\alg) . \]
	Then, unravelling the definitions, we have
	\[ G^p_U( \bfMap(X,Y) )((K)_U) \simeq \colim_{K \subset V \subset U} \Map_{\dSt_\C}( \Spec(A_V) \times X, Y ) , \]
	where the colimit ranges over all the open Stein neighbourhood of $K$ inside $U$.
	Since $Y$ is tannakian, we have a fully faithful embedding
	\[ Q_V \colon \Map_{\dSt_\C}( \Spec(A_V) \times X,  Y) \longhookrightarrow \Fun^\otimes(\Perf(Y), \Perf(\Spec(A_V) \times X)) . \]
	Since fully faithful functors are stable under filtered colimits, we obtain a fully faithful inclusion
	\[ Q_{(K)} \colon \colim_{K \subset V \subset U} \Map_{\dSt_\C}( \Spec(A_V) \times X, Y ) \longhookrightarrow \colim_{K \subset V \subset U} \Fun^\otimes(\Perf(Y), \Perf(\Spec(A_V) \times X) . \]
	We work as in \cref{thm:refined_tannakian} and introduce $\Fun^\otimes_{\mathrm{ind}}$, the mapping space in $\Ind(\Cat_\infty^{\mathrm{st}, \otimes})$.
	Then we have a tautological equivalence
	\[ \colim_{K \subset V \subset U} \Fun^\otimes(\Perf(Y), \Perf(\Spec(A_V) \times X)) \simeq \Fun^\otimes_{\mathrm{ind}} \Big( \Perf(Y), \fcolim_{K \subset V \subset U} \Perf( \Spec(A_V) \times X) \Big) . \]
	On the other hand, \cref{thm:refined_tannakian} provides us with fully faithful embeddings
	\[ Q\an_V \colon \Map_{\dAnSt_\C}(V \times X\an, Y\an) \longhookrightarrow \Fun^\otimes(\Perf(Y), \Perf(V \times X\an)) . \]
	These embeddings assemble into a fully faithful functor
	\[ Q\an_{(K)} \colon \colim_{K \subset V \subset U} \Map_{\dAnSt_\C}(V \times X\an, Y\an) \longhookrightarrow \colim_{K \subset V \subset U} \Fun^\otimes( \Perf(Y), \Perf(V \times X\an) ) . \]
	Once again, we can formally rewrite
	\[ \colim_{K \subset V \subset U} \Fun^\otimes( \Perf(Y), \Perf( V \times X\an) ) \simeq \Fun^\otimes_{\mathrm{ind}} \Big( \Perf(Y), \fcolim_{K \subset V \subset U} \Perf(V \times X\an) \Big) . \]
	Consider now the following commutative diagram:
	\[ \begin{tikzcd}
		\colim_{K \subset V \subset U} \Map_{\dSt_\C}( \Spec(A_V) \times X, Y) \arrow{r} \arrow[hook]{d}{Q_{(K)}} & \colim_{K \subset V \subset U} \Map_{\dAnSt_\C}( V \times X\an, Y\an ) \arrow[hook]{d}{Q\an_{(K)}} \\
		\Fun^\otimes_{\mathrm{ind}}\Big( \Perf(Y) , \fcolim_{K \subset V \subset U} \Perf(\Spec(A_V) \times X) \Big) \arrow{r} & \Fun^\otimes_{\mathrm{ind}}\Big( \Perf(Y), \fcolim_{K \subset V \subset U} \Perf(V \times X\an) \Big) .
	\end{tikzcd} \]
	Recall now that $X$ satisfies the GAGA property relative to $U$ (see \cref{defin:strong_GAGA}(3)).
	In other words, the canonical morphism
	\[ \fcolim_{K \subset V \subset U} \Perf(\Spec(A_V) \times X) \longrightarrow \fcolim_{K \subset V \subset U} \Perf(V \times X\an) \]
	is an equivalence in $\Ind(\Cat_\infty^{\mathrm{st}, \otimes})$.
	This implies that the bottom horizontal arrow in the above diagram is an equivalence.
	As $Q_{(K)}$ and $Q_{(K)}\an$ are fully faithful embeddings, we deduce that the top horizontal morphism is fully faithful as well.
	We are therefore left to check that it is essentially surjective as well.
	Let
	\[ [f] \in \colim_{K \subset V \subset U} \Map_{\dAnSt_\C}(V \times X\an, Y\an) \]
	be any element and let
	\[ f \colon V \times X\an \longrightarrow Y\an \]
	be a representative for $[f]$.
	By construction, we have
	\[ Q\an_{(K)}([f]) \simeq [f^*] \in \Fun^\otimes_{\mathrm{ind}} \Big(\Perf(Y), \fcolim_{K \subset V \subset U} \Perf(V \times X\an) \Big) \simeq \colim_{K \subset V \subset U} \Fun^\otimes(\Perf(Y), \Perf(V \times X\an))  . \]
	Via the equivalence
	\[ \varepsilon^*_{X,(K)} \colon \Fun^\otimes_{\mathrm{ind}} \Big( \Perf(Y), \fcolim_{K \subset V \subset U} \Perf(\Spec(A_V) \times X) \Big) \longrightarrow \Fun^\otimes_{\mathrm{ind}}\Big( \Perf(Y) , \fcolim_{K\subset V \subset U} \Perf(V \times X\an) \Big) \]
	we can select an open Stein neighbourhood $W$ of $K$ and a symmetric monoidal functor
	\[ F_W \colon \Perf(Y) \longrightarrow \Perf(\Spec(A_W) \times X) \]
	such that
	\[ [\varepsilon^*_{X,W} \circ F_W] \simeq [f^*] . \]
	Without loss of generality, we can suppose that $W \subset V$.
	Let $f_W \colon W \times X\an \longrightarrow Y\an$ be the restriction of $f$ along $W \times X\an \to V \times X\an$.
	Then since $[\varepsilon_{X,W}^* \circ F_W] \simeq [f_W^*]$, we see that up to shrinking $W$ again, we can suppose that there is a natural equivalence
	\[ \varepsilon_{X,W}^* \circ F_W \simeq f_W^* \]
	of functors $\Perf(Y) \to \Perf(W \times X\an)$.
	Let
	\[ \widetilde{F}_W \colon \QCoh(Y) \longrightarrow \QCoh(\Spec(A_V) \times X) \]
	be the extension of $F_W$ along $\Perf(Y) \hookrightarrow \QCoh(Y) \simeq \mathrm{Ind}(\Perf(Y))$.
	We claim that $\widetilde{F}_W$ commutes with flat objects and connective objects.
	As $Y$ is tannakian, this will suffice to complete the proof.
	To prove the claim, consider the commutative diagram
	\[ \begin{tikzcd}
		{} & \QCoh(Y) \arrow{dr}{f_W^*} \arrow{dl}[swap]{\widetilde{F}_W} \\
		\QCoh(\Spec(A_W) \times X) \arrow{rr}{\varepsilon_{X,W}^*} & & \cO_{V \times X\an} \Mod .
	\end{tikzcd} \]
	Since flatness and connectivity of an object in $\QCoh(\Spec(A_W) \times X)$ can be tested locally on $X$, we can reduce ourselves to the case where $X$ is affine.
	In this case, the conclusion follows from the fact that $\varepsilon_{X,W}^*$ is $t$-exact and conservative and from \cref{lem:refined_tannakian_image_necessary_conditions}.
\end{proof}

\begin{cor}
	Let $X, Y \in \dSt_k^{\mathrm{afp}}$ be derived stacks locally almost of finite presentation over $k$.
	Then the natural map
	\[ \bfMap(X,Y)\an \longrightarrow \bfAnMap(X\an, Y\an) \]
	is an equivalence whenever $Y$ is geometric, Tannakian and $\QCoh(Y) \simeq \Ind(\Perf(Y))$ and whenever $X$ belongs to one of the following cases:
	\begin{enumerate}
		\item $X$ is a derived proper geometric stack locally almost of finite presentation over $k$ such that $\bfMap(X,Y)$ is a geometric stack;
		\item $X$ is of the form $Z_{\mathrm{dR}}$ or $Z_{\mathrm{Dol}}$ for some smooth and proper scheme $Z$;
		\item $X$ is of the form $K_\rB$ for some finite homotopy type $K$.
	\end{enumerate}
\end{cor}

\begin{proof}
	We have to check that $X$ satisfies the universal GAGA property and that $\bfMap(X,Y)$ is geometric.
	For point (1), the geometricity is part of the assumption, and the universal GAGA property is exactly the content of \cref{thm:relative_GAGA}.
	In point (2), the geometricity can be deduced from Lurie's representability theorem, but it has also been verified directly in \cite{Simpson_Geometricity}.
	The universal GAGA property for these cases has been verified in Propositions \ref{prop:de_Rham_universal_GAGA} and \ref{prop:Dolbeault_universal_GAGA}.
	Finally, in point (3) the geometricity just follows from the fact that geometric stacks are closed under finite limits, while the universal GAGA property has been proven in \cref{prop:Betti_universal_GAGA}.
\end{proof}

\begin{rem}
	We discussed the question of the geometricity of $\bfMap(X,Y)$ in \cref{rem:geometricity_mapping_stacks}.
	As for the assumptions on $Y$, let us remark that they are satisfied in the following two important cases:
	\begin{enumerate}
		\item $Y$ is a quasi-compact quasi-separated Deligne-Mumford stack
		\item $Y$ is the classifying stack of an affine group scheme of finite type in characteristic 0. 
	\end{enumerate} 
	Both examples are tannakian by \cite[Theorem 3.4.2]{DAG-VIII}.
	Compact generation is proved in \cite[Theorem A]{Hall_Rydh_2017} (see also Example 9.6) for the first case and \cite[Theorem A]{Hall_Rydh_2016} for the second.
\end{rem}

\section{Applications} \label{sec:applications}

In this section we develop some applications of the main results of this paper.

\subsection{Further consequences of \cref{thm:generalized_adjunction}}

\Cref{thm:generalized_adjunction} is one of the key technical results of this paper.
It is certainly the main tool we have to deal with analytification of geometric stacks that are not Deligne-Mumford.
We use it here to deduce some other general properties of the analytification functor.
The following is a generalization of \cref{lem:perfect_complexes_affinoid} and \cref{thm:perfect_complexes_compact_Stein}:

\begin{prop}\label{prop:analyticfunctorcat}
	Let $\cC \in \Cat_\infty$ and let $\Perf_k^\cC$ be the derived stack sending $A$ to $\Fun(\cC, \Perf_k(A))$.
	Similarly, let $\AnPerf_k^\cC$ be the derived analytic $\Cat_\infty$-valued stack sending $U$ to $\Fun(\cC, \AnPerf_k(U))$.
	Then:
	\begin{enumerate}
		\item If $k$ is non-archimedean, for every derived $k$-affinoid space $U \in \dAfd_k$, there is a canonical equivalence
		\[ \Perf_k^\cC(A_U) \simeq \AnPerf_k^\cC(U) .  \]
		\item If $k = \C$, then for every derived Stein space $U \in \dStn_\C$ and every compact Stein subset $K$ in $U$, there is a canonical equivalence
		\[ \fcolim_{K \subset U \subset X} \Perf_\C^\cC(A_U) \simeq \fcolim_{K \subset U \subset X} \AnPerf_\C^\cC(U) \]
		in $\Ind(\Cat_\infty^{\mathrm{st}, \otimes})$.
	\end{enumerate}
\end{prop}

\begin{proof}
	Constructing a natural transformation
	\[ \eta^\cC \colon \big(\Perf_k^\cC \big)\an \longrightarrow \AnPerf_k^\cC \]
	is equivalent to constructing a natural transformation
	\[ \Perf_k^\cC \longrightarrow \AnPerf_k^\cC \circ (-)\an . \]
	The latter is simply induced by composition with the analytification functor
	\[ \varepsilon_X^* \colon \Fun(\cC, \Perf(X)) \longrightarrow \Fun(\cC, \Perf(X\an)) . \]
	At this point, in the non-archimedean setting the conclusion follows immediately from \cref{lem:perfect_complexes_affinoid}.
	In the \canal case, consider the functor
	\[ \Fun(\cC, -) \colon \Cat_\infty^{\mathrm{st}, \otimes} \longrightarrow \Cat_\infty^{\mathrm{st}, \otimes} . \]
	Applying the ind-construction we obtain a functor
	\[ \Ind(\Fun(\cC, -)) \colon \Ind(\Cat_\infty^{\mathrm{st}, \otimes}) \longrightarrow \Ind(\Cat_\infty^{\mathrm{st}, \otimes}) , \]
	which takes an ind-object $\fcolim_{i \in I} \cD_i$ to
	\[ \fcolim_{i \in I} \Fun(\cC, \cD_i) . \]
	Evaluating this functor on the equivalence obtained in \cref{thm:perfect_complexes_compact_Stein} we therefore get the equivalence
	\[ \fcolim_{K \subset U \subset X} \Fun(\cC, \Perf(A_U)) \simeq \fcolim_{K \subset U \subset X} \Fun(\cC, \Perf(U)) \]
	we were looking for.
\end{proof}

Similarly to what we did in \cref{sec:analytic_perfect_complexes}, we can now obtain the following analogue of \cref{prop:analytification_perfect_complexes_absolute}:

\begin{cor} \label{cor:analyticfunctorcat}
	Let $\cC \in \Cat_\infty$ be a compact object and let $\bfPerf_k^\cC$ (resp.\ $\bfAnPerf_k^\cC$) be the derived stack (resp.\ derived analytic stack) associated to $\Perf_k^\cC$ and $\AnPerf_k^\cC$.\footnote{Notice that $\bfPerf_k^\cC(X) \simeq \Fun(\cC, \Perf(X))^\simeq$, which is different from $\Fun(\cC, \bfPerf(X))$.}
	Then the canonical morphism
	\[ (\bfPerf_k^\cC)\an \longrightarrow \bfAnPerf_k^\cC \]
	is an equivalence.
\end{cor}

\begin{proof}
	In the previous proposition we constructed a morphism $\eta^\cC \colon (\Perf_k^\cC)\an \to \AnPerf_k^\cC$ which induces a canonical morphism $(\bfPerf_k^\cC)\an \to \bfAnPerf_k^\cC$.
	In order to check that the latter is an equivalence, we verify that the hypotheses of \cref{prop:general_strategy_analytification} are satisfied.
	First of all, since $\cC$ is compact, $\bfPerf_k^\cC$ is still a locally geometric stack.
	Unwinding the definitions, we see that have to check that the map
	\[ \Fun(\cC, \Perf(A_U))^\simeq \longrightarrow \Fun(\cC, \Perf(U))^\simeq \]
	is an equivalence for every $U \in \dAfd_k$ when $k$ is non-archimedean and that
	\[ \colim_{K \subset V \subset U} \Fun(\cC, \Perf(A_V))^\simeq \longrightarrow \colim_{K \subset V \subset U} \Fun(\cC, \Perf(V))^\simeq \]
	is an equivalence for every $U \in \dStn_\C$ and every compact Stein subset $K$ of $U$.
	Both statements follow at once from \cref{prop:analyticfunctorcat}.
\end{proof}

The following result answers a question raised by G.\ Ginot:

\begin{prop} \label{prop:analytification_on_limits}
	Let $I$ be a finite $\infty$-category and let $F \colon I \to \dSt_k^{\mathrm{afp}}$ be a diagram.
	Suppose that for every $i \in I$, $X_i \coloneqq F(i)$ is a geometric stack. Then
	\[ \Big( \lim_{i \in I} F(i) \Big)\an \longrightarrow \lim_{i \in I} F(i)\an \]
	is an equivalence.
\end{prop}

\begin{proof}
	Let
	\[ X \coloneqq \lim_{i \in I} X_i \in \dSt_k \]
	be the limit.
	Since $I$ is finite, we see that $X$ is a geometric stack.
	
	It is now enough to check that for every derived $k$-affinoid (resp.\ Stein) space $U \in \dAfd_k$ the canonical map
	\[ \Map_{\dAnSt_k}(U, X\an) \longrightarrow \lim_{i \in I} \Map_{\dAnSt_k}(U, X_i\an)  \]
	is an equivalence.
	We work in the $\infty$-topos $\cX_U$.
	Since $\Map_{\cX_U}(\mathbf 1_U, -)$ commutes with limits, \cref{lem:F_continuous_and_cocontinuous} implies that it is enough to prove that the canonical map
	\[ F^s_U(X\an) \longrightarrow \lim_{i \in I} F^s_U(X_i\an) \]
	is an equivalence.
	Since $X$ and all the $X_i$ are geometric stacks, we can use \cref{thm:generalized_adjunction} to obtain (functorial) equivalences
	\[ G^s_U(X) \simeq F^s_U(X\an) \quad , \quad G^s_U(X_i) \simeq F^s_U(X\an) . \]
	Since the sheafification commutes with finite colimits, we see that the canonical map
	\[ G^s_U(X) \longrightarrow \lim_{i \in I} G^s_U(X_i) \]
	is induced by sheafification from the map
	\[ G^p_U(X) \longrightarrow \lim_{i \in I} G^p_U(X_i) . \]
	It is then sufficient to prove that this second map is an equivalence.
	For every \'etale map from a derived $k$-affinoid (resp.\ Stein) space $V \to U$, we can rewrite
	\[ G^p_U(X)(V) \simeq \Map_{\dSt_k}(\Spec(A_V), X) \quad , \quad G^p_U(X_i)(V) \simeq \Map_{\dSt_k}(\Spec(A_V), X_i) , \]
	where we set as usual $A_V \coloneqq \Gamma(V; \cO_V\alg)$.
	The conclusion now follows from the fact that $X \simeq \lim_{i \in I} X_i$ in $\dSt_k$.
\end{proof}

\subsection{The derived period domain}

The following corollary clarifies a construction that arises when considering the derived period map that is constructed in \cite{DiNatale_Global_Period_2016}. 
This is a derived enhancement of Griffith's classical period map, that associates to smooth projective family of derived stacks a map from the base to a derived period domain.
\begin{cor}\label{cor:period_domain}
	The derived period domain constructed in \cite{DiNatale_Global_Period_2016} is an analytic moduli stack.
\end{cor}
\begin{proof}
We recall the definition of the derived period domain. 
Assume we are given a perfect complex $V$, concentrated in degrees $0$ to $2n$ and equipped with a $2n$-shifted bilinear form $Q$ that is non-degenerate on cohomology. 
We think of $V$ as the cohomology complex of a smooth projective variety.
Then there is a derived geometric stack $\mathbf D_{n}(V,Q)$ which classifies decreasing filtrations $F^{*}$ of $V$ of length $n+1$ which induce filtrations on cohomology groups and which satisfy the Hodge-Riemann orthogonality relation with respect to $Q$, i.e.\ $Q$ vanishes on $F^{i}\otimes F^{n+1-i}$.
For precise definitions see Theorem 3.4 of \cite{DiNatale_Global_Period_2016}.
The coarse moduli space of the underived truncation of $\mathbf D_{n}(V,Q)$ recovers the closure of the classical period domain.

The derived period domain is the open analytic substack $U$ of $\mathbf D_{n}(V,Q)\an$ given by the Hodge-Riemann degeneracy conditions on cohomology groups of $F^{*}$. $U$ is a geometric stack \cite{DiNatale_Global_Period_2016}.

The claim of this corollary is that $U$ is in fact the derived analytic moduli stack classifying filtrations of $V$ of length $n+1$ that satisfy the Hodge-Riemann conditions. 
It suffices to show that $\mathbf D_{n}(V,Q)\an$ is equivalent to the derived analytic moduli stack classifying filtrations satisfying Hodge-Riemann orthogonality, which we will denote by $\mathbf{AnD}_{n}(V,Q)$.

To prove this we observe that following the proof of Theorem 3.4 in \cite{DiNatale_Global_Period_2016} $\mathbf D_{n}(V,Q)$ is constructed in a categorical way from the stack of perfect complexes $\mathbf{Perf}_{\C}$.
The stack of $(n+1)$-term filtrations $\mathbf{Filt}_{n}$ is defined as the stack of $n$ composable morphisms in $\mathbf{Perf}_{\C}$. If $\cC_{n}$ denotes the the category of $n$ composable arrows then we have $\mathbf{Filt}_{n} = \mathbf{Perf}_{\C}^{\cC_{n}}$.
The derived flag variety $\mathbf{Flag}_{n}(V)$ is defined as the homotopy fibre of the forgetful map $\mathbf{Filt}_{n} \to \mathbf{Perf}_{\C}$ over $V \in \mathbf{Perf}_{\C}$. 
We can also define the stack $\mathbf{Flag}_{1}(Q)$ as the homotopy fibre of $\mathbf{Filt}_{1}^{ \cC^{1}} \to \mathbf{Perf}_{\C}^{\cC_{1}}$ over $Q: \Sym^{2}V \to \C[2n]$. 
Then $\mathbf D_{n}(V,Q)$ is the limit of a diagram 
\[
\mathbf{Flag}_{n}(V) \stackrel \sigma \longrightarrow \mathbf{Flag}_{1}(\Sym^{2}V) \times \mathbf{Flag}_{1}(\C[2n]) \stackrel e \longleftarrow \mathbf{Flag}_{1}(Q).
\]
Here $e$ is just given by the evaluation maps at source and target. The map $\sigma$ sends a filtration $F^n \to \dots \to F^{0}$ to $S \times (0 \to \cO)$ where $S$ is the 2-term filtration on $\Sym^{2} F^{0}$ given by the image of  $\oplus_{i} (F^{i}\otimes F^{n+1-i})$.

We can now perform the exact same constructions starting from $\bfAnPerf_{\C}$ instead of $\mathbf{Perf}_{\C}$ to obtain $\mathbf{AnD}_{n}(V,Q)$. 

We first construct $\mathbf{AnFilt}_{n} = \mathbf{AnPerf}_{\C}^{\cC_{n}}$ which can be expressed as $\mathbf{Filt}_{n}\an$ by \cref{cor:analyticfunctorcat}. Similarly $\mathbf{AnFlag}_{n}(V)$, defined as the homotopy fibre of $\mathbf{AnPerf}_{\C}^{\cC_{n}} \to \mathbf{AnPerf}_{\C}$ over $V$, is the analytification of $\mathbf{Flag}(V)$
since homotopy fibres commute with analytification. The same holds for $\mathbf{AnFlag}_{1}(Q)$
and putting all of this together $\mathbf{AnD}_{n}(V,Q)$ is a derived analytic moduli stack which is equivalent to the analytification of $\mathbf D_{n}(V,Q)$.
\end{proof}

\subsection{The derived Riemann-Hilbert correspondence}

One of the main applications of the techniques of this paper is to obtain an extended version of the derived Riemann-Hilbert correspondence first proven in \cite{Porta_Derived_Riemann_Hilbert}.
In loc.\ cit.\ the second author following a suggestion of C.\ Simpson introduced for every \canal space $X$ a morphism
\[ \eta_{\mathrm{RH}} \colon X_{\mathrm{dR}} \longrightarrow X_{\mathrm{B}} \]
called the Riemann-Hilbert transformation.
He then showed that if $X$ is smooth the canonical morphism
\[ \eta_{\mathrm{RH}}^* \colon \bfAnMap( X_{\mathrm{B}}, \bfAnPerf_\C ) \longrightarrow \bfAnMap( X_{\mathrm{dR}}, \bfAnPerf_\C ) \]
is an equivalence (see Theorem 6.11 in loc.\ cit.).

\begin{cor}
	Let $X$ be a smooth proper scheme over $\mathbb C$.
	Then $\eta_{\mathrm{RH}}$ induces an equivalence
	\[ \bfMap(X_{\mathrm{B}}, \bfPerf_\C)\an \simeq \bfMap(X_{\mathrm{dR}}, \bfPerf_\C)\an . \]
\end{cor}

\begin{proof}
	It follows from \cref{eg:derhamstack} and \cref{eg:constantstack} that $X_{\mathrm{dR}}$ and $X_{\mathrm{B}}$ satisfy the universal GAGA property.
	In other words, we proved that the canonical maps
	\[ \bfMap(X_{\mathrm B}, \bfPerf_\C)\an \longrightarrow \bfAnMap((X_{\mathrm B})\an, \bfAnPerf_\C) \]
	and
	\[ \bfMap(X_{\mathrm{dR}}, \bfPerf_\C)\an \longrightarrow \bfAnMap((X_{\mathrm{dR}})\an, \bfAnPerf_\C) \]
	are equivalences.
	Furthermore, we saw in \cref{eg:derhamstack} that there is a canonical morphism
	\[ (X_{\mathrm{dR}})\an \longrightarrow (X\an)_{\mathrm{dR}} \]
	which is furthermore an equivalence because $X$ is smooth.
	Since there is an obvious equivalence $(X_{\mathrm B})\an \simeq (X\an)_{\mathrm B}$, we can use the Riemann-Hilbert transformation for $X\an$ to obtain the equivalence we are looking for.
\end{proof}

As our last application we notice that \cref{thm:refined_tannakian_image_canal} and \cref{thm:tannakian_target} imply together that the Riemann-Hilbert correspondence with coefficients in an algebraic stack satisfying the Tannakian property is still an equivalence:

\begin{cor} \label{cor:generalized_RH_correspondence}
	Let $Y \in \dSt_\C^{\mathrm{afp}}$ be derived stack locally almost of finite presentations and satisfying the assumptions of \cref{thm:tannakian_target}.
	Then for every smooth analytic space $X$, the Riemann-Hilbert transformation $\eta_{\mathrm{RH}} \colon X_{\mathrm{dR}} \to X_{\mathrm{B}}$ induces an equivalence
	\[ \eta_{\mathrm{RH}}^* \colon \bfAnMap(X_{\mathrm{B}}, Y\an) \longrightarrow \bfAnMap(X_{\mathrm{dR}}, Y\an) . \]
\end{cor}

\begin{proof}
	Fix a derived Stein space $S \in \dStn_\C$.
	We have to prove that $\eta_{\mathrm{RH}} \colon X_{\mathrm{dR}} \to X_{\mathrm B}$ induces an equivalence
	\[ \eta_{\mathrm{RH}}^* \colon \Map_{\dAnSt_\C}( S \times X_{\mathrm B} , Y\an ) \longrightarrow \Map_{\dAnSt_\C}( S \times X_{\mathrm{dR}} , Y\an ) . \]
	Consider the commutative diagram
	\[ \begin{tikzcd}
		\Map_{\dAnSt_\C}( S \times X_{\mathrm B}, Y\an ) \arrow{d} \arrow{r} & \Map_{\dAnSt_\C}( S \times X_{\mathrm{dR}}, Y\an ) \arrow{d} \\
		\Fun^\otimes( \Perf(Y), \Perf(S \times X_{\mathrm B}) ) \arrow{r} & \Fun^\otimes( \Perf(Y), \Perf(S \times X_{\mathrm{dR}}) ) ,
	\end{tikzcd} \]
	where the vertical morphisms are the ones induced by \cref{thm:refined_tannakian}.
	This proposition shows furthermore that they are fully faithful.
	The bottom horizontal morphism is an equivalence in virtue of \cite[Theorem 6.11]{Porta_Derived_Riemann_Hilbert}.
	It follows that the top horizontal functor is fully faithful, too.
	
	We are left to check that it is essentially surjective.
	Fix a morphism
	\[ f \colon S \times X_{\mathrm{dR}} \longrightarrow Y\an , \]
	and let
	\[ F \colon \Perf(Y) \longrightarrow \Perf(S \times X_{\mathrm{dR}}) \]
	be the induced symmetric monoidal functor.
	Let
	\[ G \colon \Perf(Y) \longrightarrow \Perf(S \times X_{\mathrm B}) \]
	be the symmetric monoidal functor induced by the equivalence $\eta_{\mathrm{RH}}^* \colon \Perf(S \times X_{\mathrm B}) \xrightarrow{\sim} \Perf(S \times X_{\mathrm{dR}})$.
	We would like to invoke \cref{thm:refined_tannakian_image_canal}, but for this we first have to replace $X_{\mathrm B}$ by a colimit of derived Stein spaces.
	
	Applying the argument of \cite[Lemma 5.14 and Remark 5.15]{Porta_Yu_Higher_analytic_stacks_2014}, we produce three open hypercovers $W_\bullet$, $V_\bullet$ and $U_\bullet$ of $X$ satisfying the following conditions:
	\begin{enumerate}
		\item for every integer $m$, $U_m$, $V_m$ and $W_m$ are disjoint unions of \emph{contractible} open Stein subspaces of $X$;
		\item for every integer $m$, we have $W_m \Subset V_m \Subset U_m$.
	\end{enumerate}
	Observe that for every integer $m$, we have canonical equivalences
	\[ (U_m)_{\mathrm B} \simeq \coprod_{I_U} \Sp(\C) \quad , \quad (V_m)_{\mathrm B} \simeq \coprod_{I_V} \Sp(\C) \quad , (W_m)_{\mathrm B} \simeq \coprod_{I_W} \Sp(\C) , \]
	and that
	\[ | (W_\bullet)_{\mathrm B} | \simeq | (V_{\bullet})_{\mathrm B} | \simeq | (U_\bullet)_{\mathrm B} | \simeq X_{\mathrm B} . \]
	Therefore, we can represent $G$ as an element in the limit
	\[ \lim_{m \in \mathbf \Delta\op} \Fun^\otimes(\Perf(Y), \Perf( S \times (U_m)_{\mathrm B} ) ) . \]
	For every integer $m$, denote by $G_m$ the induced symmetric monoidal functor
	\[ G_m \colon \Perf(Y) \longrightarrow \Perf( S \times (U_m)_{\mathrm B} ) \simeq \prod_{I_U} \Perf(S) . \]
	Let
	\[ \widetilde{G}_m \colon \QCoh(Y) \longrightarrow \cO_{S \times (U_m)_{\mathrm B}} \Mod \simeq \prod_{I_U} \cO_S \Mod \]
	be symmetric monoidal functor obtained by left Kan extension along $\Perf(Y) \hookrightarrow \Ind(\Perf(Y)) \simeq \QCoh(Y)$.
	We claim that each $\widetilde{G}_m$ commutes with perfect complexes, flat objects and connective objects.
	Assuming this claim, \cref{thm:refined_tannakian_image_canal} shows that the composition
	\[ \begin{tikzcd}
		\QCoh(Y) \arrow{r}{G_m} & \cO_{S \times (U_m)_{\mathrm B}} \Mod \arrow{r} & \cO_{S \times (W_m)_{\mathrm B}} \Mod
	\end{tikzcd} \]
	can be represented by a morphism $g_m \colon S \times (W_m)_{\mathrm B} \to Y\an$.
	The full faithfulness provided by \cref{thm:refined_tannakian} shows that the morphisms $g_m$ can be glued back to a morphism
	\[ g \colon S \times X_{\mathrm B} \longrightarrow Y\an . \]
	Finally, the construction shows that $g \circ \eta_{\mathrm{RH}} \simeq f$.
	
	We are therefore left to prove the above claim.
	Reasoning as in \cite[Proposition 5.1]{Porta_Derived_Riemann_Hilbert} we see that pulling back along the canonical morphism $X \to X_{\mathrm{dR}}$ produces a conservative and $t$-exact functor
	\[ \cO_{S \times X_{\mathrm{dR}}} \Mod \longrightarrow \cO_{S \times X} \Mod . \]
	Let
	\[ \widetilde{F} \colon \QCoh(Y) \longrightarrow \cO_{S \times X_{\mathrm{dR}}} \Mod \]
	be the left Kan extension of $F$ along $\Perf(Y) \hookrightarrow \Ind(\Perf(Y)) \simeq \QCoh(Y)$ and define $\widetilde G$ similarly.
	Then
	\[ \widetilde{G} \simeq \eta^*_{\mathrm{RH}} \widetilde(F) . \]
	Invoking \cite[Corollary 5.3]{Porta_Derived_Riemann_Hilbert} it suffices to prove that $\widetilde{F}$ commutes with perfect complexes, flat objects and connective objects.
	
	Observe first that there is a canonical equivalence
	\[ | (U_\bullet)_{\mathrm{dR}} | \simeq X_{\mathrm{dR}} . \]
	Consider next the following \v{C}ech nerve:
	\[ U_{\bullet, \star} \coloneqq \Cech( U_\bullet \to (U_\bullet)_{\mathrm{dR}} ) . \]
	We can identify $U_{\bullet, \star}$ with a bisimplicial object in $\dAnSt_\C$.
	Moreover, \cite[Lemma 4.1]{Porta_Derived_Riemann_Hilbert} provides canonical identifications
	\[ U_{m,n} \simeq \colim_{i \in \mathbb N} \Delta_{U_m}^{n, (i)} , \]
	where $\Delta_{U_m}^n$ denotes the (small) diagonal of $U_m$ in $(U_m)^{\times n}$ and $\Delta_{U_m}^{n,(i)}$ denotes the $i$th infinitesimal neighbourhood of $\Delta_{U_m}^n$ inside $(U_m)^{\times n}$.
	With similar notations, we obtain the following descriptions:
	Let $J \coloneqq \mathbf \Delta\op \times \mathbf \Delta\op \times \mathbb N$.
	Since $\dAnSt_\C$ is an $\infty$-topos, colimits are universal in $\dAnSt_\C$ and in particular we obtain
	\[ \colim_{([m], [n], i) \in J} S \times \Delta_{U_m}^{n,(i)} \simeq S \times X_{\mathrm{dR}} . \]	
	It follows that we can represent the functor $F \colon \Perf(Y) \to \Perf(S \times X_{\mathrm{dR}})$ as an element in the limit
	\[ \lim_{([m], [n], i) \in J} \Fun^\otimes\Big( \Perf(Y), \Perf\big( S \times \Delta_{U_m}^{n,(i)} \big) \Big) . \]
	Let
	\[ F_{m,n}^i \colon \Perf(Y) \longrightarrow \Perf\big( S \times \Delta_{U_m}^{n,(i)} \big) \]
	denote the projection of $F$ on $\Perf\big( S \times \Delta_{U_m}^{n,(i)} \big)$ and let
	\[ \widetilde{F}_{m,n}^i \colon \QCoh(Y) \longrightarrow \cO_{\Delta_{U_m}^{n,(i)}} \Mod \]
	be the left Kan extension of $F_{m,n}^i$ along $\Perf(Y) \hookrightarrow \Ind(\Perf(Y)) \simeq \QCoh(Y)$.
	Notice that each $S \times \Delta_{U_m}^{n,(i)}$ is an derived Stein space.
	Therefore \cref{lem:refined_tannakian_image_necessary_conditions} implies that $\widetilde{F}_{m,n}^i$ preserves perfect complexes, flat objects and connective objects.
	From here, we deduce that the same is true for $\widetilde{F}$. 
	The proof is therefore complete.
\end{proof}

\appendix

\section{Some lemmas on derived Stein spaces}

In this section we prove some basic facts on Stein spaces that do not fit in the main body of the text.
We mainly focus on the \canal setting as the non-archimedean setting has already been addressed in \cite{Porta_Yu_Mapping}.\\

Given a derived \canal space $X \in \dAnc$ we denote by $X_{\mathrm{\'et}}$ its small \'etale site.
This is the $\infty$-site spanned by \'etale maps $Y \to X$ where $Y$ is a derived Stein space.
The truncation functor
\[ \trunc \colon X_{\mathrm{\'et}} \longrightarrow (\trunc(X))_{\mathrm{\'et}} \]
is an equivalence of $\infty$-categories.
This follows directly from \cite[Lemma 3.4]{Porta_DCAGI}.
In virtue of this fact, we give the following definition:

\begin{defin} \label{def:relatively_compact_Stein}
	Let $f \colon U \to V$ be an open immersion of derived Stein spaces.
	We say that \emph{$U$ is relatively compact in $V$ (via $f$)} if the closure of $\trunc(Y)$ inside $\trunc(X)$ is compact.
	In this case, we write $Y \Subset X$.
\end{defin}

\begin{lem} \label{lem:homotopy_group_global_section_Stein}
	Let $U$ be a derived Stein space and let $A_U \coloneqq \Gamma(U;\cO_U\alg)$.
	Then there is a canonical equivalence
	\[ \pi_i(A_U) \simeq \Gamma( U; \pi_i(\cO_U\alg) ) . \]
\end{lem}

\begin{proof}
	First of all, we observe that $\pi_i(\cO_U\alg)$ is by assumption a coherent sheaf on the underived Stein space $\trunc(U)$.
	Therefore, Cartan's theorem B implies that $\Gamma(U; \pi_i(\cO_U\alg))$ is concentrated in cohomological degree zero.
	In turn, this implies that the spectral sequence computing $\Gamma(U; \cO_U\alg)$ degenerates at page $\mathrm E_2$, yielding the desired equivalence.
\end{proof}

\begin{lem} \label{lem:flatness_global_sections_nested_compact_Stein}
	Let $W \Subset V \Subset U$ be a nested sequence of relatively compact derived Stein spaces.
	Set
	\[ A_U \coloneqq \Gamma(U; \cO_U\alg) \quad , \quad A_V \coloneqq \Gamma(V; \cO_V\alg) \quad , \quad A_W \coloneqq \Gamma(W; \cO_W\alg) . \]
	Let $\cF \in \Coh(U)$.
	Then the natural map $A_V \to A_W$ is flat.
\end{lem}

\begin{proof}
	Since $W \Subset V$, we can use \cite[Lemma 8.13]{Porta_Yu_Higher_analytic_stacks_2014} to see that $\pi_0(A_V) \to \pi_0(A_W)$ is flat.
	All we are left to check is therefore that the canonical map
	\[ \pi_i(A_V) \otimes_{\pi_0(A_V)} \pi_0(A_W) \longrightarrow \pi_i(A_W) \]
	is an isomorphism.
	\Cref{lem:homotopy_group_global_section_Stein} provides us with natural equivalences
	\[ \pi_i(A_V) \simeq \Gamma(V; \pi_i(\cO_V\alg)) \quad , \quad \pi_i(A_W) \simeq \Gamma(W; \pi_i(\cO_W\alg)) . \]
	We now observe that $\pi_i(\cO_V\alg)$ is a coherent sheaf on the underived Stein space $\trunc(V)$ and furthermore $\pi_i(\cO_V\alg) |_W = \pi_i(\cO_W\alg)$.
	Therefore, \cite[Proposition 2]{Douady_Proper_1973} implies that
	\[ \Gamma(W; \pi_i(\cO_W\alg)) \simeq \pi_i(\cO_V\alg)(W) \simeq \Gamma(V; \cO_V\alg) {\cotimes}_{\pi_0(A_V)} \pi_0(A_W) . \]
	We now observe that $\pi_i(\cO_V\alg)$ is the restriction to $\trunc(V)$ of the coherent sheaf $\pi_0(\cO_U\alg)$ on $\trunc(U)$.
	Since $V \Subset U$, \cite[Lemmas 8.11 and 8.12]{Porta_Yu_Higher_analytic_stacks_2014} imply that $\Gamma(V; \cO_V\alg)$ is finitely generated over $\pi_0(A_V)$.
	Therefore the canonical map
	\[ \Gamma(V; \cO_V\alg) \otimes_{\pi_0(A_V)} \pi_0(A_W) \longrightarrow \Gamma(V; \pi_i(\cO_V\alg)) {\cotimes}_{\pi_0(A_V)} \pi_0(A_W) \]
	is an equivalence.
	The conclusion follows.
\end{proof}

The same technique used to prove the above lemma allows also to prove the following more general result:

\begin{cor} \label{cor:coherent_sheaves_and_restriction_derived_Stein}
	Let $W \Subset V \Subset U$ be a nested sequence of relatively compact derived Stein spaces.
	Let $A_U$, $A_V$ and $A_W$ be defined as in \cref{lem:flatness_global_sections_nested_compact_Stein}.
	Then for any $\cF \in \Coh(U)$, the natural map
	\[ \gamma_\cF \colon \Gamma(V; \cF|_V) \otimes_{A_V} A_W \longrightarrow \Gamma(W;\cF|_W) \]
	is an equivalence.
\end{cor}

\begin{proof}
	It is enough to check that for every integer $i \in \mathbb Z$ the map $\gamma_\cF$ induces an isomorphism
	\[ \pi_i( \Gamma(V; \cF |_V) \otimes_{A_V} A_W) \longrightarrow \pi_i( \Gamma(W; \cF|_W) ) . \]
	Thanks to \cref{lem:flatness_global_sections_nested_compact_Stein} we know that the map $A_V \to A_W$ is flat.
	As a consequence, the Tor spectral sequence of \cite[7.2.1.19]{Lurie_Higher_algebra} degenerates at the page $E_2$, yielding an equivalence
	\[ \pi_i \left( \Gamma(V; \cF |_V ) \otimes_{A_V} A_W \right) \simeq \pi_i( \Gamma(V; \cF |_V) ) \otimes_{A_V} A_W . \]
	On the other hand, Cartan's theorem B supplies a natural equivalence
	\[ \pi_i ( \Gamma(V;\cF |_V ) ) \simeq \Gamma(V; \pi_i(\cF) |_V ) . \]
	Since $\pi_i(\cF)$ is a coherent sheaf on the underived Stein space $\trunc(U)$, we can use \cite[Proposition 2]{Douady_Proper_1973} to obtain an equivalence
	\[ \Gamma(V; \pi_i(\cF) |_V ) {\cotimes}_{\pi_0(A_V)} \pi_0(A_W) \simeq \Gamma(W; \pi_i(\cF) |_W ) . \]
	We now combine \cite[Lemmas 8.11 and 8.12]{Porta_Yu_Higher_analytic_stacks_2014} to the sheaf $\pi(\cF)$ to deduce that $\Gamma(V; \pi_i(\cF)|_V)$ is finitely generated over $\pi_0(A_V)$.
	In particular, the natural map
	\[ \Gamma(V; \pi_i(\cF) |_V ) \otimes_{\pi_0(A_V)} \pi_0(A_W) \longrightarrow \Gamma(V; \pi_i(\cF) |_V) {\cotimes}_{\pi_0(A_V)} \pi_0(A_W) \]
	is an isomorphism.
	The conclusion follows.
\end{proof}

\bibliographystyle{plain}
\bibliography{dahema}

\def\cprime{$'$}
\begin{thebibliography}{10}

\bibitem{Berkovich_Spectral_1990}
Vladimir~G. Berkovich.
\newblock {\em Spectral theory and analytic geometry over non-{A}rchimedean
  fields}, volume~33 of {\em Mathematical Surveys and Monographs}.
\newblock American Mathematical Society, Providence, RI, 1990.

\bibitem{Bhatt_algebraization_2014}
Bhargav Bhatt.
\newblock {Algebraization and Tannaka duality}.
\newblock {\em arXiv preprint arXiv:1404.7483}, 2014.

\bibitem{Deligne_Equations_differentielles}
Pierre Deligne.
\newblock {\em \'Equations diff\'erentielles \`a points singuliers
  r\'eguliers}.
\newblock Lecture Notes in Mathematics, Vol. 163. Springer-Verlag, Berlin-New
  York, 1970.

\bibitem{Demailly_Complex_analytic}
Jean-Pierre Demailly.
\newblock Complex analytic and differential geometry.
\newblock 2012.

\bibitem{DiNatale_Global_Period_2016}
Carmelo Di~Natale and Julian~VS Holstein.
\newblock The global derived period map.
\newblock {\em arXiv preprint arXiv:1607.05984}, 2016.

\bibitem{Douady_Proper_1973}
Adrien Douady.
\newblock Le th{\'e}or{\`e}me des images directes de {G}rauert [d'apr{\`e}s
  {K}iehl-{V}erdier].
\newblock In A.~Dold and B.~Eckmann, editors, {\em S{\'e}minaire Bourbaki vol.
  1971/72 Expos{\'e}s 400--417}, volume 317 of {\em Lecture Notes in
  Mathematics}, pages 73--87. Springer Berlin Heidelberg, 1973.

\bibitem{Grauert_Theory_Stein_1979}
Hans Grauert and Reinhold Remmert.
\newblock {\em Theory of {S}tein spaces}, volume 236 of {\em Grundlehren der
  Mathematischen Wissenschaften [Fundamental Principles of Mathematical
  Sciences]}.
\newblock Springer-Verlag, Berlin-New York, 1979.
\newblock Translated from the German by Alan Huckleberry.

\bibitem{Grauert_Coherent_1984}
Hans Grauert and Reinhold Remmert.
\newblock {\em Coherent analytic sheaves}, volume 265 of {\em Grundlehren der
  Mathematischen Wissenschaften [Fundamental Principles of Mathematical
  Sciences]}.
\newblock Springer-Verlag, Berlin, 1984.

\bibitem{SGA1}
Alexander Grothendieck.
\newblock {\em Rev\^etements \'etales et groupe fondamental. {F}asc. {II}:
  {E}xpos\'es 6, 8 \`a 11}, volume 1960/61 of {\em S\'eminaire de G\'eom\'etrie
  Alg\'ebrique}.
\newblock Institut des Hautes \'Etudes Scientifiques, Paris, 1963.

\bibitem{Hall_Rydh_Compact_Generation}
Jack Hall and David Rydh.
\newblock Algebraic groups and compact generation of their derived categories
  of representations.
\newblock {\em arXiv preprint arXiv:1405.1890}, 2014.

\bibitem{Hall_Rydh_2017}
Jack Hall and David Rydh.
\newblock Perfect complexes on algebraic stacks.
\newblock {\em Compositio Mathematica}, 153(11):2318--2367, 2017.

\bibitem{Hall_Rydh_2016}
Jack Hall and David Rydh.
\newblock {Mayer-Vietoris squares in algebraic geometry}.
\newblock {\em {arXiv: 1606.08517}}, pages 1--26, June 2016.

\bibitem{Hochschild_Mostow_Holomorphic}
G~Hochschild and GD~Mostow.
\newblock Holomorphic cohomology of complex analytic linear groups.
\newblock {\em Nagoya Mathematical Journal}, 27(2):531--542, 1966.

\bibitem{Lurie_Tannaka_duality}
Jacob Lurie.
\newblock Tannaka duality for geometric stacks.
\newblock {\em arXiv preprint math/0412266}, 2004.

\bibitem{HTT}
Jacob Lurie.
\newblock {\em Higher topos theory}, volume 170 of {\em Annals of Mathematics
  Studies}.
\newblock Princeton University Press, Princeton, NJ, 2009.

\bibitem{DAG-IX}
Jacob Lurie.
\newblock {DAG IX}: Closed immersions.
\newblock Preprint, 2011.

\bibitem{DAG-V}
Jacob Lurie.
\newblock {DAG V}: Structured spaces.
\newblock Preprint, 2011.

\bibitem{DAG-VIII}
Jacob Lurie.
\newblock {DAG VIII}: Quasi-coherent sheaves and {T}annaka duality theorems.
\newblock Preprint, 2011.

\bibitem{DAG-XII}
Jacob Lurie.
\newblock {DAG XII}: Proper morphisms, completions and the {G}rothendieck
  existence theorem.
\newblock Preprint, 2011.

\bibitem{DAG-XIV}
Jacob Lurie.
\newblock {DAG XIV}: Representability theorems.
\newblock Preprint, 2012.

\bibitem{Lurie_Higher_algebra}
Jacob Lurie.
\newblock Higher algebra.
\newblock Preprint, September 2017.

\bibitem{Lurie_SAG}
Jacob Lurie.
\newblock Spectral {A}lgebraic {G}eometry.
\newblock Preprint, 2018.

\bibitem{Mann_Robalo_Brane_actions}
Etienne Mann and Marco Robalo.
\newblock Brane actions, categorifications of {G}romov-{W}itten theory and
  quantum {K}-theory.
\newblock {\em Geom. Topol.}, 22(3):1759--1836, 2018.

\bibitem{Negut_Shuffle}
Andrei Negut.
\newblock Shuffle algebras associated to surfaces.
\newblock {\em arXiv preprint arXiv:1703.02027}, 2017.

\bibitem{Olsson_Mapping_stack}
Martin~C. Olsson.
\newblock {$\underline {\rm Hom}$}-stacks and restriction of scalars.
\newblock {\em Duke Math. J.}, 134(1):139--164, 2006.

\bibitem{PTVV_2013}
Tony Pantev, Bertrand To{\"e}n, Michel Vaqui{\'e}, and Gabriele Vezzosi.
\newblock Shifted symplectic structures.
\newblock {\em Publ. Math. Inst. Hautes \'Etudes Sci.}, 117:271--328, 2013.

\bibitem{Porta_DCAGI}
Mauro Porta.
\newblock Derived complex analytic geometry {I}: {GAGA} theorems.
\newblock {\em arXiv preprint arXiv:1506.09042}, 2015.
\newblock To appear in Journal of Algebraic Geometry.

\bibitem{Porta_Derived_Riemann_Hilbert}
Mauro Porta.
\newblock The derived {R}iemann-{H}ilbert correspondence.
\newblock {\em arXiv preprint arXiv:1703.03907}, 2017.

\bibitem{Porta_Yu_Higher_analytic_stacks_2014}
Mauro Porta and Tony~Yue Yu.
\newblock Higher analytic stacks and {GAGA} theorems.
\newblock {\em arXiv preprint arXiv:1412.5166}, 2014.

\bibitem{Porta_Yu_Representability}
Mauro Porta and Tony~Yue Yu.
\newblock Representability theorem in derived analytic geometry.
\newblock {\em arXiv preprint arXiv:1704.01683}, 2017.
\newblock To appear in Journal of European Mathematical Society.

\bibitem{Porta_Yu_Mapping}
Mauro Porta and Tony~Yue Yu.
\newblock Derived {H}om spaces in rigid analytic geometry.
\newblock {\em arXiv preprint arXiv:1801.07730}, 2018.

\bibitem{Porta_Yu_DNAnG_I}
Mauro Porta and Tony~Yue Yu.
\newblock Derived non-archimedean analytic spaces.
\newblock {\em Selecta Math. (N.S.)}, 24(2):609--665, 2018.

\bibitem{Porta_Yu_NQK}
Mauro Porta and Tony~Yue Yu.
\newblock Non-archimedean quantum $k$-invariants.
\newblock In preparation, 2018.

\bibitem{Schurg_Toen_Vezzosi_Determinant}
Timo Sch\"{u}rg, Bertrand To\"{e}n, and Gabriele Vezzosi.
\newblock Derived algebraic geometry, determinants of perfect complexes, and
  applications to obstruction theories for maps and complexes.
\newblock {\em J. Reine Angew. Math.}, 702:1--40, 2015.

\bibitem{Simpson_Geometricity}
Carlos Simpson.
\newblock Geometricity of the {H}odge filtration on the $\infty$-stack of
  perfect complexes over ${X}_{\mathrm{d{r}}}$.
\newblock {\em arXiv preprint arXiv:0510269v2}, 2005.

\bibitem{Taylor_Several_complex}
Joseph~L. Taylor.
\newblock {\em Several complex variables with connections to algebraic geometry
  and {L}ie groups}, volume~46 of {\em Graduate Studies in Mathematics}.
\newblock American Mathematical Society, Providence, RI, 2002.

\bibitem{Toen_Vaquie_Moduli_of_objects}
Bertrand To\"en and Michel Vaqui\'e.
\newblock Moduli of objects in dg-categories.
\newblock {\em Ann. Sci. \'Ecole Norm. Sup. (4)}, 40(3):387--444, 2007.

\bibitem{Toen_Algebrisation_2008}
Bertrand To{\"e}n and Michel Vaqui{\'e}.
\newblock Alg\'ebrisation des vari\'et\'es analytiques complexes et
  cat\'egories d\'eriv\'ees.
\newblock {\em Math. Ann.}, 342(4):789--831, 2008.

\end{thebibliography}

\end{document}